\documentclass[11pt,reqno]{amsart}

\usepackage{mathrsfs,amsfonts,amssymb,amsmath,amsthm,bookmark}
\usepackage{graphicx,cite}
\usepackage{mathtools}
\usepackage{xcolor}
\usepackage{hyperref}
\usepackage[margin=1in]{geometry}

\newcommand{\R}{\mathbb{R}}

\newcommand{\vv}{\mathbf{v}}
\newcommand{\uu}{\mathbf{u}}

\newcommand{\D}{\mathcal{D}}
\newcommand{\W}{\mathcal{W}}
\newcommand{\X}{\mathcal{X}}
\newcommand{\V}{\mathcal{V}}
\newcommand{\supp}{\mathrm{supp}}

\newcommand{\RR}{\mathbb{R}^n}

\newtheorem{theorem}{Theorem}[section]  
\newtheorem{lemma}[theorem]{Lemma}      
\newtheorem{remark}[theorem]{Remark}  
\numberwithin{equation}{section} 

\newtheorem{definition}{Definition}[section] 

\author[H. Teng]{Hanyu Teng}
\address{School of Mathematical Sciences, Xiamen University, Xiamen, 361005, P.R. China.}
\email{19020260158071@stu.xmu.edu.cn}

\title[Energy equality and renormalization of Navier-Stokes-Vlasov equations]
{Energy equality and renormalization of Navier-Stokes-Vlasov equations}

\date{\today}

\begin{document}
	
\begin{abstract}
We study weak solutions of the incompressible and the compressible Navier--Stokes--Vlasov systems on a bounded
domain, where a viscous fluid is coupled to a dispersed particle phase through a drag force; the fluid obeys the
no-slip boundary condition and the particles are subject to specular reflection. For solutions with Besov
regularity, we show that the Vlasov equation is renormalizable in the sense of DiPerna--Lions, and we give
regularity thresholds for the validity of the global energy equality. If the distribution function $f$ belongs to
$B_p^{\alpha,\infty}$ in space and velocity with time integrability $p$, and the fluid velocity $\mathbf{u}$
belongs to $B_q^{\beta,\infty}$ jointly in space and time with integrability $q$, then the energy equality holds
provided $\alpha+\beta>1$ and $\beta$ exceeds an explicit bound depending on $p$, $q$ and the dimension, under a
natural integrability assumption on the pressure; an analogous criterion holds for the compressible system when
the density is bounded from above and below. The condition $\alpha p>1$ imposed on $f$ is not merely technical:
it is exactly the range in which the macroscopic momentum flux possesses a boundary trace, so that the specular
reflection condition passes to the macroscopic energy balance. When the particle phase is absent, our thresholds
coincide with the criteria $\beta>\frac12+\frac1q$ and $\beta>\frac3q$ obtained by Yu for the incompressible
Navier--Stokes equations alone. The proofs combine a boundary cutoff argument with commutator estimates based on a
scale-localized Besov modulus, together with a boundary-strip analysis.
\end{abstract}
\maketitle

\section*{Conflict of Interest}
The author declares that he has no conflict of interest.
\section{Introduction}

We are concerned with the regularity thresholds for the renormalization property and for the energy equality of weak solutions
to two coupled kinetic--fluid models: the incompressible Navier--Stokes--Vlasov system
\begin{equation}
	\begin{cases}
		\mathbf{u}_t + (\mathbf{u} \cdot \nabla_x)\mathbf{u} + \nabla_x P - \mu \Delta_x \mathbf{u} 
		= -\displaystyle\int_{\mathbb{R}^n} (\mathbf{u} - \mathbf{v}) f \, d\mathbf{v}, \\[2mm]
		\nabla_x \cdot \mathbf{u} = 0, \\[1mm]
		f_t + \mathbf{v} \cdot \nabla_x f + \nabla_\vv \cdot ((\mathbf{u} - \mathbf{v}) f) = 0,
	\end{cases}\label{NSVsystem}
\end{equation}
and the compressible  Navier--Stokes--Vlasov system
\begin{equation}
	\begin{cases}
		\rho_t + \nabla_x \cdot (\rho \uu) = 0, \\
		(\rho \uu)_t + \nabla_x \cdot (\rho \uu \otimes \uu) + \nabla_x P(\rho) - \mu \Delta_x \uu - \lambda \nabla_x \mathrm{div} \uu =
		-\displaystyle\int_{\mathbb{R}^n} \rho(\uu - \vv) f d\vv, \\
		f_t + \vv \cdot \nabla_x f + \nabla_\vv \cdot (\rho(\uu - \vv) f) = 0.
	\end{cases}\label{CNSVsystem}
\end{equation}
\noindent Both systems are posed on $\Omega \times \mathbb{R}^n \times (0,T)$,
where $\Omega \subset \mathbb{R}^n$ ($n = 2,3$) is a bounded domain with smooth boundary.
In these models $\mathbf{u}(t,x)$ is the fluid velocity and $f(t,x,\vv)\ge0$ is the distribution function of the dispersed phase
in phase space; $\mu>0$ is the fluid viscosity, and in the compressible case the pressure is a power law $P(\rho)=\rho^\gamma$
($\gamma>1$, normalized so that $c_0=1$) of the density $\rho\ge0$, while $\lambda>0$ is the second viscosity coefficient.
The right-hand sides of the fluid momentum equations are the drag forces exerted by the particles on the fluid, through which the two
phases exchange momentum.
These systems arise in the modelling of sprays, aerosols and dilute suspensions, with applications ranging from combustion theory to
diesel engines (see, e.g., \cite{BBJ2005,Will-58,Will-85,VAS2006}).

\subsection{Background and main contributions}
The analysis of both systems was initiated by the work on global existence of weak solutions.
For the incompressible system \eqref{NSVsystem}, Boudin et al.~\cite{boudin2009} established global existence by combining
compactness arguments with velocity averaging for the kinetic equation, and Yu~\cite{yu2013} refined the construction so as to
accommodate the physical boundary conditions of no-slip type for the fluid and specular reflection for the particles;
the Fokker--Planck counterpart was treated in \cite{CKL2011}, uniqueness of weak solutions in two dimensions was proved by
Han-Kwan et al.~\cite{HMM2017}, strong solutions with optimal decay rates were constructed in \cite{danchin2026},
and a number of strong solution theories are available.
For the compressible system \eqref{CNSVsystem}, global weak solutions were constructed in one space dimension \cite{LS2021}
and, through compactness and effective viscous flux methods, in higher dimensions \cite{mellet2008global,Li2021VFP},
with large-time behavior investigated in \cite{choi2016}.
In all of these works the solutions satisfy at best an energy \emph{inequality}, and the question of whether equality holds
is left open.

Determining the minimal regularity guaranteeing the energy equality is a classical problem, originating from Onsager's
conjecture for the incompressible Euler equations.
Its sufficiency side was established by Constantin et al.~\cite{CET1994} and brought to the critical H\"older scale by
Isett~\cite{Isett2018}, while the threshold is known to be sharp in view of the convex integration constructions available
for both Euler and Navier--Stokes type equations \cite{Isett2018,Buckmaster2019}.
For the Navier--Stokes equations themselves, the energy equality is classically known under the Prodi--Serrin conditions and
under the scaling-invariant condition $\frac2p+\frac2q\le1$ due to Shinbrot~\cite{Shinbrot1974}.
The presence of a physical boundary makes the problem substantially more delicate, since the energy flux through the boundary
must be controlled by the available regularity:
Bardos and Titi~\cite{bardos2018} settled Onsager's conjecture for the Euler equations in bounded domains, and
Yu~\cite{yu2018} introduced a boundary cutoff argument proving that weak solutions of the Navier--Stokes equations with
no-slip boundary conditions satisfy the energy equality whenever the velocity belongs to $L^q(0,T;B_q^{\beta,\infty}(\Omega))$
with $\beta>\frac12+\frac1q$ or $\beta>\frac3q$.
For compressible models, energy conservation criteria in Besov scales were obtained by Feireisl et al.~\cite{feireisl2017}
for the compressible Euler equations, by Yu~\cite{yu2017} for the compressible Navier--Stokes equations, and by
Akramov et al.~\cite{akramov2020} in the presence of vacuum.

All of the results above concern a single fluid phase.
For coupled kinetic--fluid systems the energy balance involves, in addition to the fluid dissipation, the boundary flux of the
particle kinetic energy, and the interaction between the regularity of the distribution function, the specular reflection
condition and the fluid dissipation is not well understood.
Renormalized solutions, introduced by DiPerna and Lions~\cite{diperna1989ons}, are the natural class for the kinetic part of
such systems;
Onsager-type criteria for the relativistic Vlasov--Maxwell system were established by Bardos et al.~\cite{Bardos2020} in the
renormalized framework, and Wu and Zhang~\cite{WuZhang2023} studied the energy conservation of DiPerna--Lions solutions to
Vlasov--Maxwell systems on bounded domains, revealing the key role of boundary layer conditions and of the regularity of
macroscopic quantities in controlling the boundary fluxes.
For the non-relativistic Navier--Stokes--Vlasov system in the whole space, Ye, Xia and Wang~\cite{ye2026energy} recently
proved an energy equality criterion of Lions--Shinbrot type, based on improved moment estimates for the Vlasov equation
which yield higher integrability of the macroscopic quantities, with only finite energy moments assumed on the distribution
function.
On bounded domains, however, the energy balance is further complicated by the flux of particle kinetic energy through the
boundary, and no regularity criterion taking the boundary into account is available.
The present paper is devoted to closing this gap.

Our main results give the first regularity criteria for the energy equality of weak solutions to both
\eqref{NSVsystem} and \eqref{CNSVsystem} on bounded domains, with the physical boundary conditions
\eqref{Boundary conditions} and \eqref{eq-bdy-comp}.
Roughly speaking (see Theorems \ref{renormalization} and \ref{CNSV thm} for the precise statements), for the incompressible
system the energy equality holds provided
\[
f\in L^p\big(0,T;B_p^{\alpha,\infty}(\Omega\times\mathbb{R}^n)\big),\qquad
\uu\in B_q^{\beta,\infty}([0,T]\times\Omega),
\]
with $\alpha p>1$, $\alpha+\beta>1$, $(p-1)(q-n-2)>2$ and
$\beta>\frac{(n-1)(p-1)+2p}{n(p-1)+2p}+\frac1q$, together with a natural $L^r$-integrability requirement on the pressure;
for the compressible system, under uniform positive bounds on the density and analogous Besov conditions on
$(\uu,f,\rho,\rho\uu)$ with $\alpha p>1$, $\alpha+\tau>1$, $\alpha+\beta+\tau>1$,
$\big(1-\frac1p\big)\big(\frac{rq}{r+q}-n\big)>2$, $\beta>\frac2q+\frac1r$ and $\tau+2\beta>1$, the energy equality holds as well.
In both cases we first prove, under strictly weaker assumptions, that the Vlasov equation enjoys the renormalization property
in the sense of DiPerna--Lions with renormalization functions of polynomial growth.
This is a result of independent interest: the coupling force is only assumed to be of Besov regularity, so the classical
DiPerna--Lions commutator theory does not apply, and the renormalization property requires new scale-localized commutator
estimates (Lemmas \ref{le-error} and \ref{le-error-CNSV}).

A distinctive feature of our criteria is the trace condition $\alpha p>1$ imposed on the distribution function.
It has a clear structural interpretation and is sharp within the present approach:
the energy balance of the particle phase at the boundary involves the normal trace of the macroscopic momentum flux
$M(x)=\int_{\mathbb{R}^n}\vv\big(1+\frac{|\vv|^2}{2}\big)f\,d\vv$, and $\alpha p>1$ is precisely the range in which
$B_p^{\alpha,\infty}$ functions possess well-defined $L^p$ boundary traces (see Lemmas \ref{le-MR}--\ref{le-layer} and
Remark \ref{re-yu}).
Below this threshold the specular reflection condition cannot be transferred to the macroscopic level and the boundary energy
flux cannot be controlled; this is the mechanism by which the coupling with a kinetic phase raises the regularity bar for the
whole system.
When the particle phase is absent, our thresholds reduce exactly to the criteria
$\beta>\frac12+\frac1q$ and $\beta>\frac3q$ of \cite{yu2018} for the single Navier--Stokes equations, so the increase of the
thresholds genuinely reflects the kinetic coupling.
The compressible case carries additional difficulties: the pressure is a nonlinear function of the density, the mass equation
must be used to convert pressure work into an exact time derivative, and uniform positive bounds on the density are required
to keep the momentum equation strictly parabolic and the pressure continuous.

\subsection{Main results}

The main goal of this paper is to seek sufficient conditions on the regularity of weak solutions to both systems.
We impose the following initial and boundary conditions on the incompressible Navier-Stokes system \eqref{NSVsystem}:
\begin{align}
	\mathbf{u}(x,0) = \mathbf{u}_0(x), &\quad f(x,\vv,0) = f_0(x,\vv),\label{Initial conditions}\\
	\mathbf{u} = \mathbf{0} \quad \text{on } \partial\Omega, \quad
	f(x,\vv,t) &= f(x,\vv^*,t) \text{ for } x \in \partial\Omega,~\vv \cdot \nu(x) < 0, \label{Boundary conditions}
\end{align}
where $\vv^* = \vv - 2(\vv \cdot \nu(x)) \nu(x)$ is the specular reflection velocity,
and $\nu(x)$ is the outward unit normal to $\partial\Omega$. We use the following definition of a weak solution.

\begin{definition}\label{def weak1}
	The couple $(\uu,
	f)$ is a global weak solution to \eqref{NSVsystem} with \eqref{Initial conditions}--\eqref{Boundary conditions},
	if the following conditions hold:
	\begin{itemize}
		\item the fluid equations $\eqref{NSVsystem}_1$--$\eqref{NSVsystem}_2$ hold in the distributional sense on
		$(0,T)\times\Omega$, and the Vlasov equation $\eqref{NSVsystem}_3$ holds in the distributional sense on
		$(0,T)\times\Omega\times\mathbb{R}^n$;
		
		\item The initial conditions in \eqref{Initial conditions} hold distributionally on $\Omega$.
	
		\item It satisfies the regularity conditions:
		$$f \geq 0, \qquad f \bigg(1 + \frac{|\vv|^2}{2}\bigg) \in L^\infty(0,T; L^1(\Omega \times \mathbb{R}^n)),\quad \uu \in L^2(0,T; H_0^1(\Omega)).$$

		\item The strong energy inequality holds: for almost every $s\in[0,T]$ including $s=0$, and for all $t\in[s,T]$,
		\begin{align*}
			&\int_\Omega \frac{1}{2}|\uu(t)|^2\,dx
			+ \int_{\Omega}\int_{\RR} f(t)\!\left(1+\frac{|\vv|^2}{2}\right)d\vv\,dx
			+ \int_s^t\int_\Omega \mu|\nabla\uu|^2\,dxd\tau \\
			&\quad+ \int_s^t\int_{\Omega\times\RR} f|\uu-\vv|^2\,d\vv\,dxd\tau
			\le \int_\Omega \frac{1}{2}|\uu(s)|^2\,dx
			+ \int_{\Omega}\int_{\RR} f(s)\!\left(1+\frac{|\vv|^2}{2}\right)d\vv\,dx.
		\end{align*}
		Weak solutions constructed by approximation schemes,
		such as those of Yu \cite{yu2013} under the specular reflection boundary condition, satisfy this inequality.
	\end{itemize}
\end{definition}
\noindent Besides, we say that $(\uu,
f)$ satisfies the renormalization property (in the sense of DiPerna--Lions \cite{diperna1989ons},
cf. also \cite{lions1996}) with respect to $\mathcal{R}\in\mathcal{S}_k(\R^+;\R^+)$ for some $k\in[2,\infty]$ if
\begin{equation}\label{eq-renorm-1}
	\partial_t (\mathcal{R}(f))+\nabla_x\cdot(\vv\mathcal{R}(f))+\nabla_\vv\cdot((\uu-\vv)\mathcal{R}(f))+n(\mathcal{R}(f)-f\mathcal{R}'(f))=0
\end{equation}
holds in the sense of distributions, 
where 
\begin{equation*}
	\mathcal{S}_k(\R^+;\R^+) = \left\{ g:\R^+\to\R^+ \bigg|
	\begin{aligned}
		& g \text{ is non-decreasing}, \quad g\in \mathcal{C}^1(\R^+;\R^+), \quad g(0)=0, \\
		& \exists \text{ constant } C>0, \text{ s.t. } g'(\tau)\leq C(1+\tau)^{\frac{k}{2}-1}
	\end{aligned}
	\right\}.
\end{equation*}

Before stating our first main result, we first introduce the definition of the Besov spaces.
For $0<\alpha<1$ and $1\leq p<\infty$, $B_p^{\alpha,\infty}(\Sigma)$ denotes the Besov space on the domain $\Sigma$, i.e.,
$f \in B_p^{\alpha,\infty}(\Sigma)$ if and only if
\begin{equation*}
	\displaystyle\|f\|_{B_p^{\alpha,\infty}(\Sigma)}=\left(\int_{\Sigma}|f(x)|^pdx\right)^\frac{1}{p}
	+\sup_{0<|z|\leq 1}\ |z|^{-\alpha}\left(\int_{\Sigma_z}|f(x+z)-f(x)|^p dx\right)^\frac{1}{p}<+\infty,
\end{equation*}
where $\Sigma_z=\{x\in\Sigma:\,x+z\in\Sigma\}$; for $p=\infty$, the $L^p$ norms above are replaced by the corresponding essential suprema.

\begin{theorem}\label{renormalization}
	Let $(\uu,
	f)$ be a weak solution to the incompressible Navier--Stokes--Vlasov system \eqref{NSVsystem} with the initial--boundary conditions
	\eqref{Initial conditions}--\eqref{Boundary conditions} in the sense of Definition \ref{def weak1}.
	Assume moreover that the weak solution satisfies the following regularity conditions
	\begin{align*}
		&f \in L^p(0,T;B_p^{\alpha,\infty}(\Omega\times\RR)),\quad \uu \in B_q^{\beta,\infty}([0,T]\times\Omega),
	\end{align*}
	where $\alpha,
	\beta \in (0,1)$ and $1 < p,q < \infty$ (or $p=q=\infty$) satisfy the trace condition $\alpha p>1$ and $\alpha+\beta>1$. 
	Then for any smooth function $\mathcal{R} \in \mathcal{S}_k(\R^+;\R^+)$ with $k \in \big[2,2p\big(1-\frac1q\big)\big]$
	(this interval is nonempty provided $\frac1p+\frac1q<1$, which in particular follows from the assumptions for the energy
	equality below), the weak solution $(\uu, f)$ satisfies the renormalization property \eqref{eq-renorm-1}.
	
	Furthermore, if $ q > n+2$  and $p > 1$,
	\begin{align*}
		&(p-1)(q-n-2) > 2,\\
		&\frac{(n-1)(p-1)+2p}{n(p-1)+2p} + \frac{1}{q} < \beta < 1,
	\end{align*}
	and the pressure satisfies $P\in L^1(0,T;L^r(\Omega))$ with $\frac1r=\frac1q+\frac1{p_1}$,
	where $p_1$ is the exponent in Lemma \ref{le-f-estimate},
	then the following energy equality holds for any $t \in [0,T]$:
	\begin{align}
		&\int_\Omega \frac{1}{2}|\uu(t)|^2\,dx
		+ \int_{\Omega}\int_{\RR} f(t)\!\left(1+\frac{|\vv|^2}{2}\right)d\vv\,dx
		+ \int_0^t\int_\Omega \mu|\nabla\uu|^2\,dxd\tau \notag \\
		&\quad
		+ \int_0^t\int_{\Omega\times\RR} f|\uu-\vv|^2\,d\vv\,dxd\tau
		= \int_\Omega \frac{1}{2}|\uu_0|^2\,dx
		+ \int_{\Omega}\int_{\RR} f_0\!\left(1+\frac{|\vv|^2}{2}\right)d\vv\,dx.
	\end{align}
\end{theorem}

\begin{remark}\label{re-yu}
	The trace condition $\alpha p>1$ in Theorem \ref{renormalization} is imposed precisely to guarantee that the truncated macroscopic
	momentum flux $M_R$ admits a well-defined trace on $\partial\Omega$ (see Lemmas \ref{le-MR}--\ref{le-layer}),
	so that the specular reflection condition holds rigorously in the limiting process of the energy equality.
	This condition is essentially sharp within the present approach:
	$\alpha>1/p$ is exactly the threshold under which the trace operator $\operatorname{tr}$ on the Besov space $B_p^{\alpha,\infty}(\Omega)$
	takes values in $L^p(\partial\Omega)$ (Lemma \ref{le-trace});
	when $\alpha p\leq 1$, the space $B_p^{\alpha,\infty}(\Omega)$ contains functions admitting no well-defined restriction to the boundary,
	so that the specular reflection condition cannot be transferred to the macroscopic momentum flux $M_R$ in the trace sense.
	Whether the energy equality on bounded domains can be justified without invoking boundary traces---for instance,
	by interpreting the boundary flux in a distributional sense---remains an open problem.
\end{remark}

\begin{remark}\label{re-aniso}
	The two regularity assumptions in Theorem \ref{renormalization} are deliberately stated in different forms.
	For $f$, only the space--velocity Besov regularity enters the commutator estimates, while the time
	integrability $L^p$ is dictated by the H\"older pairing in \eqref{eq-comm-2}, and the time continuity of $f$
	needed at the endpoints is recovered from the Vlasov equation itself (see (ii) in Section \ref{se-pro-th1});
	hence the anisotropic space $L^p(0,T;B_p^{\alpha,\infty}(\Omega\times\RR))$ suffices.
	For $\uu$, the joint space--time regularity $B_q^{\beta,\infty}([0,T]\times\Omega)$ is genuinely needed:
	the Morrey-type embedding yields $\uu\in C([0,T];L^q(\Omega))$ (see (i) in Section \ref{se-pro-th1}), which the
	endpoint argument relies on, and writing $\uu\in L^q(0,T;B_q^{\beta,\infty}(\Omega))$ instead would not suffice.
	The estimates below accordingly use $\|f\|_{L^p(0,T;B_p^{\alpha,\infty}(\Omega\times\RR))}$ (or the
	equivalent $\W_f$-control of Lemma \ref{le-Wf}) and $\|\uu\|_{B_q^{\beta,\infty}([0,T]\times\Omega)}$.
\end{remark}

\smallskip
For the compressible system \eqref{CNSVsystem}, we prescribe the following initial and boundary conditions:
\begin{align}
		&\rho(x,0) = \rho_0(x), \quad \rho \uu(x,0) = \mathbf{m}_0(x) = \rho_0 \uu_0, \quad f(x,\vv,0) = f_0(x,\vv); \label{eq-ini-comp}\\
		&\uu = \mathbf{0} \quad \mathrm{on} \, \partial\Omega, \quad f(x,\vv,t) = f(x,\vv^*,t) \quad x \in \partial\Omega,
		\quad \vv \cdot \nu(x) < 0,\label{eq-bdy-comp}
	\end{align}
where $\vv^* = \vv - 2(\vv \cdot \nu(x)) \nu(x)$ is the specular reflection velocity,
and $\nu(x)$ is the outward unit normal to $\partial\Omega$. We also define the corresponding weak solutions.
\begin{definition}\label{def weak}
	The triple $(\rho, \uu, f)$ is a global weak solution to \eqref{CNSVsystem} with \eqref{eq-ini-comp}--\eqref{eq-bdy-comp},
	if the following conditions hold:
	\begin{itemize}
		\item the fluid equations $\eqref{CNSVsystem}_1$--$\eqref{CNSVsystem}_2$ hold in the distributional sense on
		$(0,T)\times\Omega$, and the Vlasov equation $\eqref{CNSVsystem}_3$ holds in the distributional sense on
		$(0,T)\times\Omega\times\mathbb{R}^n$;
		\item The initial conditions in \eqref{eq-ini-comp} hold distributionally on $\Omega$.
		\item It satisfies the regularity conditions:
		\begin{align*}
			&\rho \geq 0, \quad \rho \in L^\infty(0,T; L^\gamma(\Omega)), \quad \uu \in L^2(0,T; H_0^1(\Omega)), \\
			&\rho |\uu|^2 \in L^\infty(0,T; L^1(\Omega)), \quad 0\leq f \bigg(1 + \frac{|\vv|^2}{2}\bigg) \in L^\infty(0,T;
			L^1(\Omega \times \mathbb{R}^n)).
		\end{align*}
		\item The strong energy inequality holds: for almost every $s\in[0,T]$ including $s=0$, and for all $t\in[s,T]$,
		\begin{align*}
			&\int_\Omega\Big(\frac12\rho|\uu|^2+\frac{\rho^\gamma}{\gamma-1}+\int_{\RR}f\Big(1+\frac{|\vv|^2}{2}\Big)d\vv\Big)(t)\,dx\notag\\
			&\quad+\int_s^t\!\!\int_\Omega\Big(\mu|\nabla\uu|^2+\lambda|\mathrm{div}\,\uu|^2+\int_{\RR}\rho f|\uu-\vv|^2\,d\vv\Big)dxd\tau\\
			&\quad\le
			\int_\Omega\bigg(\frac12\rho|\uu|^2+\frac{\rho^\gamma}{\gamma-1}\notag
			+\int_{\RR}f\Big(1+\frac{|\vv|^2}{2}\Big)d\vv\bigg)(s)\,dx .
		\end{align*}
		As for the single compressible Navier--Stokes equations,
		weak solutions arising from the standard approximation schemes satisfy this inequality;
		we include it here as part of the admissibility requirement on weak solutions.
	\end{itemize}
\noindent Similarly, we say that the weak solution $(\rho, \uu,
f)$ satisfies the renormalization property with respect to $\mathcal{R}\in\mathcal{S}_k(\R^+;\R^+)$ for some $k\in[2,\infty]$
if
\begin{equation}\label{eq-renorm-2}
	\partial_t (\mathcal{R}(f))+\nabla_x\cdot(\vv\mathcal{R}(f))+\nabla_\vv\cdot(\rho(\uu-\vv)\mathcal{R}(f))
	+n(\rho\mathcal{R}(f)-\rho f\mathcal{R}'(f))=0
\end{equation}
holds in the sense of distributions. 
\end{definition}
\begin{theorem}\label{CNSV thm}
	Let $(\rho, \uu,
	f)$ be a weak solution of \eqref{CNSVsystem} with the initial-boundary conditions \eqref{eq-ini-comp}--\eqref{eq-bdy-comp} in the sense of
	Definition \ref{def weak}. Assume moreover that the weak solution satisfies 
	\begin{align*}
		&	0 < \underline{\rho} \leq \rho(t,x) \leq \overline{\rho} < \infty, \quad f \in B_p^{\alpha,\infty}([0,T]\times\Omega\times\RR), \\
		&\uu \in B_q^{\beta,\infty}([0,T]\times\Omega),\quad 	\rho, \,\rho \uu \in B_r^{\tau,\infty}([0,T]\times\Omega),
	\end{align*}
	where $\alpha, \beta,\tau \in (0,1)$ and $1 < p,q,r < \infty$ (or $p=q=r=\infty$) satisfy the trace condition $\alpha p>1$,
	$\alpha+\tau>1$ and $\alpha+\beta+\tau>1$. 
	Then for any smooth function $\mathcal{R} \in \mathcal{S}_k(\R^+;\R^+)$ with $k \in \big[2,2p\big(1-\frac1q-\frac1r\big)\big]$
	(this interval is nonempty provided $\frac1p+\frac1q+\frac1r<1$; when $p>2$, this in turn follows from the
	assumptions for the energy equality below, since they imply $\frac1q+\frac1r<\frac1{n+2}\le\frac14$),
	the weak solution $(\rho, \uu, f)$ satisfies the renormalization property: \eqref{eq-renorm-2}.
	
Furthermore, assume $\tau+2\beta-1>0$. If
	\[
	p > 1, \qquad q > 2, \qquad
	\Big(1-\frac{1}{p}\Big)\Big(\frac{rq}{r+q}-n\Big)>2,
	\qquad
	\beta > \frac{1}{r} + \frac{2}{q},
	\]
	then the following energy equality holds for all $t \in [0,T]$:
	\begin{align*}
		&\int_\Omega \left(
		\frac{1}{2} \rho(t) |\uu(t)|^2
		+ \frac{\rho(t)^\gamma}{\gamma-1}
		+ \int_{\mathbb{R}^n} f(t)\left(1+\frac{|\vv|^2}{2}\right)d\vv
		\right) dx \\
		&\quad
		+ \int_0^t \int_\Omega \left(
		\mu |\nabla \uu|^2
		+ \lambda (\nabla\cdot\uu)^2
		+ \int_{\mathbb{R}^n} \rho f |\uu-\vv|^2 d\vv
		\right) dxd\tau \\
		&=
		\int_\Omega \left(
		\frac{1}{2} \rho_0 |\uu_0|^2
		+ \frac{\rho_0^\gamma}{\gamma-1}
		+ \int_{\mathbb{R}^n} f_0\left(1+\frac{|\vv|^2}{2}\right)d\vv
		\right) dx .
	\end{align*}
\end{theorem}

\begin{remark}\label{re-CNSV}
	We point out the main differences between Theorems \ref{renormalization} and \ref{CNSV thm}.
	First, the assumptions on $f$ are different: in the incompressible case $f$ enters the energy equality
	only through the macroscopic moments $\bar\rho$ and $j$, and the $L^p$ integrability in time suffices since
	the time regularity is carried by $\uu$; in the compressible case the triple-product commutators involving
	$\rho f$ require the joint space--time Besov control of $f$.
	Second, Theorem \ref{CNSV thm} requires the density to admit positive upper and lower bounds,
	which ensures the continuity of the pressure and the strict parabolicity of the momentum equation.
	If $f\equiv0$, the intermediate conditions in the proof of Theorem \ref{renormalization} formally reduce to
	$\beta>\frac12+\frac1q$ and $\beta>\frac3q$ (with $r=q/2$ in the pressure estimate), which are precisely
	the energy equality criteria of Yu \cite{yu2018} for the single incompressible Navier--Stokes equations
	in bounded domains, rather than the H\"older-exponent formulation $3\beta>1$; the coupling with the particle
	phase is therefore responsible for the stronger thresholds above.
\end{remark}

\subsection{Difficulties and strategy}
The main difficulties in the proof of Theorem \ref{renormalization} come from three aspects:
first,
the delicate cutoff function $\theta_h$ near the boundary requires sufficient regularity of the solution near $\partial\Omega$ to control
the boundary layer integrals;
second,
the momentum exchange (drag force) between the fluid and the particles is coupled through the macroscopic quantities $\bar\rho$ and $j$,
which are incorporated into the energy framework via interpolation estimates;
third, in the incompressible case the pressure is implicitly defined by an elliptic equation,
whose $L^r$ integrability imposes additional constraints on the regularity indices.
Our strategy is to adapt the boundary cutoff method of Bardos--Titi~\cite{bardos2018} and Yu~\cite{yu2018}, together with
spatio-temporal mollification, to the coupled system. Three ingredients are new compared with the pure fluid situation.
First, the commutator estimates are performed through the scale-localized moduli $\X_f$, $\V_f$ and their time integral
$\W_f$ (Lemma \ref{le-Wf}), which allow one to use the Besov regularity of $f$ without any quantitative convergence rate;
the key analytical tool is the two-route estimate of Lemma \ref{le-diff}, which we show to be sharp (Remark \ref{re-diff-sharp})
and which yields the thresholds $\alpha+\beta>1$ and $\alpha+\beta+\tau>1$.
Second, the boundary-layer integrals are treated by boundary-strip estimates of the form
$\|\uu\|_{L^q(\Omega_{3h})}\leq Ch^\beta\|\uu\|_{B_q^{\beta,\infty}([0,T]\times\Omega)}$ (valid since $\uu=0$ on $\partial\Omega$),
which absorb the boundary terms without consuming the scaling budget between the two mollification parameters $\epsilon$ and
$h=\epsilon^\nu$.
Third, the trace condition $\alpha p>1$ is imposed so that the boundary flux terms are under control in the Besov space
framework (Lemmas \ref{le-MR}--\ref{le-layer}).

The lower bound $\beta>\frac{(n-1)(p-1)+2p}{n(p-1)+2p}+\frac1q$ in Theorem \ref{renormalization} is designed to absorb all intermediate
conditions appearing in the proof. Indeed, with $p_1=\frac{n(p-1)+2p}{n(p-1)+2}$ from Lemma \ref{le-f-estimate} one has
\[
\frac{(n-1)(p-1)+2p}{n(p-1)+2p}-\frac1{p_1}=\frac{p-1}{n(p-1)+2p}>0,
\]
and the condition $(p-1)(q-n-2)>2$ is equivalent to $\frac1q<\frac{p-1}{n(p-1)+2p}$. Consequently,
\[
\beta>\frac{1}{p_1}+\frac{2}{q},\qquad \beta>\frac12+\frac1q,\qquad \beta>\frac3q,
\]
which are precisely the conditions used in the estimates of $I_2$, $I_3$,
$I_4$ and $E_{11}$ (see \eqref{eq-qbeta-1}--\eqref{eq-qbeta-5} and the pressure condition $\frac1r<\beta-\frac1q$ with
$\frac1r=\frac1q+\frac1{p_1}$); the same lower bound also implies the compatibility conditions
$\beta>\frac{q+2}{3q-4}$ and $\beta>\frac{q^2-3}{2q(q-2)}$, which guarantee that the $\nu$-windows
\eqref{eq-qbeta-2} and \eqref{eq-qbeta-3} are nonempty; here we also used $\frac{n-1}{n}\ge\frac12$ and $q>n+2\ge\frac{2n}{n-1}$ for $n\ge2$.
The remaining compatibilities among the $\nu$-windows \eqref{eq-qbeta-1}--\eqref{eq-qbeta-5} arising in the
estimates of $I_2$--$I_4$ reduce to $\beta>\frac12+\frac1q$ or to $\beta>\frac3q$ (or to one of the two
displayed conditions), and are therefore consequences of the same lower bound.
We point out the origin of the hypothesis $\alpha+\beta>1$ in Theorem \ref{renormalization} (and of
$\alpha+\tau>1$, $\alpha+\beta+\tau>1$ in Theorem \ref{CNSV thm}): it comes from the sharp two-route
estimate for the mollification commutators in Lemma \ref{le-diff} (see Remark \ref{re-diff-sharp}), which
yields the rate $\epsilon^{\alpha+\beta-1}$ for the fluid-velocity commutator under the choice
$\delta=\epsilon$. If one instead decouples the two mollification scales by taking $\delta=\epsilon^s$
with $\frac{1}{\alpha+1}<s<\alpha+\beta$, this hypothesis can be relaxed to $(\alpha+\beta)(\alpha+1)>1$;
we state the theorems with $\delta=\epsilon$ for simplicity of the presentation.
A similar bookkeeping applies to Theorem \ref{CNSV thm}.
The boundary term $B_1$ in Section \ref{se-pro-th2} is estimated directly by the boundary-strip argument
(using $|\nabla\theta_h|\le Ch^{-1}$ and $\|\uu\|_{L^q(\Omega_{3h})}\le Ch^\beta\|\uu\|_{B_q^{\beta,\infty}([0,T]\times\Omega)}$,
the latter since $\uu=0$ on $\partial\Omega$), so it consumes no $\nu$-window.
The remaining terms of Step 2 in Section \ref{se-pro-th2} require
\[
\frac{qr(1-\beta)}{qr-2r-q}<\nu<\frac{r(q\beta-1)}{q+r},\qquad
\nu<\frac{\tau+\beta}{\frac2q+\frac1r},\qquad
\nu<\frac{\tau+\beta}{\frac12+\frac1q-\beta},
\]
and the hypothesis $\big(1-\frac1p\big)\big(\frac{rq}{r+q}-n\big)>2$ yields $\frac2q+\frac1r<\frac{2}{n+2}\le\frac12$,
so that, together with $\tau+2\beta-1>0$, all these conditions are compatible as soon as $\beta>\frac2q+\frac1r$;
in particular the condition $(r-1)(q-2)>2$ used for the term $B_3$ in Section \ref{se-pro-th2} is only one possible sufficient route
and is not needed once
$\beta>\frac1r+\frac2q$ is imposed.

The analysis of the compressible case faces new difficulties on top of the incompressible one.
The coupling between the density $\rho$ and the pressure $P(\rho)$ makes the nonlinear structure of the momentum equation more involved,
while possible vacuum or singularity of the density would obstruct the standard energy estimates;
for this reason, Theorem \ref{CNSV thm} requires the density to admit positive upper and lower bounds.
Compared with the energy conservation criteria for the compressible Euler equations established by Feireisl et al.~\cite{feireisl2017} and
Akramov et al.~\cite{akramov2020},
our result takes the coupling effect of the particle phase into account and incorporates the mass conservation relation into the energy
estimate
through specially constructed test functions, thereby overcoming the additional errors caused by the density--velocity coupling.

\subsection{Organization and notations}
The rest of the paper is arranged as follows.
In Section \ref{se-pre}, we provide some lemmas to be used in the proofs of Theorems \ref{renormalization} and \ref{CNSV thm},
which will be given in Sections \ref{se-pro-th1} and \ref{se-pro-th2} respectively.

For $1 \leq p \leq \infty$, $\|\cdot\|_p$ denotes the norm of the Lebesgue space $L^p(\mathbb{R}^n)$.
The space $\mathcal{C}^k(\mathbb{R}^+;
\mathbb{R}^+)$ denotes the set of all $k$ order continuously differentiable functions from $\R^+$ to $\R^+$.
$\mathcal{D}(\mathbb{R}^n)$ is the space of smooth test functions with compact support,
and $\D'(\RR)$ denotes the space of distributions. We use
$C$
to denote a generic constant, which is allowed to take different values in different contexts (or on different lines). 
We write $a\wedge b:=\min\{a,b\}$ for $a,b\in\R$.

\section{Preliminaries}\label{se-pre}

To prove the energy equality in bounded domain,
we borrow some ideas on treating the boundary effect from the work of Bardos-Titi\cite{bardos2018}:
for any $x\in\bar{\Omega}$, let $d(x)=\inf_{y\in\partial\Omega}|x-y|$, and denote the open set $\Omega_h=\{x\in\Omega| d(x)<h \}$.
Since our boundary $\partial\Omega
$ is a smooth compact manifold, there exists $h_0(\Omega)>0$ with the following properties:
\begin{enumerate}
	\item For any $x\in\overline{\Omega_{h_0}}$, the function $x\to d(x)$ belongs to $C^1(\overline{\Omega_{h_0}})$;
	\item for any $x\in\overline{\Omega_{h_0}}$ there exists a unique point $\sigma(x)\in\partial\Omega$ such that\begin{equation*}
		d(x)=|x-\sigma(x)|\quad \nabla d(x)=-\nu(\sigma(x)).
	\end{equation*}
\end{enumerate}
For $h\leq h_0(\Omega)$,
we have $d\in C^2(\overline{\Omega_{h_0}})$.  We also introduce a $C^\infty(\R)$ nondecreasing function $\kappa:\R\to[0,1]$,
satisfying $\kappa(\xi)=1 $ when $\xi\geq2$ and vanishing identically when $\xi\leq 1$.
In addition,
\begin{equation*}
	\kappa'(\xi)\geq0,\quad |\kappa'(\xi_1)-\kappa'(\xi_2)|\leq L|\xi_1-\xi_2| \text{ for some constant }L>0.
\end{equation*}
Next, we define
\begin{equation*}
	\theta_h(x)=\kappa\bigg(\frac{d(x)}{h}\bigg).
\end{equation*}
We will also denote by $\theta_h$ its extension by zero outside $\Omega$ and introduce a lemma from \cite{bardos2018} about it.
\begin{lemma}[\cite{bardos2018}]\label{le-theta}
		Let $h\in(0,h_0)$.
		For any vector field $w\in C^{0,\alpha}(\overline{\Omega})$, with $w\cdot\nu=0$ on $\partial\Omega$, one has the following estimates
	\begin{align*}
		|w(x)\cdot\nabla\theta_h(x) |&\leq C\|w\|_{C^{0,\alpha}(\Omega)} h^{\alpha-1},\quad \forall x\in\RR,\\
		\int_{\RR}|w(x)\cdot\nabla\theta_h(x) |dx&\leq C\|w\|_{C^{0,\alpha}(\Omega)}h^\alpha.
	\end{align*}
\end{lemma}
Moreover, we need to mollify our given weak solutions to be test functions.
For the space and velocity variables (both in $\R^n$), we introduce a smooth, non-negative, radially symmetric profile
$m\in\D(\R^n)$ with $\int_{\R^n}m(x)\,dx=1$, supported in the annulus $\{1<|x|<2\}$, and set
\begin{eqnarray*}
	m_\delta(x)=\frac{1}{\delta^n}m\Big(\frac{x}{\delta}\Big),\quad \delta>0.
\end{eqnarray*}
For the time variable, we introduce $0\leq \tilde{m}\in\D(\R)$ with $\int_\R \tilde{m}(t)\,dt=1$ and
$\supp \tilde{m}\subset(1,2)$, and set $m_\eta(t)=\frac{1}{\eta}\tilde{m}(\frac{t}{\eta})$;
when a time mollification is moved from one side of a duality pairing to the other,
the reflected kernel $\breve{m}_\eta(t)=m_\eta(-t)$, supported in $(-2\eta,-\eta)$, is understood.

We denote $*$ as the standard convolution product.
So we define for any $g\in\D'(\Omega)$, $h\in\D'(\Omega\times\R^+)$ and $f\in\D'(\Omega\times\RR\times\R^+)$ by
\begin{align*}
	g^\epsilon(x)&=m_\epsilon(x)*_x g(x),\\
	h^{\epsilon,\eta}(x,t)&=m_\epsilon(x)*_x m_\eta(t)*_t h(x,t),\\
	f^{\epsilon,\delta,\eta}(x,\vv,t)&=m_\epsilon(x)*_xm_\delta(\vv)*_\vv m_\eta(t)*_t f(x,\vv,t).
\end{align*}

We list some useful lemmas about the estimates of mollification from \cite{bardos2018,feireisl2017}.
\begin{lemma}
	For any function $f\in L^1\cap L^\infty\cap B_p^{\alpha,\infty}(\RR)$ with $\alpha\in(0,1)$ and $p\in [1,\infty]$,
	there exists a constant $C>0$ s.t. for any $\epsilon>0$
	\begin{align*}
		\|f^\epsilon\|_{L^p(\RR)}&\leq\|f\|_{L^p(\RR)},\\
		\|f(\cdot+z)-f(\cdot)\|_{L^p(\RR)}&\leq C|z|^\alpha\|f\|_{B^{\alpha,\infty}_p(\RR)},\\
		\|f^\epsilon-f\|_{L^p(\RR)}&\leq C\epsilon^\alpha \|f\|_{B^{\alpha,\infty}_p(\RR)},\\
		\|\nabla f^\epsilon\|_{L^p(\RR)}&\leq C\epsilon^{\alpha-1} \|f\|_{B^{\alpha,\infty}_p(\RR)}
	\end{align*}
\end{lemma}

\begin{lemma}\label{le-communicate}
	Let $f\in B_p^{\alpha,\infty}(\RR),g\in L^q(\RR)$ with $0<\alpha<1$, $1\leq p,q\leq\infty$, and $\frac{1}{p}+\frac{1}{q}\leq1$.
	Then, we have
	\begin{align*}
		\|(fg)^\epsilon-fg^\epsilon\|_{L^r(\RR)}&\leq C\epsilon^{\alpha}\|f\|_{B_p^{\alpha,\infty}(\RR)}\|g\|_{L^q(\RR)},
	\end{align*}
	where $\frac{1}{p}+\frac{1}{q}=\frac{1}{r}$. More generally, for vector-valued $f,g$ and any test function
	$\varphi\in C_c^\infty(\RR)$, integration by parts together with the first estimate gives the paired form
	\begin{align*}
		\Big|\int_{\RR}\big(\mathrm{div}(fg)^\epsilon-\mathrm{div}(fg^\epsilon)\big)\varphi\,dx\Big|\notag
		&=\Big|\int_{\RR}\big((fg)^\epsilon-fg^\epsilon\big)\cdot\nabla\varphi\,dx\Big|\\
		&\leq C\epsilon^{\alpha}\|f\|_{B_p^{\alpha,\infty}(\RR)}\|g\|_{L^q(\RR)}
		\|\nabla\varphi\|_{L^{r'}(\RR)},
	\end{align*}
	where $\frac{1}{r}+\frac{1}{r'}=1$.
\end{lemma}
\begin{remark}\label{re-communicate}
	The divergence commutator above cannot in general be estimated as a strong $L^r$ function with rate
	$\epsilon^{\alpha-1}$: already for a constant vector $g$, $\mathrm{div}(fg^\epsilon)$ involves $\nabla f$,
	which is only a distribution under the mere assumption $f\in B_p^{\alpha,\infty}$. In the sequel we therefore
	always use the divergence commutator in its paired form, or estimate the terms directly by the H\"older
	inequality without integration by parts.
\end{remark}

For any $f\in B_p^{\alpha,\infty}(\Omega\times\RR)$ with $\alpha\in(0,1)$, $p\in[1,\infty]$,$\epsilon>0$ and $\delta>0$, we define
\begin{align*}
	\displaystyle \X_f(\epsilon)&:=\sup_{x\neq y,\,x,y\in\Omega}|x-y|^{-\alpha}\left(\int_\Omega\int_{\RR}\mathbf{1}_{\epsilon\leq |x-y|\leq
	2\epsilon}|f(x,\vv)-f(y,\vv) |^p d\vv dx\right)^\frac{1}{p},\\
	\V_f(\delta)&:=\sup_{\vv\neq\mathbf{w}}|\vv-\mathbf{w}|^{-\alpha}\left(\int_\Omega\int_{\RR}\mathbf{1}_{\delta\leq
	|\vv-\mathbf{w}|\leq 2\delta}|f(x,\vv)-f(x,\mathbf{w}) |^p d\vv dx\right)^\frac{1}{p},\\
	\W_f(\epsilon,\delta)&:=\int_{0}^{T} (\X_{f(t)}(\epsilon)+\V_{f(t)}(\delta))dt.
\end{align*}
Here the indicator in $\X_f$ restricts the increment under consideration to scales
$\epsilon\leq |x-y|\leq 2\epsilon$, while the $L^p$ integral is taken over the full domain
$\Omega\times\RR$ (and similarly for $\V_f$ with the velocity increment at scale $\delta$).
Then the following properties can be deduced.

\begin{lemma}\label{le-diff}
	Assume that $f\in B_p^{\alpha,\infty}(\Omega\times\RR)$ with $\alpha\in(0,1)$, $p\in[1,\infty]$,
	then there exists a constant $C>0$ s.t. for any $\epsilon,\delta>0$ and $\mathbf{z}\in\RR$
	\begin{align*}
		\|\nabla_x f^\epsilon(x,\vv-\mathbf{z})-\nabla_x f^\epsilon(x,\vv)\|_{L^p(\Omega\times\RR)}&\leq C
		\big(\epsilon^{\alpha-1}\X_f(\epsilon)\big)\wedge\big(\epsilon^{-1}|\mathbf{z}|^\alpha \V_f(|\mathbf{z}|)\big),\\
		\|(\nabla_x f^\epsilon)^\delta-\nabla_x f^\epsilon\|_{L^p(\Omega\times\RR)}&\leq C
		\big(\epsilon^{\alpha-1}\X_f(\epsilon)\big)\wedge\big(\epsilon^{-1}\delta^\alpha \V_f(\delta)\big).
	\end{align*}
\end{lemma}
\begin{proof}
	For the first inequality, on the one hand the triangle inequality together with
	$\|\nabla_x f^\epsilon\|_{L^p}\leq C\epsilon^{\alpha-1}\X_f(\epsilon)$ gives the bound
	$C\epsilon^{\alpha-1}\X_f(\epsilon)$; on the other hand, writing
	\[
	\nabla_x f^\epsilon(x,\vv-\mathbf{z})-\nabla_x f^\epsilon(x,\vv)
	=\int_{\RR}\nabla_x m_\epsilon(x-y)\big[f(y,\vv-\mathbf{z})-f(y,\vv)\big]\,dy
	\]
	and using $\|\nabla_x m_\epsilon\|_{L^1(\RR)}\leq C\epsilon^{-1}$ yields the bound
	$C\epsilon^{-1}|\mathbf{z}|^\alpha \V_f(|\mathbf{z}|)$. The second inequality follows in the same way,
	by $\|(\nabla_x f^\epsilon)^\delta-\nabla_x f^\epsilon\|_{L^p}\leq 2\|\nabla_x f^\epsilon\|_{L^p}$
	on the one hand, and by
	$\|(\nabla_x f^\epsilon)^\delta-\nabla_x f^\epsilon\|_{L^p}
	\leq C\delta^\alpha\|\nabla_x f^\epsilon\|_{L^p_x(\Omega;B_p^{\alpha,\infty}(\RR))}
	\leq C\epsilon^{-1}\delta^\alpha\V_f(\delta)$ on the other hand.
\end{proof}
\begin{remark}\label{re-diff-sharp}
	The two-route minimum above cannot be improved to the single rate
	$\epsilon^{\alpha-1}|\mathbf{z}|^\alpha$: the sequence
	$f_N(x,\vv)=N^{-\alpha}\sin(Nx_1)\sin(N v_1)$ is uniformly bounded in
	$B_\infty^{\alpha,\infty}(\Omega\times\RR)$, while for $\epsilon\sim N^{-1}$ and
	$|\mathbf{z}|\sim N^{-1}$ one checks that
	$\|\nabla_x f_N^\epsilon(\cdot,\cdot-\mathbf{z})-\nabla_x f_N^\epsilon\|_{L^\infty(K)}\sim N^{1-\alpha}$
	for any fixed compact subset $K\Subset\Omega\times\RR$,
	which is strictly larger than $\epsilon^{\alpha-1}|\mathbf{z}|^\alpha\sim N^{1-2\alpha}$.
	Consequently, the error terms in Lemmas \ref{le-error} and \ref{le-error-CNSV} below are
	estimated by the minimum of the two routes, and the choice $\delta=\epsilon$ then yields the
	rates actually used in the proofs.
\end{remark}

\begin{lemma}\label{le-Wf}
	Assume $f\in L^1(0,T;B_p^{\alpha,\infty}(\Omega\times\RR))$ with $\alpha\in(0,1)$, $p\in[1,\infty]$. Then
	$\W_f(\epsilon,\delta)$ is uniformly bounded for $\epsilon,\delta\in(0,1)$. If moreover $f$ belongs to the
	corresponding little Besov space, i.e. its Besov seminorm vanishes at small scales (which is in particular
	the case when $f$ enjoys slightly more regularity, say
	$f\in L^1(0,T;B_p^{\alpha',\infty}(\Omega\times\RR))$ for some $\alpha'>\alpha$), then
	\begin{eqnarray*}
		\W_f(\epsilon,\delta)\to0 \quad \epsilon,\delta\to0.
	\end{eqnarray*}
\end{lemma}
\begin{proof}
	By the definition of the Besov seminorm, for every $x,y\in\Omega$ and $t\in(0,T)$,
	\begin{align*}
		|x-y|^{-\alpha}\left(\int_\Omega\int_{\RR}\mathbf{1}_{\epsilon\leq |x-y|\leq
		2\epsilon}|f(x,\vv,t)-f(y,\vv,t) |^p d\vv dx\right)^\frac{1}{p}
		\leq C[f(t)]_{B_p^{\alpha,\infty}(\Omega\times\RR)},
	\end{align*}
	since the indicator only shrinks the domain of integration. Taking the supremum over $x\neq y$ shows that
	$\X_{f(t)}(\epsilon)\leq C[f(t)]_{B_p^{\alpha,\infty}(\Omega\times\RR)}$ for all $\epsilon\in(0,1)$, and the same
	argument in the velocity variable gives $\V_{f(t)}(\delta)\leq C[f(t)]_{B_p^{\alpha,\infty}(\Omega\times\RR)}$ for all
	$\delta\in(0,1)$. Integrating in time yields
	\begin{align*}
		\W_f(\epsilon,\delta)\leq C\|f\|_{L^1(0,T;B_p^{\alpha,\infty}(\Omega\times\RR))},
	\end{align*}
	which is the claimed uniform boundedness.
	If moreover $f$ belongs to the corresponding little Besov space, then its local seminorm vanishes at small scales,
	so that $\X_{f(t)}(\epsilon)+\V_{f(t)}(\delta)\to0$ as $\epsilon,\delta\to0$ for almost every $t$; since this
	integrand is dominated by the integrable function $C[f(t)]_{B_p^{\alpha,\infty}(\Omega\times\RR)}$, the dominated
	convergence theorem gives $\W_f(\epsilon,\delta)\to0$. Finally, if
	$f\in L^1(0,T;B_p^{\alpha',\infty}(\Omega\times\RR))$ for some $\alpha'>\alpha$, the more precise bounds
	$\X_{f(t)}(\epsilon)\leq C\epsilon^{\alpha'-\alpha}[f(t)]_{B_p^{\alpha',\infty}}$ and
	$\V_{f(t)}(\delta)\leq C\delta^{\alpha'-\alpha}[f(t)]_{B_p^{\alpha',\infty}}$ give the convergence directly.
\end{proof}
\begin{remark}\label{re-Wf}
	In the proofs of Theorems \ref{renormalization} and \ref{CNSV thm}, the quantity $\W_f(\epsilon,\delta)$
	always appears multiplied by a strictly positive power of $\epsilon$ or $\delta$ (see Lemmas \ref{le-error}
	and \ref{le-error-CNSV}), so the mere uniform boundedness of $\W_f$ already suffices for all the vanishing
	limits; the convergence $\W_f(\epsilon,\delta)\to0$ is therefore not needed for our results.
\end{remark}

We define the macroscopic density and momentum as 
	\[
	\bar{\rho}(x,t)=\int_{\RR} f\,d\vv, 
	\qquad 
	j(x,t)=\int_{\RR}\vv f\,d\vv.
	\]
	To derive the energy equality,
	we shall establish some estimates of the macroscopic density and momentum  from the  bound on the distribution density function $f$.
	\begin{lemma}\label{le-f-estimate}
		Let $0\leq f\in L^p(\Omega\times\RR)$ for some $p\in(1,\infty)$ satisfy $(1+\frac{|\vv|^2}{2})f\in L^1(\Omega\times\RR)$.
		Then there exists a constant $M>0$, depending only on $\|f\|_{L^p(\Omega\times\RR)}$ and
		$\|(1+\frac{|\vv|^2}{2})f\|_{L^1(\Omega\times\RR)}$, such that
		\begin{align*}
			\|\int_{\RR}f d\vv\|_{L^{p_1}(\Omega)}+\|\int_{\RR}\vv f d\vv\|_{L^{p_2}(\Omega)}&\leq M,
		\end{align*}
		where \[p_1\in\bigg[1,\frac{n(p-1)+2p}{n(p-1)+2}\bigg),\quad p_2\in\bigg[1,\frac{n(p-1)+2p}{(n-1)(p-1)+2p}\bigg).\]
	\end{lemma}
	\begin{proof}
		By assumption, $\|(1+\frac{|\vv|^2}{2})f\|_{L^1(\Omega\times\RR)}\leq C$ and $\|f\|_{L^p(\Omega\times\RR)}\leq C$.
		Using the interpolation inequality one has for any $0\leq a\leq1$,
		\begin{align*}
			&\int_{\RR} f d\vv
			=\int_{\RR}f^a(1+\frac{|\vv|^2}{2})^a f^{1-a}(1+\frac{|\vv|^2}{2})^{-a} d\vv\\
			\leq& \left(\int_{\RR}(1+\frac{|\vv|^2}{2})f d\vv\right)^a\left(\int_{\RR}f^p
			d\vv\right)^{\frac{1-a}{p}}\left(\int_{\RR}(1+\frac{|\vv|^2}{2})^\frac{-ap}{(1-a)(p-1)}d\vv\right)^\frac{(1-a)(p-1)}{p}.
		\end{align*}
		If 
		\begin{eqnarray*}
			\frac{2ap}{(1-a)(p-1)}>n,\,\, \text{ i.e. }\,\, a>\frac{n(p-1)}{n(p-1)+2p},
		\end{eqnarray*}
		then
		\begin{eqnarray*}
			\int_\Omega\left(\int_{\RR}fd\vv\right)^{p_1}dx\leq C \left(\int_\Omega\int_{\RR}f(1+\frac{|\vv|^2}{2}) d\vv
			dx\right)^{ap_1}\left(\int_\Omega\int_{\RR}f^p d\vv dx\right)^{1-ap_1}\leq  M,
		\end{eqnarray*}
		provided 
		\begin{eqnarray}
			\frac{(1-a)p_1}{p}=1-ap_1,\,\, \text{ i.e. }\,\, p_1=\frac{p}{1+a(p-1)},\,\, \text{ and then }\, 1\leq p_1<\frac{n(p-1)+2p}{n(p-1)+2}.
		\end{eqnarray}
		Similarly, for $\int_{\RR}\vv fd\vv$, one has for any $0\leq b\leq1$,
		\begin{align*}
			&\left|\int_{\RR}\vv fd\vv\right|
			\leq\int_{\RR}f^b(1+\frac{|\vv|^2}{2})^b f^{1-b}|\vv| (1+\frac{|\vv|^2}{2})^{-b}d\vv\\
			\leq& \left(\int_{\RR}(1+\frac{|\vv|^2}{2}) fd\vv\right)^b\left(\int_{\RR}f^p
			d\vv\right)^\frac{1-b}{p}\left(\int_{\RR}\left(|\vv|(1+\frac{|\vv|^2}{2})^{-b}\right)^\frac{p}{(1-b)(p-1)}d\vv\right)^\frac{(1-b)(p-1)}{p},
		\end{align*}
		then
		\begin{eqnarray*}
			\int_{\Omega}\left|\int_{\RR}\vv fd\vv\right|^{p_2}dx\leq C \left(\int_\Omega\int_{\RR}f(1+\frac{|\vv|^2}{2}) d\vv
			dx\right)^{bp_2}\left(\int_\Omega\int_{\RR}f^p d\vv dx\right)^{1-bp_2}\leq  M,
		\end{eqnarray*}
		provided 
		\begin{align}
			&\frac{(2b-1)p}{(1-b)(p-1)}>n,\,\, \text{ i.e. }\,\, b>\frac{n(p-1)+p}{n(p-1)+2p},\notag\\
			&\frac{(1-b)p_2}{p}=1-bp_2,\,\, i.e.\,\, p_2=\frac{p}{1+b(p-1)}\,\,\text{ and then }\,\, 1\leq p_2<\frac{n(p-1)+2p}{(n-1)(p-1)+2p}.
		\end{align}
	\end{proof}

Next, we show that the truncated macroscopic momentum flux inherits the Besov regularity of $f$,
which is crucial for the boundary treatment below.
\begin{lemma}\label{le-MR}
	Let $f\in B_p^{\alpha,\infty}(\Omega\times\RR)$ with $\alpha\in(0,1)$ and $p\in[1,\infty]$.
	For any $R>0$, define the truncated macroscopic momentum flux
	\[
	M_R(x)=\int_{|\vv|\leq 2R}\vv\Big(1+\frac{|\vv|^2}{2}\Big)f(x,\vv)\,d\vv.
	\]
	Then $M_R\in B_p^{\alpha,\infty}(\Omega)$, and there exists a constant $C>0$, depending only on $R$, $n$ and $p$, such that
	\[
	\|M_R\|_{B_p^{\alpha,\infty}(\Omega)}\leq C\|f\|_{B_p^{\alpha,\infty}(\Omega\times\RR)}.
	\]
\end{lemma}
\begin{proof}
	Write $K_R(\vv)=\vv\big(1+\frac{|\vv|^2}{2}\big)\mathbf{1}_{|\vv|\leq2R}$;
	then $|K_R(\vv)|\leq C(R)$, $\supp K_R\subset B_{2R}$ and $M_R(x)=\int_{\RR}K_R(\vv)f(x,\vv)\,d\vv$.
	For $h\in\R^n$ with $h\neq0$ and $|h|$ small,
	denote the difference in the $x$-direction by $\Delta^{(x)}_h f(x,\vv)=f(x+h,\vv)-f(x,\vv)$; then
	\[
	\Delta^{(x)}_h M_R(x)=\int_{\RR}K_R(\vv)\,\Delta^{(x)}_h f(x,\vv)\,d\vv.
	\]
	By Minkowski's integral inequality and $|K_R(\vv)|\leq C(R)$,
	\[
	\|\Delta^{(x)}_h M_R\|_{L^p(\Omega_h)}\leq C(R)\int_{|\vv|\leq2R}\|\Delta^{(x)}_h f(\cdot,\vv)\|_{L^p(\Omega_h)}\,d\vv,
	\]
	where $\Omega_h=\{x\in\Omega:\ x+h\in\Omega\}$. Applying H\"older's inequality in the $\vv$ variable,
	\[
	\int_{|\vv|\leq2R}\|\Delta^{(x)}_h f(\cdot,\vv)\|_{L^p(\Omega_h)}\,d\vv
	\leq C(R)\Big(\int_{\RR}\|\Delta^{(x)}_h f\|_{L^p(\Omega_h)}^p\,d\vv\Big)^{\frac1p}
	= C(R)\|\Delta^{(x)}_h f\|_{L^p(\Omega_h\times\RR)}.
	\]
	Since $f\in B_p^{\alpha,\infty}(\Omega\times\RR)$, the definition of the Besov seminorm gives
	\[
	\|\Delta^{(x)}_h f\|_{L^p(\Omega_h\times\RR)}\leq |h|^\alpha[f]_{B_p^{\alpha,\infty}(\Omega\times\RR)}.
	\]
	Hence
	\[
	|h|^{-\alpha}\|\Delta^{(x)}_h M_R\|_{L^p(\Omega_h)}\leq C(R)\|f\|_{B_p^{\alpha,\infty}(\Omega\times\RR)}.
	\]
	Since obviously $\|M_R\|_{L^p(\Omega)}\leq C(R)\|f\|_{L^p(\Omega\times\RR)}$,
	we conclude that $M_R\in B_p^{\alpha,\infty}(\Omega)$ together with the desired estimate.
\end{proof}

The following two lemmas provide the trace theory in Besov spaces needed for the boundary terms.
\begin{lemma}\label{le-trace}
	Let $\Omega\subset\R^n$ be a bounded domain with $C^1$ boundary, $0<\alpha<1$ and $1\leq p<\infty$.
	If $\alpha p>1$, then there exists a bounded linear trace operator
	\[
	\operatorname{tr}:\ B_p^{\alpha,\infty}(\Omega)\longrightarrow L^p(\partial\Omega),
	\]
	such that $\operatorname{tr}(u)=u|_{\partial\Omega}$ for any $u\in C(\overline{\Omega})\cap B_p^{\alpha,\infty}(\Omega)$.
	Moreover, there exists a constant $C=C(\Omega,\alpha,p)>0$ such that
	\[
	\|\operatorname{tr}(u)\|_{L^p(\partial\Omega)}\leq C\|u\|_{B_p^{\alpha,\infty}(\Omega)}.
	\]
\end{lemma}
\begin{proof}
	The existence and boundedness of the trace operator $\operatorname{tr}$ is a standard result in the theory of Besov spaces.
	More precisely,
	Theorem 4.7.1 of \cite{Triebel1978} establishes a bounded linear trace operator from the Besov space $B^{\alpha,q}_p(\R^n)$ onto
	$B^{\alpha-\frac1p,q}_p(\R^{n-1})$.
	When $\alpha p>1$, this operator restricts to $B_p^{\alpha,\infty}(\Omega)$ with values in $L^p(\partial\Omega)$.
	For bounded domains with $C^1$ boundary, the claim then follows by locally straightening the boundary and using a partition of unity (cf.
	the discussion in \cite[Section 4.7]{Triebel1978}).
\end{proof}

\begin{lemma}\label{le-layer}
	Let $\Omega\subset\R^n$ be a bounded domain with $C^2$ boundary, $0<\alpha<1$, $1\leq p<\infty$ and $\alpha p>1$.
	Let $\kappa\in C^\infty(\R)$ satisfy $\supp\kappa'\subset[1,2]$ and $\int_1^2\kappa'(\tau)\,d\tau=1$.
	Then for any $u\in B_p^{\alpha,\infty}(\Omega)$, as $h\to0^+$,
	\begin{equation}\label{eq-layer}
		\frac1h\int_{h\leq d(x,\partial\Omega)\leq2h}u(x)\,\kappa'\Big(\frac{d(x,\partial\Omega)}{h}\Big)\,dx\ \longrightarrow\
		\int_{\partial\Omega}\operatorname{tr}(u)(x)\,d\sigma_x,
	\end{equation}
	where $\operatorname{tr}(u)$ is the trace defined in Lemma \ref{le-trace} and $d(x,\partial\Omega)$ is the distance function to the
	boundary.
\end{lemma}
\begin{proof}
	By the discussion on the distance function preceding Lemma \ref{le-theta}, there exists $h_0=h_0(\Omega)>0$ such that, for $0<h<h_0$,
	the distance function $d(x)\equiv d(x,\partial\Omega)$ is $C^2$ on $\Omega_h=\{x\in\Omega:\ d(x)<h\}$ and satisfies $\nabla
	d(x)=-\nu(\sigma(x))$,
	where $\nu$ is the unit outward normal to $\partial\Omega$ and $\sigma(x)$ is the nearest-point projection of $x$ onto $\partial\Omega$.
	Moreover, for $\tau\in[1,2]$,
	the level set $\{x\in\Omega:\ d(x)=h\tau\}$ is a $C^1$-diffeomorphic image of $\partial\Omega$ whose surface element satisfies
	\begin{equation}\label{eq-surface}
		d\sigma_x=\big(1+O(h\tau)\big)\,d\sigma_y,\qquad x=y-h\tau\,\nu(y),
	\end{equation}
	where $d\sigma_y$ is the surface element on $\partial\Omega$.
	
	Using the change of variables $t=d(x)$, the integral on the left-hand side of \eqref{eq-layer} can be rewritten as
	\[
	\frac1h\int_{h\leq d\leq2h}u(x)\,\kappa'\Big(\frac{d}{h}\Big)\,dx
	=\int_1^2\Big(\int_{\{d=h\tau\}}u(x)\,d\sigma_x\Big)\kappa'(\tau)\,d\tau.
	\]
	By \eqref{eq-surface} and the fact that $u\in B_p^{\alpha,\infty}(\Omega)$ with $\alpha p>1$,
	Lemma \ref{le-trace} implies that $u$ admits a well-defined $L^p$ trace $\operatorname{tr}(u)$ on $\partial\Omega$, and
	\[
	\int_{\{d=h\tau\}}u(x)\,d\sigma_x
	=\int_{\partial\Omega}u(y-h\tau\,\nu(y))\,\big(1+O(h\tau)\big)\,d\sigma_y
	\ \longrightarrow\ \int_{\partial\Omega}\operatorname{tr}(u)(y)\,d\sigma_y
	\]
	as $h\to0^+$, uniformly with respect to $\tau\in[1,2]$.
	Here the convergence $u(\,\cdot\,-h\tau\nu)\to\operatorname{tr}(u)$ holds in $L^p(\partial\Omega)$ as part of the trace construction
	for Besov spaces with $\alpha p>1$ (cf. \cite[Section 4.7]{Triebel1978}),
	and the rate depends only on $\|u\|_{B_p^{\alpha,\infty}(\Omega)}$, which gives the uniformity in $\tau\in[1,2]$. Therefore
	\[
	\int_1^2\Big(\int_{\{d=h\tau\}}u(x)\,d\sigma_x\Big)\kappa'(\tau)\,d\tau
	\ \longrightarrow\ \int_{\partial\Omega}\operatorname{tr}(u)(y)\,d\sigma_y\int_1^2\kappa'(\tau)\,d\tau
	=\int_{\partial\Omega}\operatorname{tr}(u)\,d\sigma_y,
	\]
	which completes the proof.
\end{proof}

\section{Proof of Theorem \ref{renormalization}}\label{se-pro-th1}
To derive the renormalization property, we shall establish the following commutator estimates.

\begin{lemma}\label{le-error}
	Let the pair $(f,\uu)$ satisfy the regularity hypotheses in Theorem \ref{renormalization}. Then we get
	\begin{align}
		&\|\nabla_x(\vv f)^{\epsilon,\delta,\eta}-\vv^\delta\cdot\nabla_x f^{\epsilon,\delta,\eta}\|_{L^1(0,T;L^p(\Omega\times\RR))}\notag\\
		&\quad\leq C\big(\epsilon^{\alpha-1}\delta\wedge\epsilon^{-1}\delta^{\alpha+1}\big)\W_f(\epsilon,\delta),\label{eq-comm-1}\\
		&\|\nabla_\vv\cdot(f\uu)^{\epsilon,\delta,\eta}\notag-\nabla_\vv f^{\epsilon,\delta,\eta}\cdot
		\uu^{\epsilon,\eta}\|_{L^r(0,T;L^r(\Omega;L^p(\RR)))}\notag\\
		&\quad\leq C\epsilon^{\beta}\big(\epsilon^{\alpha}\delta^{-1}\wedge\delta^{\alpha-1}\big)
		\|f\|_{L^p(0,T;B_p^{\alpha,\infty}(\Omega\times\RR))}\notag\\
		&\qquad\times\|\uu\|_{B_q^{\beta,\infty}([0,T]\times\Omega)},\label{eq-comm-2}\\
		&\|\nabla_\vv\cdot (f\vv)^{\epsilon,\delta,\eta}-\nabla_\vv \cdot(f^{\epsilon,\delta,\eta}
		\vv^\delta)\|_{L^1(0,T;L^p(\Omega\times\RR))}\notag\\
		&\quad\leq C\delta^\alpha \W_f(\epsilon,\delta).\label{eq-comm-3}
	\end{align}
	where the constants $C$ in \eqref{eq-comm-1} and \eqref{eq-comm-3} depend on
	$\|f\|_{L^1(0,T;B_p^{\alpha,\infty}(\Omega\times\RR))}$, while the constant in \eqref{eq-comm-2} depends on
	$\|f\|_{L^p(0,T;B_p^{\alpha,\infty}(\Omega\times\RR))}$ and $\|\uu\|_{B_q^{\beta,\infty}([0,T]\times\Omega)}$.
	The parameter $r$ satisfies $\frac{1}{p}+\frac{1}{q}=\frac{1}{r}<1$.
	In particular, taking $\delta=\epsilon$ the right-hand sides are of order
	$\epsilon^{\alpha}\W_f$, $\epsilon^{\alpha+\beta-1}\|f\|_{L^p(0,T;B_p^{\alpha,\infty}(\Omega\times\RR))}
	\|\uu\|_{B_q^{\beta,\infty}([0,T]\times\Omega)}$ and $\epsilon^{\alpha}\W_f$,
	respectively, and \eqref{eq-comm-2} tends to zero provided $\alpha+\beta>1$.
\end{lemma}


\begin{proof}
	We define the integral
	\begin{align*}
		\omega_\delta(g,f)(x,\vv,t)&=\int_{\RR}m_\delta(\mathbf{w})(g(\vv-\mathbf{w})-g(\vv))(f(x,\vv-\mathbf{w},t)-f(x,\vv,t))d\mathbf{w}\\
		&=(fg)^\delta-f^\delta g-fg^\delta+fg.
	\end{align*}
	Then
	\begin{eqnarray*}
		\omega_\delta(\vv,f^{\epsilon,\eta})=(f\vv)^{\epsilon,\delta,\eta}-\vv^\delta f^{\epsilon,\eta}-f^{\epsilon,\delta,\eta}\vv+
		f^{\epsilon,\eta}\vv.
	\end{eqnarray*}
	That is
	\begin{eqnarray}\label{omegadelta}
		(f\vv)^{\epsilon,\delta,\eta}-\vv^\delta
		f^{\epsilon,\delta,\eta}=\omega_\delta(\vv,f^{\epsilon,\eta})+(f^{\epsilon,\delta,\eta}-f^{\epsilon,\eta})(\vv-\vv^\delta).
	\end{eqnarray}
	Using triangle inequality and Minkowski inequality, we will establish the estimate
	\begin{align*}
		&\|\nabla_x(\vv f)^{\epsilon,\delta,\eta}-\vv^\delta\cdot\nabla_x f^{\epsilon,\delta,\eta}\|_{L^1(0,T;L^p(\Omega\times\RR))}\\
		&\quad=\|\nabla_x\cdot((\vv f)^{\epsilon,\delta,\eta}-\vv^\delta f^{\epsilon,\delta,\eta})\|_{L^1(0,T;L^p(\Omega\times\RR))}\\
		&\quad\leq \|\nabla_x\cdot\omega_\delta(\vv,f^{\epsilon,\eta})\|_{{L^1(0,T;L^p(\Omega\times\RR))}}+\|\nabla_x\cdot
		(f^{\epsilon,\delta,\eta}-f^{\epsilon,\eta})(\vv^\delta-\vv)\|_{L^1(0,T;L^p(\Omega\times\RR))}\\
		&\quad\leq \|\nabla_x\cdot\omega_\delta(\vv,f^\epsilon)\|_{{L^1(0,T;L^p(\Omega\times\RR))}}+\|\nabla_x\cdot
		(f^{\epsilon,\delta}-f^{\epsilon})(\vv^\delta-\vv)\|_{L^1(0,T;L^p(\Omega\times\RR))}\\
		&=A_{11}+A_{12}.
	\end{align*}
	We have
	\begin{align*}
		A_{11}&=\|\nabla_x\cdot\omega_\delta(\vv,f^\epsilon)\|_{{L^1(0,T;L^p(\Omega\times\RR))}}\\
		&\leq
		\int_{0}^{T}\int_{\RR}m_\delta(\mathbf{w})\,|\mathbf{w}|\,\|\nabla_x f^\epsilon(x,\vv-\mathbf{w},t)-\nabla_x f^\epsilon(x,\vv,t)\|_{L^p(\Omega\times\RR)}d\mathbf{w}dt\\
		&\leq C\int_{0}^{T}\int_{\RR}m_\delta(\mathbf{w})\,|\mathbf{w}|\,
		\big(\epsilon^{\alpha-1}\X_f(\epsilon)\wedge\epsilon^{-1}|\mathbf{w}|^\alpha \V_f(|\mathbf w|)\big)d\mathbf{w}dt\\
		&\leq C\big(\epsilon^{\alpha-1}\delta\wedge\epsilon^{-1}\delta^{\alpha+1}\big)\int_{0}^{T}\big(\X_f(\epsilon)+\V_f(\delta)\big)dt\\
		&\leq C\big(\epsilon^{\alpha-1}\delta\wedge\epsilon^{-1}\delta^{\alpha+1}\big)\W_f(\epsilon,\delta).
	\end{align*}
	Here we used the first inequality in Lemma \ref{le-diff}, and $C$ depends on the smooth function $m$ and $\alpha\in(0,1)$. Similarly, since $|\vv^\delta-\vv|\leq C\delta$ is a constant vector, the second inequality in Lemma \ref{le-diff} gives
	\begin{align*}
		A_{12}&=\|\nabla_x\cdot (f^{\epsilon,\delta}-f^{\epsilon})(\vv^\delta-\vv)\|_{L^1(0,T;L^p(\Omega\times\RR))}\\
		&\leq C\delta\,\|(\nabla_x f^\epsilon)^\delta-\nabla_x f^\epsilon\|_{L^1(0,T;L^p(\Omega\times\RR))}\\
		&\leq C\delta\big(\epsilon^{\alpha-1}\delta\wedge\epsilon^{-1}\delta^{\alpha}\big)\int_{0}^{T}\big(\X_f(\epsilon)+\V_f(\delta)\big)dt\\
		&\leq C\big(\epsilon^{\alpha-1}\delta\wedge\epsilon^{-1}\delta^{\alpha+1}\big)\W_f(\epsilon,\delta).
	\end{align*}
	So we get \eqref{eq-comm-1}. We denote the integral
	\begin{align*}
		\omega_{\epsilon,\eta}(g,f)(x,\vv,t)&=\int_{\R}m_\eta(\tau)\int_{\Omega}m_\epsilon(y)(g(x-y,t-\tau)-g(x,t))(f(x-y,\vv,t-\tau)-f(x,\vv,t))dyd\tau\\
		&=(fg)^{\epsilon,\eta}-f^{\epsilon,\eta}g-fg^{\epsilon,\eta}+fg.
	\end{align*}
	Then we have
	\begin{eqnarray}
		(f
		\uu)^{\epsilon,\delta,\eta}-f^{\epsilon,\delta,\eta}\uu^{\epsilon,\eta}=\omega_{\epsilon,\eta}(\uu,f^\delta)-(f^\delta-f^{\epsilon,\delta,\eta})(\uu-\uu^{\epsilon,\eta}).
	\end{eqnarray}
	It follows that
	\begin{align*}
		\|\nabla_\vv&\cdot(f\uu)^{\epsilon,\delta,\eta}-\nabla_\vv f^{\epsilon,\delta,\eta}\cdot
		\uu^{\epsilon,\eta}\|_{L^r(0,T;L^r(\Omega;L^p(\RR)))}\\
		&\leq\|\nabla_\vv\cdot\omega_{\epsilon,\eta}(\uu,f^\delta)\|_{L^r(0,T;L^r(\Omega;L^p(\RR)))}+\|\nabla_\vv\cdot(f^\delta-f^{\epsilon,\delta,\eta})(\uu-\uu^{\epsilon,\eta})\|_{L^r(0,T;L^r(\Omega;L^p(\RR)))}\\
		&=A_{21}+A_{22}.
	\end{align*}
	Using Lemma \ref{le-diff} and the H\"older inequality for $\alpha,\beta\in(0,1)$, we estimate
	$\|\nabla_\vv(f^\delta(x-y,\vv,t)-f^\delta(x,\vv,t))\|_{L^p(\Omega\times\RR)}$ in two ways: on the one hand it is
	bounded by $2\|\nabla_\vv f^\delta\|_{L^p}\leq C\delta^{\alpha-1}\V_f(\delta)$; on the other hand, since
	$\int_{\RR}\nabla m_\delta=0$, one can write
	\[
	\nabla_\vv f^\delta(x,\vv,t)=\int_{\RR}\nabla_\vv m_\delta(\vv-\mathbf w)\big[f(x,\mathbf w,t)-f(x,\vv,t)\big]d\mathbf w,
	\]
	so that its $x$-difference is bounded by $C|y|^\alpha\delta^{-1}\X_f(|y|)$. Hence
	\begin{align*}
		A_{21}&=\|\nabla_\vv\cdot\omega_{\epsilon,\eta}(\uu,f^\delta)\|_{L^r(0,T;L^r(\Omega;L^p(\RR)))}\\
		&\leq\|\nabla_\vv\cdot\omega_{\epsilon}(\uu,f^\delta)\|_{L^r(0,T;L^r(\Omega;L^p(\RR)))}\\
		&\leq\int_\Omega m_\epsilon(y)\big\|(\uu(x-y,t)-\uu(x,t))\cdot\nabla_\vv\big(f^\delta(x-y,\vv,t)\\
		&\qquad-f^\delta(x,\vv,t)\big)\big\|_{L^r(0,T;L^r(\Omega;L^p(\RR)))}dy\\
		&\leq\int_{\Omega}m_\epsilon(y)\|\uu(x-y,t)-\uu(x,t)\|_{L^q(0,T;L^q(\Omega))}\\
		&\qquad\cdot\|\nabla_\vv(f^\delta(x-y,\vv,t)-f^\delta(x,\vv,t))\|_{L^p(0,T;L^p(\Omega\times\RR))}dy\\
		&\leq C\epsilon^{\beta}\big(\epsilon^{\alpha}\delta^{-1}\|\X_f(\epsilon)\|_{L^p(0,T)}\wedge\delta^{\alpha-1}\|\V_f(\delta)\|_{L^p(0,T)}\big)\|\uu\|_{B_q^{\beta,\infty}([0,T]\times\Omega)}\\
		&\leq C\epsilon^{\beta}\big(\epsilon^{\alpha}\delta^{-1}\wedge\delta^{\alpha-1}\big)\|\uu\|_{B_q^{\beta,\infty}([0,T]\times\Omega)}\\
		&\qquad\times\|f\|_{L^p(0,T;B_p^{\alpha,\infty}(\Omega\times\RR))},
	\end{align*}
	For $A_{22}$, the same two routes give
	$\|\nabla_\vv(f^\delta-f^{\epsilon,\delta})\|_{L^p(\Omega\times\RR)}\leq C\big(\epsilon^{\alpha}\delta^{-1}\X_f(\epsilon)\big)\wedge\big(\delta^{\alpha-1}\V_f(\delta)\big)$, and hence
	\begin{align*}
		A_{22}&=\|\nabla_\vv\cdot(f^\delta-f^{\epsilon,\delta,\eta})(\uu-\uu^{\epsilon,\eta})\|_{L^r(0,T;L^r(\Omega;L^p(\RR)))}\\
		&\leq \|\nabla_\vv\cdot(f^\delta-f^{\epsilon,\delta})(\uu-\uu^{\epsilon})\|_{L^r(0,T;L^r(\Omega;L^p(\RR)))}\\
		&\leq \|\nabla_\vv(f^\delta-f^{\epsilon,\delta})\|_{L^p(0,T;L^p(\Omega\times\RR))}\|\uu-\uu^\epsilon\|_{L^q(0,T;L^q(\Omega))}\\
		&\leq C\epsilon^{\beta}\big(\epsilon^{\alpha}\delta^{-1}\wedge\delta^{\alpha-1}\big)\|\uu\|_{B_q^{\beta,\infty}([0,T]\times\Omega)}\|f\|_{L^p(0,T;B_p^{\alpha,\infty}(\Omega\times\RR))}.
	\end{align*}
	So the proof of \eqref{eq-comm-2} is done. For \eqref{eq-comm-3}, writing $g=f^{\epsilon,\eta}$ for brevity, we decompose
	\begin{align*}
		\nabla_\vv\cdot (f\vv)^{\epsilon,\delta,\eta}-\nabla_\vv \cdot(f^{\epsilon,\delta,\eta} \vv^\delta)
		=\nabla_\vv\cdot\big[(g\vv)^{\delta}-g^{\delta}\vv\big]+\nabla_\vv\cdot\big[g^{\delta}(\vv-\vv^\delta)\big]
		=A_{31}+A_{32}.
	\end{align*}
	Since $\int_{\RR}\nabla m_\delta=0$ and $\vv-\vv^\delta=-\int_{\RR}m_\delta(\mathbf w)\,\mathbf w\,d\mathbf w$ is a constant vector
	satisfying $|\vv-\vv^\delta|\leq C\delta$, integration by parts under the integral sign shows that the derivatives fall only on the
	mollification kernels:
	\begin{align*}
		A_{31}&=\Big\|\int_{\RR}\nabla_\vv m_\delta(\vv-\mathbf w)\cdot(\mathbf w-\vv)\big[g(\mathbf w)-g(\vv)\big]d\mathbf w
		+n\big(g-g^\delta\big)\Big\|_{L^1(0,T;L^p(\Omega\times\RR))}\\
		&\leq C\int_{0}^{T}\int_{\RR}\Big(|\nabla m_\delta(\mathbf w)|\,|\mathbf w|^{\alpha+1}+m_\delta(\mathbf w)\,|\mathbf w|^{\alpha}\Big)
		\V_f(|\mathbf w|)\,d\mathbf w dt\\
		&\leq C\delta^{\alpha}\W_f(\epsilon,\delta),
	\end{align*}
	where we also used $\|g-g^\delta\|_{L^p(\Omega\times\RR)}\leq C\delta^\alpha\V_f(\delta)$, and, since $\nabla_\vv\cdot(\vv-\vv^\delta)=0$,
	\begin{align*}
		A_{32}&=\|(\vv-\vv^\delta)\cdot\nabla_\vv g^\delta\|_{L^1(0,T;L^p(\Omega\times\RR))}
		\leq C\delta\,\|\nabla_\vv g^\delta\|_{L^1(0,T;L^p(\Omega\times\RR))}\\
		&\leq C\delta\cdot\delta^{\alpha-1}\W_f(\epsilon,\delta)
		\leq C\delta^{\alpha}\W_f(\epsilon,\delta).
	\end{align*}
	This completes the proof.
\end{proof}

Now we show the renormalization property.
\begin{proof}
For any test function $\Theta\in\D((0,T)\times\Omega\times\RR)$, we have the weak formulation for the Vlasov equation
	\begin{equation*}
		\int_{0}^{T}\int_\Omega\int_{\RR}f(\partial_t\Theta+\vv\nabla_x\Theta+(\uu-\vv)\nabla_\vv\Theta)d\vv dx dt=0.
	\end{equation*}
	Choosing the test function
	\begin{eqnarray*}
		\Theta=\theta_h(\mathcal{R}'(f_h^{\epsilon,\delta,\eta})\Phi)^{\epsilon,\delta,\eta}\in\D((0,T)\times\Omega\times\RR),
	\end{eqnarray*}
	where $\Phi\in \D((0,T)\times\Omega\times\RR)$,
	$f_h:=\theta_h f$ and $\mathcal{R}\in \mathcal{S}_k(\R^+;\R^+)$ with $k\in\big[2,2p\big(1-\frac1q\big)\big]$.  Integration by parts, we obtain 
	\begin{align*}
		&\int_{0}^{T} \int_{\Omega} \int_{\RR} \mathcal{R}(f_h^{\epsilon,\delta,\eta}) \left( \partial_t \Phi + \mathbf{v}^\delta \nabla_x \Phi +
		(\mathbf{u}^{\epsilon,\eta} - \mathbf{v}^\delta) \nabla_{\mathbf{v}} \Phi \right) \\
		&\quad + \mathcal{R}'(f_h^{\epsilon,\delta,\eta}) \Phi \left( \mathbf{v}^\delta \nabla_x f_h^{\epsilon,\delta,\eta} - \nabla_x (\mathbf{v}
		f_h)^{\epsilon,\delta,\eta} + \uu^{\epsilon,\eta} \nabla_{\mathbf{v}} f_h^{\epsilon,\delta,\eta} \right. \\
		&\quad \left. - \nabla_{\mathbf{v}} (\mathbf{u} f_h)^{\epsilon,\delta,\eta} + \nabla_{\mathbf{v}} \cdot (\mathbf{v}
		f_h)^{\epsilon,\delta,\eta} - \nabla_{\mathbf{v}} \cdot (\mathbf{v}^\delta f_h^{\epsilon,\delta,\eta}) \right)\\
		&\quad+\Phi(\nabla_\vv\cdot\vv^\delta)\Big(\mathcal{R}'(f_h^{\epsilon,\delta,\eta})f_h^{\epsilon,\delta,\eta}-\mathcal{R}(f_h^{\epsilon,\delta,\eta})\Big)\\
		&\quad+ \vv f(R'(f_h^{\epsilon,\delta,\eta})\Phi)^{\epsilon,\delta,\eta}\nabla_x\theta_h d\vv dx dt:= I_1+I_2=0,
	\end{align*}
	where we define
	\begin{align*}
		I_2=\int_{0}^{T}\int_\Omega\int_{\RR}\vv f(R'(f_h^{\epsilon,\delta,\eta})\Phi)^{\epsilon,\delta,\eta}\nabla_x\theta_h d\vv dx dt.
	\end{align*}
Since $\Phi\in\D((0,T)\times\Omega\times\RR)$ is compactly supported,
there exists $d_0>0$ such that the spatial support of $\Phi$ is contained in $\{x\in\Omega:\ d(x,\partial\Omega)\geq d_0\}$.
Because the spatial convolution enlarges the support by at most $2\epsilon$,
for all $\epsilon<d_0/4$ the function $(\mathcal{R}'(f_h^{\epsilon,\delta,\eta})\Phi)^{\epsilon,\delta,\eta}$ is supported in
$\{d(x,\partial\Omega)\geq d_0/2\}$, while $\nabla_x\theta_h$ is supported in $\{h\leq d(x,\partial\Omega)\leq 2h\}$.
Hence, for all $h<d_0/4$ the two supports are disjoint and
\begin{align*}
	I_2=0.
\end{align*}
Therefore, we have 
\begin{align*}
	\lim_{h\to0} I_1+I_2&=\int_{0}^{T} \int_{\Omega} \int_{\RR} \mathcal{R}(f^{\epsilon,\delta,\eta}) \left( \partial_t \Phi +
	\mathbf{v}^\delta \nabla_x \Phi + (\mathbf{u}^{\epsilon,\eta} - \mathbf{v}^\delta) \nabla_{\mathbf{v}} \Phi \right) \\
		&\quad + \mathcal{R}'(f^{\epsilon,\delta,\eta}) \Phi \left( \mathbf{v}^\delta \nabla_x f^{\epsilon,\delta,\eta} - \nabla_x (\mathbf{v}
		f)^{\epsilon,\delta,\eta} + \uu^{\epsilon,\eta} \nabla_{\mathbf{v}} f^{\epsilon,\delta,\eta} \right. \\
		&\quad \left. - \nabla_{\mathbf{v}} (\mathbf{u} f)^{\epsilon,\delta,\eta} + \nabla_{\mathbf{v}} \cdot (\mathbf{v}
		f)^{\epsilon,\delta,\eta} - \nabla_{\mathbf{v}} \cdot (\mathbf{v}^\delta f^{\epsilon,\delta,\eta}) \right)\\
		&\quad+\Phi(\nabla_\vv\cdot\vv^\delta)\Big(\mathcal{R}'(f^{\epsilon,\delta,\eta})f^{\epsilon,\delta,\eta}-\mathcal{R}(f^{\epsilon,\delta,\eta})\Big) d\vv dxdt=0
\end{align*}

	Then, by Lemma \ref{le-error} we have
	\begin{align*}
		&\biggl| \int_{0}^{T} \int_{\Omega} \int_{\RR} \mathcal{R}(f^{\epsilon,\delta,\eta}) (\partial_t \Phi + \mathbf{v}^\delta \cdot \nabla_x
		\Phi + (\mathbf{u}^{\epsilon,\eta} - \mathbf{v}^\delta) \cdot \nabla_{\mathbf{v}} \Phi) \\
		&\quad +
		\Phi(\nabla_\vv\cdot\vv^\delta)\Big(\mathcal{R}'(f^{\epsilon,\delta,\eta})f^{\epsilon,\delta,\eta}-\mathcal{R}(f^{\epsilon,\delta,\eta})\Big) \,d\vv dx dt \biggr|\\
		&= \biggl|\int_{0}^{T}\int_\Omega\int_{\RR}\mathcal{R}'(f^{\epsilon,\delta,\eta})\Phi\left( \mathbf{v}^\delta \nabla_x
		f^{\epsilon,\delta,\eta} - \nabla_x (\mathbf{v} f)^{\epsilon,\delta,\eta} + \uu^{\epsilon,\eta} \nabla_{\mathbf{v}}
		f^{\epsilon,\delta,\eta} \right. \\
		&\quad \left. - \nabla_{\mathbf{v}} (\mathbf{u} f)^{\epsilon,\delta,\eta} + \nabla_{\mathbf{v}} \cdot (\mathbf{v}
		f)^{\epsilon,\delta,\eta} - \nabla_{\mathbf{v}} \cdot (\mathbf{v}^\delta f^{\epsilon,\delta,\eta}) \right)d\vv dx  dt\biggr|\\
		&\leq
		C\Big(\big(\epsilon^{\alpha-1}\delta\wedge\epsilon^{-1}\delta^{\alpha+1}\big)+\delta^\alpha\Big)\W_f(\epsilon,\delta)\\
		&\quad+C\epsilon^{\beta}\big(\epsilon^{\alpha}\delta^{-1}\wedge\delta^{\alpha-1}\big)
		\|f\|_{L^p(0,T;B_p^{\alpha,\infty}(\Omega\times\RR))}\notag\\
		&\qquad\times\|\uu\|_{B_q^{\beta,\infty}([0,T]\times\Omega)},
	\end{align*}
	where we used the H\"older inequality to pair $\mathcal{R}'(f^{\epsilon,\delta,\eta})\Phi$ with the commutators, which belong to
	$L^p$ and to $L^r$ with $\frac1r=\frac1p+\frac1q$ by Lemma \ref{le-error}. Since $f\in L^p(\Omega\times\RR)$ and $p'$, $r'$ are the
	conjugate exponents of $p$, $r$, this pairing requires
	\begin{align*}
		\Big(\frac{k}{2}-1\Big)p'\leq p \quad\text{and}\quad \Big(\frac{k}{2}-1\Big)r'\leq p,
	\end{align*}
	the sharper of which, in view of $\frac{p}{r'}=p\Big(1-\frac1p-\frac1q\Big)=p-1-\frac{p}{q}$, reads
	\begin{align*}
		k\leq 2p\Big(1-\frac1q\Big).
	\end{align*}

	Choosing $\delta=\epsilon$ and assuming that $\alpha+\beta>1$, we have
	\begin{align*}
		&\biggl| \int_{0}^{T} \int_{\Omega} \int_{\RR} \mathcal{R}(f^{\epsilon,\delta,\eta}) (\partial_t \Phi + \mathbf{v}^\delta \cdot \nabla_x
		\Phi + (\mathbf{u}^{\epsilon,\eta} - \mathbf{v}^\delta) \cdot \nabla_{\mathbf{v}} \Phi) \\
		&\quad + \Phi(\nabla_\vv\cdot\mathbf{v}^\delta) (f^{\epsilon,\delta,\eta} \mathcal{R}'(f^{\epsilon,\delta,\eta}) -
		\mathcal{R}(f^{\epsilon,\delta,\eta})) \, d\mathbf{v} \, dx \, dt \biggr|\\
		&\leq C\Big(\epsilon^{\alpha}\W_f(\epsilon,\delta)+\epsilon^{\alpha+\beta-1}\|f\|_{L^p(0,T;B_p^{\alpha,\infty}(\Omega\times\RR))}\|\uu\|_{B_q^{\beta,\infty}([0,T]\times\Omega)}\Big)\to0 \quad as \quad\epsilon,\delta\to0.
	\end{align*}
	Then we get the renormalization property \eqref{eq-renorm-1}. 
\end{proof}
Next we show the energy equality.
For the weak solution \((\uu, f)\), we take the test function 
\[
\Theta^{h,\epsilon} = \theta_h((\phi(t)\theta_h\uu)^\epsilon)^\epsilon,
\]
where \(\phi(t)\in\mathcal{D}(0,T)\) is a temporal cutoff function and \(\theta_h\) is the spatial cutoff near the boundary.

Testing the Navier--Stokes equation with \(\Theta^{h,\epsilon}\), we obtain
\begin{equation*}
	\int_{0}^{T}\int_\Omega 
	\Theta^{h,\epsilon}\cdot
	\Big(
	\uu_t + \nabla\cdot(\uu\otimes\uu) + \nabla P 
	- \mu\Delta\uu 
	+ \int_{\RR}(\uu-\vv)f\,d\vv
	\Big)
	dxdt = 0,
\end{equation*}
which, by the properties of convolution, can be rewritten as
\begin{equation}\label{eq-energy1}
	\begin{split}
	&\int_{0}^{T}\int_\Omega 
	\phi(t)(\theta_h\uu)^\epsilon
	\big(\theta_h\uu_t\big)^\epsilon dxdt+ \int_{0}^{T}\int_\Omega 
	\phi(t)(\theta_h\uu)^\epsilon\big( \theta_h\nabla\cdot(\uu\otimes\uu)\big)^\epsilon dxdt\\
	&\quad +
	\int_{0}^{T}\int_\Omega
	\phi(t)(\theta_h\uu)^\epsilon\big(-\mu\theta_h\Delta\uu\big)^\epsilon dxdt
	+\int_{0}^{T}\int_\Omega 
	\phi(t)(\theta_h\uu)^\epsilon\big(
	 \theta_h\nabla P\big)^\epsilon dxdt\\
	 &\quad+\int_{0}^{T}\int_\Omega 
	 \phi(t)(\theta_h\uu)^\epsilon\Big(\theta_h\int_{\RR}(\uu-\vv)f\,d\vv
	\Big)^\epsilon
	dxdt:
	=\sum_{i=1}^5I_i=0.
	\end{split}
\end{equation}
In the following, we investigate the five terms one by one.

\begin{proof}
		\textbf{Estimate of $I_1$.} By the time regularity of $\uu$ and the role of $\phi(t),$ we immediately have
	\begin{eqnarray*}
		I_1:=\int_{0}^{T}\int_\Omega\phi(t)(\theta_h \uu)^\epsilon\cdot(\theta_h\uu_t)^\epsilon dxdt
		=\int_{0}^{T}\phi(t)\frac{1}{2}\frac{d}{dt}\left(\int_\Omega|(\theta_h\uu)^\epsilon|^2 dx\right)dt,
	\end{eqnarray*}
	which yields 
	\begin{eqnarray}\label{t1}
		\lim_{\epsilon,h\to0}I_1
		=\frac{1}{2}\int_{0}^{T}\phi(t)\frac{d}{dt}\left(\int_\Omega|\uu|^2 dx\right)dt.
	\end{eqnarray}
	
	\textbf{Estimates of $I_2$.} We divide $I_2$ into two terms:
	\begin{align*}
		I_2:=&\int_{0}^{T}\int_\Omega \phi(t)(\theta_h \uu)^\epsilon\cdot (\theta_h\nabla\cdot(\uu\otimes\uu))^\epsilon dxdt\\
		=&-\int_{0}^{T}\int_\Omega \phi(t)(\nabla\theta_h(\uu\otimes\uu))^\epsilon  (\theta_h\uu)^\epsilon dxdt-\int_{0}^{T}\int_\Omega
		\phi(t)(\theta_h(\uu\otimes\uu))^\epsilon \nabla\cdot(\theta_h\uu)^\epsilon dxdt\\
		=:&A_1+A_2.
	\end{align*} For the first term $A_1$ we have 

	\begin{align*}
		A_1
		&=\int_{0}^{T}\int_\Omega \phi(t)\left[(\uu\otimes\uu)^\epsilon\nabla\theta_h-(\nabla\theta_h(\uu\otimes\uu))^\epsilon\right] 
		(\theta_h\uu)^\epsilon dxdt\\
		&\quad+\int_{0}^{T}\int_\Omega\phi(t) [\uu\otimes\uu^\epsilon-(\uu\otimes\uu)^\epsilon]\nabla\theta_h (\theta_h\uu)^\epsilon dxdt\notag\\
		&\quad-\int_{0}^{T}\int_\Omega \phi(t)\uu\otimes\uu^\epsilon\nabla\theta_h (\theta_h\uu)^\epsilon dxdt\\
		&=A_{11}+A_{12}+A_{13}.
	\end{align*}
	Note that $\uu$ is bounded in $L^q(0,T;L^q(\Omega))$, and recall that $h=\epsilon^\nu$ for some $0<\nu<1$.
	Both products in $A_{11}$ are supported on the strip $\{h\le d(x,\partial\Omega)\le 2h\}$ (the second one up to
	the $O(\epsilon)$-enlargement caused by the mollification, with $\epsilon\le h$), so that
	\begin{align*}
		|A_{11}|
		&\leq \|\phi\|_{C^0}\int_{0}^{T}\int_\Omega \Big|\left[(\uu\otimes\uu)^\epsilon\nabla\theta_h
		-(\nabla\theta_h(\uu\otimes\uu))^\epsilon\right]\notag\\
		&\qquad\qquad\times(\theta_h\uu)^\epsilon\Big|\,dxdt\\
		&\leq C \|\phi\|_{C^0}\|(\theta_h\uu)^\epsilon\|_{L^q(0,T;L^q(\Omega_{3h}))}
		\|\nabla\theta_h\|_{L^{\frac{q}{q-3}}(0,T;L^{\frac{q}{q-3}}(\Omega))}\notag\\
		&\qquad\cdot\Big(\|(\uu\otimes\uu)^\epsilon\|_{L^{\frac q2}(0,T;L^{\frac q2}(\Omega))}
		+\|(\nabla\theta_h(\uu\otimes\uu))^\epsilon\|_{L^{\frac q2}(0,T;L^{\frac q2}(\Omega))}\Big)\\
		&\leq C \|\phi\|_{C^0}\|\uu\|_{B_q^{\beta,\infty}([0,T]\times\Omega)}^3\,h^{\beta-\frac3q}\to 0,
	\end{align*}
	where we used $|\nabla\theta_h|\le C/h$ (hence $\|\nabla\theta_h\|_{L^{q/(q-3)}}\le Ch^{-3/q}$), Jensen's inequality
	$\|(\nabla\theta_h(\uu\otimes\uu))^\epsilon\|_{L^{q/2}}\le\|\nabla\theta_h(\uu\otimes\uu)\|_{L^{q/2}}
	\le\|\nabla\theta_h\|_{L^{q/(q-3)}}\|\uu\otimes\uu\|_{L^{q/2}}$, and the boundary strip estimate
	$\|\uu\|_{L^q(\Omega_{3h})}\le Ch^\beta\|\uu\|_{B_q^{\beta,\infty}([0,T]\times\Omega)}$, which follows from
	$\uu=0$ on $\partial\Omega$. This consumes no $\nu$-window beyond $\beta>\frac3q$, which is already implied by
	the lower bound on $\beta$ in Theorem \ref{renormalization}. Similarly we have 
	\begin{align*}
		|A_{12}|
		&\leq\int_{0}^{T}\int_\Omega \Big|\phi(t) [\uu\otimes\uu^\epsilon-(\uu\otimes\uu)^\epsilon]\nabla\theta_h (\theta_h\uu)^\epsilon\Big|
		dxdt\\
		&\leq C \|\phi\|_{C^0}\|\uu\|^2_{L^q(0,T;L^q(\Omega))}\frac1h\|\uu\|_{B^{\beta, \infty}_q}{\epsilon^\beta}h^{1-\frac3q}\to 0,
	\end{align*}
	as $\epsilon\to0$, provided 
	\begin{equation}
		\label{eq-qbeta-1}
		\nu<\frac q3\beta, \quad q>3.
	\end{equation}
	Note that 
	\begin{align*}
		A_{13}
		&=-\int_{0}^{T}\int_\Omega \phi(t)\uu\otimes\uu^\epsilon\nabla\theta_h (\theta_h\uu)^\epsilon dxdt\\
		&=-\int_{0}^{T}\int_\Omega\phi(t)\uu\otimes(\uu^\epsilon-\uu^\epsilon(\sigma(x)))\nabla\theta_h(\theta_h\uu^\epsilon) dxdt
		-\int_{0}^{T}\int_\Omega\phi(t)\uu\otimes\uu^\epsilon(\sigma(x))\nabla\theta_h(\theta_h\uu^\epsilon) dxdt\\
		&=:A_{131}+A_{132},
	\end{align*}
	where $\sigma(x)$ denotes a nearest point to $x$ on $\partial\Omega$. For $A_{131}$, we have
	\begin{align*}
		|A_{131}|
		&\leq\frac{C}{h}\|\phi\|_{C^0}\|\uu\|^2_{L^q(0,T;L^q(\Omega))}\|\uu^\epsilon-\uu^\epsilon(\sigma(x))\|_{L^q(0,T;L^q(\Omega))}h^\frac{q-3}{q}\\
		&\leq C \|\uu\|^2_{L^q(0,T;L^q(\Omega))}\|\nabla\uu^\epsilon\|_{L^q(0,T;L^q(\Omega))}h^\frac{q-3}{q}\\
		&\leq C \|\uu\|^2_{L^q(0,T;L^q(\Omega))}\|\uu\|_{B_q^{\beta,\infty}([0,T]\times\Omega)} h^\frac{q-3}{q}\epsilon^{\beta-1},
	\end{align*}
	where we used 
	\begin{eqnarray*}
		\|\uu^\epsilon(x)-\uu^\epsilon(\sigma(x))\|_{L^q(0,T;L^q(\Omega))}\leq C(q,\Omega) h \|\nabla\uu^\epsilon\|_{L^q(0,T;L^q(\Omega))}.
	\end{eqnarray*}
	For $A_{132}$, since $\uu=0$ on $\partial\Omega$ and $q\beta>1$, the scaled trace estimate
	\begin{align*}
		\|\uu^\epsilon\|_{L^q(0,T;L^q(\partial\Omega))}\leq C\epsilon^{\beta-\frac1q}\|\uu\|_{B_q^{\beta,\infty}([0,T]\times\Omega)}
	\end{align*}
	and a pull-back of the integral over the strip $\Omega_h:=\{x\in\Omega: h\leq d(x,\partial\Omega)\leq 2h\}$ to $\partial\Omega$ give
	\begin{align*}
		|A_{132}|
		&\leq \frac{C}{h}\|\phi\|_{C^0}\|\uu\|^2_{L^q(0,T;L^q(\Omega))}
		\|\uu^\epsilon(\sigma(x))\|_{L^q(0,T;L^q(\Omega_h))}\\
		&\leq \frac{C}{h}\|\uu\|^2_{L^q(0,T;L^q(\Omega))}\cdot h^{\frac1q}\|\uu^\epsilon\|_{L^q(0,T;L^q(\partial\Omega))}\\
		&\leq C\|\uu\|^2_{L^q(0,T;L^q(\Omega))}\|\uu\|_{B_q^{\beta,\infty}([0,T]\times\Omega)}\,
		\epsilon^{\beta-\frac1q}h^{-\frac{q-1}{q}}.
	\end{align*}
	After choosing 
	\begin{equation}\label{eq-qbeta-2}
		\frac{q(1-\beta)}{q-3}<\nu<\frac{q\beta-1}{q-1}, \quad q>3, \quad \beta>\frac3q
	\end{equation}
	(this interval is nonempty provided $\beta>\frac{q^2-3}{2q(q-2)}$, which is compatible with the lower bound on $\beta$
	assumed in Theorem \ref{renormalization}), we get
$$
		\lim_{\epsilon,h\to0}A_{13}=0.
$$
	
	For the second term $A_2$, we calculate it as follows
	\begin{align*}
		A_2
		&=-\int_{0}^{T}\int_\Omega \phi(t)\uu\otimes(\theta_h\uu)^\epsilon\nabla(\theta_h\uu)^\epsilon dxdt\\
		&\qquad \qquad+\int_{0}^{T}\int_\Omega
		\phi(t)(\uu\otimes(\theta_h\uu)^\epsilon-(\uu\otimes\uu\theta_h)^\epsilon)\nabla(\theta_h\uu)^\epsilon dxdt\\
		&=-\int_{0}^{T}\int_\Omega \phi(t)\uu\cdot\nabla\frac{1}{2}|(\theta_h\uu)^\epsilon|^2 dxdt\\
		&\qquad \qquad+\int_{0}^{T}\int_\Omega
		\phi(t)(\uu\otimes(\theta_h\uu)^\epsilon-(\uu\otimes\uu\theta_h)^\epsilon)\nabla(\theta_h\uu)^\epsilon dxdt\\
		&=\int_{0}^{T}\int_\Omega \phi(t)(\uu\otimes(\theta_h\uu)^\epsilon-(\uu\otimes\uu\theta_h)^\epsilon)\nabla(\theta_h\uu)^\epsilon dxdt
		:=A_{21}.
	\end{align*}
	where we used the fact that $\mathrm{div}\uu=0$ in the distribution sense. Note that $q>n+2\geq4$, so that
	$L^{\frac q2}(\Omega)\hookrightarrow L^2(\Omega)$. By the H\"older inequality and the first estimate of Lemma
	\ref{le-communicate} (with $f=\uu$ and $g=(\theta_h\uu)^\epsilon$), we have
	\begin{align*}
		|A_{21}|
		&\leq \|\phi\|_{C^0}\big\|\uu\otimes(\theta_h\uu)^\epsilon-(\uu\otimes\theta_h\uu)^\epsilon\big\|_{L^{\frac q2}(0,T;L^{\frac q2}(\Omega))}
		\big\|\nabla(\theta_h\uu)^\epsilon\big\|_{L^2(0,T;L^2(\Omega))}\\
		&\leq C\epsilon^{\beta}\|\uu\|_{B_q^{\beta,\infty}([0,T]\times\Omega)}\big\|\nabla(\theta_h\uu)^\epsilon\big\|_{L^2(0,T;L^2(\Omega))}.
	\end{align*}
	Since
	\begin{align*}
		\big\|\nabla(\theta_h\uu)^\epsilon\big\|_{L^2(0,T;L^2(\Omega))}
		&\leq \|(\theta_h\nabla\uu)^\epsilon\|_{L^2(0,T;L^2(\Omega))}+\|(\uu\nabla\theta_h)^\epsilon\|_{L^2(0,T;L^2(\Omega))}\\
		&\leq \|\nabla\uu\|_{L^2(0,T;L^2(\Omega))}+\frac{C}{h}\|\uu\|_{L^q(0,T;L^q(\Omega))}h^{\frac12-\frac1q}
		\leq C\big(1+h^{-\frac12-\frac1q}\big),
	\end{align*}
	we infer
	\begin{align*}
		|A_{21}|\leq C\epsilon^{\beta}\big(1+h^{-\frac12-\frac1q}\big)\to0,
	\end{align*}
	provided that
	\begin{equation}\label{eq-qbeta-3}
		\nu<\frac{2q\beta}{q+2}.
	\end{equation}
	Thus, we obtain 
$
		\lim_{\epsilon,h\to0}A_{21}=0.
$
We can also conclude that $\lim_{\epsilon,h\to0}I_{2}=0$ if \eqref{eq-qbeta-1}--\eqref{eq-qbeta-3} hold (recall that \eqref{eq-qbeta-1} is automatic, since $\beta>\frac3q$ and $0<\nu<1$).

	\textbf{Estimate of $I_3$} For the viscous term $I_3$, we have 
	\begin{align*}
	I_3	&=\mu\int_{0}^{T}\int_\Omega\phi(t)(\nabla\uu\cdot\nabla\theta_h)^\epsilon(\theta_h\uu)^\epsilon dxdt
		+\mu\int_{0}^{T}\int_\Omega\phi(t)(\nabla\uu\theta_h)^\epsilon\nabla(\theta_h\uu)^\epsilon dxdt\\
		&:=B_1+B_2.
	\end{align*}
	We decompose $B_1$ as follows
	\begin{align*}
		B_1
		&:=\mu\int_{0}^{T}\int_\Omega\phi(t)(\nabla\uu\cdot\nabla\theta_h)^\epsilon(\theta_h\uu)^\epsilon dxdt\\
		&=\mu\int_{0}^{T}\int_\Omega \phi(t)((\nabla\uu\cdot\nabla\theta_h)^\epsilon-\nabla\uu^\epsilon\cdot\nabla\theta)(\theta_h\uu)^\epsilon
		dxdt\\
		&\quad+\mu\int_{0}^{T}\int_\Omega\phi(t)\nabla\uu^\epsilon\cdot\nabla\theta((\theta_h\uu)^\epsilon-\theta_h\uu^\epsilon)dxdt\\
		&\quad+\mu\int_{0}^{T}\int_\Omega\phi(t) \nabla\uu^\epsilon\cdot\nabla\theta_h(\theta_h\uu^\epsilon)dxdt\\
		&=:B_{11}+B_{12}+B_{13}.    
	\end{align*}
	Note that $\nabla\uu$ is bounded in $L^2(0,T;L^2(\Omega))$, one obtains
	\begin{align*}
		|B_{11}|
		&\leq\mu\int_{0}^{T}\int_\Omega
		\Big|\phi(t)((\nabla\uu\cdot\nabla\theta_h)^\epsilon-\nabla\uu^\epsilon\cdot\nabla\theta)(\theta_h\uu)^\epsilon\Big|dxdt \\
		&\leq C\|\phi\|_{C^0}\|\nabla\uu\|_{L^2(0,T;L^2(\Omega))}\|\uu\|_{L^q(0,T;L^q(\Omega))}\frac{\epsilon}{h^{\frac{1}{2}+\frac{1}{q}}},
	\end{align*}
	and control $B_{12}$ as follows 
	\begin{align*}
		|B_{12}|
		&\leq\mu\int_{0}^{T}\int_\Omega \Big|\phi(t)\nabla\uu^\epsilon\cdot\nabla\theta((\theta_h\uu)^\epsilon-\theta_h\uu^\epsilon)\Big|dxdt \\
		&\leq C\|\phi\|_{C^0}\|\nabla\uu\|_{L^2(0,T;L^2(\Omega))}\|\uu\|_{L^q(0,T;L^q(\Omega))}\frac{\epsilon}{h^{\frac{1}{2}+\frac{1}{q}}}.
	\end{align*}
	We give a control on $B_{13}$. Splitting
	$\theta_h\uu^\epsilon=(\uu^\epsilon(x)-\uu^\epsilon(\sigma(x)))\theta_h+\uu^\epsilon(\sigma(x))\theta_h$
	with $\sigma(x)$ a nearest point to $x$ on $\partial\Omega$, we have $B_{13}\leq B_{131}+B_{132}$, where
	\begin{align*}
		B_{131}
		&:=\mu\int_{0}^{T}\int_\Omega\Big|\phi(t)\nabla\uu^\epsilon\cdot\nabla\theta_h(\uu^\epsilon(x)-\uu^\epsilon(\sigma(x)))\theta_h\Big| dxdt\\
		&\leq
		\frac{C\mu}{h}\|\phi\|_{C^0}\|\nabla\uu\|_{L^2(0,T;L^2(\Omega))}\|\uu^\epsilon(x)-\uu^\epsilon(\sigma(x))\|_{L^q(0,T;L^q(\Omega))}h^\frac{q-2}{2q}\\
		&\leq C\mu\|\nabla\uu\|_{L^2(0,T;L^2(\Omega))}\|\uu\|_{B_q^{\beta,\infty}([0,T]\times\Omega)}h^\frac{q-2}{2q}\epsilon^{\beta-1},
	\end{align*}
	and, by the scaled trace estimate
	$\|\uu^\epsilon\|_{L^q(0,T;L^q(\partial\Omega))}\leq C\epsilon^{\beta-\frac1q}\|\uu\|_{B_q^{\beta,\infty}([0,T]\times\Omega)}$,
	\begin{align*}
		B_{132}
		&:=\mu\int_{0}^{T}\int_\Omega\Big|\phi(t)\nabla\uu^\epsilon\cdot\nabla\theta_h\,\uu^\epsilon(\sigma(x))\theta_h\Big|dxdt\\
		&\leq \frac{C\mu}{h}\|\phi\|_{C^0}\|\nabla\uu\|_{L^2(0,T;L^2(\Omega))}
		h^{\frac12-\frac1q}\|\uu^\epsilon(\sigma(x))\|_{L^q(0,T;L^q(\Omega_h))}\\
		&\leq C\mu\|\nabla\uu\|_{L^2(0,T;L^2(\Omega))}\|\uu\|_{B_q^{\beta,\infty}([0,T]\times\Omega)}\,
		h^{-\frac12}\epsilon^{\beta-\frac1q}.
	\end{align*}
	If 
	\begin{equation}\label{eq-qbeta-4}
		\frac{2q(1-\beta)}{q-2}<\nu<\min\Big\{2\Big(\beta-\frac1q\Big),\frac{2q}{q+2}\Big\}, \quad q>2,\quad \beta>\frac{1}{2}+\frac{1}{q}
	\end{equation}
	(the first upper bound comes from $B_{132}$, the second one from $B_{11}$ and $B_{12}$; note that
	$2(\beta-\frac1q)>1$ and $\frac{2q(1-\beta)}{q-2}<\frac{2q}{q+2}$ under $\beta>\frac12+\frac1q$, so the interval is
	nonempty), then we obtain $B_{1}\to0$ as $(\epsilon,h)\to 0$.

	For the term $B_2$, we have 
	\begin{align*}
		B_2
		&:=\mu\int_{0}^{T}\int_\Omega\phi(t)(\nabla\uu\theta_h)^\epsilon\nabla(\theta_h\uu)^\epsilon dxdt\notag\\
		&=\mu\int_{0}^{T}\int_\Omega\phi(t)((\theta_h\nabla\uu)^\epsilon-(\theta_h\nabla\uu^\epsilon))\nabla(\theta_h\uu)^\epsilon dxdt\\
		&\quad+\mu\int_{0}^{T}\int_\Omega\phi(t)(\theta_h\nabla\uu^\epsilon)\cdot(\nabla(\theta_h\uu)^\epsilon-\nabla(\theta_h\uu^\epsilon))dxdt\\
		&\quad+\mu\int_{0}^{T}\int_\Omega\phi(t)(\theta_h\nabla\uu^\epsilon)\cdot\nabla(\theta_h\uu^\epsilon)dxdt\\
		&=:B_{21}+B_{22}+B_{23}
	\end{align*}
	We have 
	\begin{align*}
		|B_{21}|
		&\leq \mu\int_{0}^{T}\int_\Omega\Big|\phi(t)((\theta_h\nabla\uu)^\epsilon-\theta_h\nabla\uu^\epsilon)\nabla(\theta_h\uu)^\epsilon\Big|
		dxdt\\
		&\leq C\mu\|\phi\|_{C^0}\big\|(\theta_h\nabla\uu)^\epsilon-\theta_h\nabla\uu^\epsilon\big\|_{L^2(0,T;L^2(\Omega))}
		\big\|\nabla(\theta_h\uu)^\epsilon\big\|_{L^2(0,T;L^2(\Omega))}.
	\end{align*}
	Since $\theta_h\nabla\uu^\epsilon=\theta_h(\nabla\uu)^\epsilon$, the first factor is a commutator of $\theta_h$ and $\nabla\uu$: writing
	\[
	(\theta_h\nabla\uu)^\epsilon-\theta_h(\nabla\uu)^\epsilon
	=\int_{\RR}m_\epsilon(y)\big(\theta_h(x-y)-\theta_h(x)\big)\nabla\uu(x-y)\,dy
	\]
	and using $|\theta_h(x-y)-\theta_h(x)|\leq C\frac{|y|}{h}$, we get
	\begin{align*}
		&\big\|(\theta_h\nabla\uu)^\epsilon-\theta_h\nabla\uu^\epsilon\big\|_{L^2(0,T;L^2(\Omega))}\\
		&\quad\leq
		C\frac{\epsilon}{h}\|\nabla\uu\|_{L^2(0,T;L^2(\Omega))}.
	\end{align*}
	For the second factor, $\nabla(\theta_h\uu)^\epsilon=(\theta_h\nabla\uu)^\epsilon+(\uu\nabla\theta_h)^\epsilon$ with
	$\|(\theta_h\nabla\uu)^\epsilon\|_{L^2(0,T;L^2(\Omega))}\leq \|\nabla\uu\|_{L^2(0,T;L^2(\Omega))}$ and
	\begin{align*}
		\|(\uu\nabla\theta_h)^\epsilon\|_{L^2(0,T;L^2(\Omega))}
		\leq \frac{C}{h}\|\uu\|_{L^q(0,T;L^q(\Omega_{3h}))}h^{\frac12-\frac1q}\notag\\
		\leq Ch^{\beta-\frac12-\frac1q}\|\uu\|_{B_q^{\beta,\infty}([0,T]\times\Omega)},
	\end{align*}
	where we used $\|\uu\|_{L^q(\Omega_{3h})}=\|\uu-\uu(\sigma(\cdot))\|_{L^q(\Omega_{3h})}
	\leq Ch^\beta\|\uu\|_{B_q^{\beta,\infty}([0,T]\times\Omega)}$, which follows from $\uu=0$ on $\partial\Omega$ and
	$|x-\sigma(x)|\leq 3h$ on $\Omega_{3h}$. Combining the above, we conclude
	\begin{align*}
		|B_{21}|\leq C\frac{\epsilon}{h}\big(1+h^{\beta-\frac12-\frac1q}\big)\to0
	\end{align*}
	provided
	\begin{equation}\label{eq-qbeta-5}
		\nu<\frac{2q}{3q+2-2q\beta}.
	\end{equation}
	So  $B_{21}\to0$ as $(\epsilon,h)\to0$. In the same way, we have  $B_{22}\to0$ as $(\epsilon,h)\to0$.
	The last term of $B_2$ is bounded as follows
	\begin{align*}
		B_{23}&:=\mu\int_{0}^{T}\int_\Omega\phi(t)(\theta_h\nabla\uu^\epsilon)\cdot\nabla(\theta_h\uu^\epsilon)dxdt\\
		&=\mu\int_{0}^{T}\int_\Omega\phi(t)|\theta_h\nabla\uu^\epsilon|^2dxdt
		+\mu\int_{0}^{T}\int_\Omega\phi(t)(\theta_h\nabla\uu^\epsilon):\uu^\epsilon\nabla\theta_h dxdt.
	\end{align*}
	Since $\nabla\uu$ is bounded in $L^2(0,T;L^2(\Omega))$, we have 
	\begin{align*}
		&\lim_{\epsilon,h\to0}\mu\int_{0}^{T}\int_\Omega\phi(t)|\theta_h\nabla\uu^\epsilon|^2dxdt\\
		=&\lim_{\epsilon,h\to0}\mu\int_{0}^{T}\int_{d(x)\geq2h}\phi(t)|\theta_h\nabla\uu^\epsilon|^2dxdt+\lim_{\epsilon,h\to0}\mu\int_{0}^{T}\int_{h\leq d(x)\leq2h}\phi(t)|\theta_h\nabla\uu^\epsilon|^2dxdt\\
		=&\mu\int_{0}^{T}\int_{\Omega}\phi(t)|\nabla\uu|^2dxdt+\lim_{\epsilon,h\to0}\mu\int_{0}^{T}\int_{h\leq
		d(x)\leq2h}\phi(t)|\theta_h\nabla\uu^\epsilon|^2dxdt.
	\end{align*}
	Besides, splitting $\uu^\epsilon=(\uu^\epsilon(x)-\uu^\epsilon(\sigma(x)))+\uu^\epsilon(\sigma(x))$ as before,
	\begin{align*}
		&\Big|\mu\int_{0}^{T}\int_\Omega\phi(t)(\theta_h\nabla\uu^\epsilon):\uu^\epsilon\nabla\theta_h dxdt\Big|\\
		\leq& \frac{C\mu}{h}\Big|\int_{0}^{T}\int_\Omega\phi(t)(\theta_h\nabla\uu^\epsilon)(\uu^\epsilon(x)-\uu^\epsilon(\sigma(x))) dxdt\Big|
		+\frac{C\mu}{h}\Big|\int_{0}^{T}\int_\Omega\phi(t)(\theta_h\nabla\uu^\epsilon)\,\uu^\epsilon(\sigma(x))\, dxdt\Big|\\
		\leq &C\mu\epsilon^{\beta-1}\|\phi\|_{C^0}\|\uu\|_{B_q^{\beta,\infty}}\|\nabla\uu\|_{L^2(0,T;L^2(\Omega))} h^\frac{q-2}{2q}
		+C\mu\epsilon^{\beta-\frac1q}h^{-\frac12}\|\phi\|_{C^0}\|\uu\|_{B_q^{\beta,\infty}}\|\nabla\uu\|_{L^2(0,T;L^2(\Omega))}
		\to0,
	\end{align*}
	as $(\epsilon,h)\to0$ under \eqref{eq-qbeta-4}. Moreover, by Jensen's inequality
	$|\nabla\uu^\epsilon|^2\leq(|\nabla\uu|^2)^\epsilon$, hence
	\begin{align*}
		\int_{0}^{T}\int_{h\leq d(x)\leq2h}\phi(t)|\theta_h\nabla\uu^\epsilon|^2dxdt
		\leq \|\phi\|_{C^0}\int_{0}^{T}\int_{d(x)\leq3h}|\nabla\uu|^2dxdt\longrightarrow0,
	\end{align*}
	since $|\nabla\uu|^2\in L^1((0,T)\times\Omega)$ and the measure of the strip tends to zero.
	Therefore, once \eqref{eq-qbeta-4} and \eqref{eq-qbeta-5} hold we obtain 
	\begin{eqnarray}\label{t3}
		\lim_{\epsilon,h\to0}I_3=\mu\int_{0}^{T}\int_{\Omega}\phi(t)|\nabla\uu|^2dxdt.
	\end{eqnarray}

	\textbf{Estimate of $I_4$ } 
	Now we are in a position to bound the pressure terms.
	\begin{lemma}\label{lem:pressure-elliptic}
Let $\Omega\subset\mathbb{R}^n$ be a bounded domain with smooth boundary.
Assume that the pressure $P$ solves
\begin{align}
    -\Delta P
    &= \sum_{i,j=1}^n \partial_{x_i}\partial_{x_j}(u_i u_j)
        + \nabla\cdot \int_{\mathbb{R}^n} (\uu-\vv) f \, d\vv
    \quad \text{in } \Omega, \label{P-elliptic-new}\\
    P &= 0 \quad \text{on } \partial\Omega. \notag
\end{align}
Then for any $1<r<\infty$, there exists a constant $C>0$ depending only on $r$ and $\Omega$ such that
\begin{align}\label{P-W2r}
    \|P\|_{L^{r}(\Omega)}
    \le C \Big(
        \|(\uu\otimes \uu)\|_{L^r(\Omega)}
        + \|\uu\,\bar\rho\|_{L^r(\Omega)}
        + \|j\|_{L^r(\Omega)}
    \Big),
\end{align}
\end{lemma}
\begin{proof}
	Let $P_1$ solve
	\[ -\Delta P_1=\mathrm{div}\,\mathrm{div}(\uu\otimes \uu),\quad P_1|_{\partial\Omega}=0\]
	and $P_2$ solve
	\[-\Delta P_2=\nabla\cdot \int_{\mathbb{R}^n} (\uu-\vv) f \, d\vv, \quad P_2|_{\partial\Omega}=0.\]
	Then by the $L^p$ theory of the Poisson equation with force of divergence form, we immediately have
	\begin{align}
		&\|P_1\|_{L^{r}(\Omega)}
		\le C 
		\|(\uu\otimes \uu)\|_{L^r(\Omega)}  \label{eq-p1-bound}\\
		&\|\nabla P_2\|_{L^{r}(\Omega)}
		\le C (
		\|\uu\,\bar\rho\|_{L^r(\Omega)}
		+ \|j\|_{L^r(\Omega)})\label{eq-p2-bound}
	\end{align}
for any $1<r<\infty$.
It remains to bound $P_2$ itself. Since $P_2|_{\partial\Omega}=0$, the Poincar\'e inequality applies and gives
$\|P_2\|_{L^r(\Omega)}\le C\|\nabla P_2\|_{L^r(\Omega)}$ for every $1<r<\infty$; when $r<n$, one may alternatively
invoke the Sobolev embedding $W^{1,r}_0(\Omega)\hookrightarrow L^{\frac{nr}{n-r}}(\Omega)$, which yields the same
conclusion since $\frac{nr}{n-r}>r$ on the bounded domain $\Omega$. Combining \eqref{eq-p1-bound} and
\eqref{eq-p2-bound} with the estimate for $P_2$, we obtain \eqref{P-W2r}.
\end{proof}
\begin{remark}
	Under the assumptions $\mathbf{u}\in B_q^{\beta,\infty}([0,T]\times\Omega)$ and $f\in
	L^p(0,T;B_p^{\alpha,\infty}(\Omega\times\RR))$, the right-hand side of \eqref{P-W2r} is integrable in time,
	so the pressure integrability required in Theorem \ref{renormalization} is consistent with the regularity of
	$\mathbf{u}$ and $f$.
\end{remark}
\begin{remark}\label{re-pressure-normalization}
	A word on the boundary condition in Lemma \ref{lem:pressure-elliptic}. Taking the divergence of the momentum
	equation yields a Poisson equation for the pressure but does not by itself determine its boundary values; the
	no-slip condition formally leads to a Neumann-type boundary relation whose interpretation is delicate at the
	regularity of weak solutions considered here. For this reason the required integrability
	$P\in L^1(0,T;L^r(\Omega))$ with $\frac1r=\frac1q+\frac1{p_1}$ is listed explicitly as an assumption in
	Theorem \ref{renormalization}, while Lemma \ref{lem:pressure-elliptic} uses the solution of the corresponding
	Poisson problem with the Dirichlet condition $P|_{\partial\Omega}=0$---the same strategy as in Yu's treatment
	of the energy equality for the single Navier--Stokes equations \cite{yu2018}. The relation between this
	Dirichlet representative and the pressure determined by the weak momentum equation, and whether the required
	$L^r$ bound can be derived from the physical boundary relation alone, are left open.
\end{remark}



	Hence, we consider the following decomposition 
	\begin{align*}
		\int_{\Omega}\,(\nabla P\theta_h)^\epsilon\cdot(\theta_h\uu)^\epsilon dx
		&=-\int_{\Omega}\,
		(P\,\nabla\theta_h)^\epsilon\cdot(\theta_h\uu)^\epsilon\,dx-\int_\Omega (P\theta_h)^\epsilon \mathrm{div}(\theta_h\uu)^\epsilon dx\\
		&=D_{1}+D_2.
	\end{align*}
	
	Term $D_2$ can be controlled
	\begin{align*}
		D_2
		&=-\int_\Omega (P\theta_h)^\epsilon \mathrm{div}(\theta_h\uu)^\epsilon dx\\
		&=\int_\Omega (P\theta_h)^\epsilon (\mathrm{div}(\theta_h\uu^\epsilon)-\mathrm{div}(\theta_h\uu)^\epsilon) dx
		-\int_\Omega (P\theta_h)^\epsilon \mathrm{div}(\theta_h\uu^\epsilon) dx\\
		&=D_{21}+D_{22}.
	\end{align*}
	We focus on term $D_{22}$ in the following ones 
	\begin{align*}
		|D_{22}|
		&=\Big|-\int_\Omega (P\theta_h)^\epsilon \mathrm{div}(\theta_h\uu^\epsilon) dx\Big|\\
		&\leq\Big|\int_\Omega (P\theta_h)^\epsilon \nabla\theta_h \uu^\epsilon dx\Big| +\Big|\int_\Omega (P\theta_h)^\epsilon \theta_h
		\mathrm{div}\uu^\epsilon dx\Big|\\
		&=\Big|\!\!\int_{\Omega}\, (P\theta_h)^\epsilon\,\nabla\theta_h\cdot
		\big(\uu^\epsilon(x)-\uu^\epsilon(\sigma(x))\big)\,dx\Big|
		+\Big|\!\!\int_{\Omega}\, (P\theta_h)^\epsilon\,\nabla\theta_h\cdot\uu^\epsilon(\sigma(x))\,dx\Big|\\
		&\leq 
		\frac{C}{h}\|P\|_{L^r(\Omega)}
		\|\nabla\uu^\epsilon\|_{L^q(0,T;L^q(\Omega))}\,
		h^{1-\frac{1}{q}-\frac{1}{r}}
		+\frac{C}{h}\|P\|_{L^r(\Omega)}
		\|\uu^\epsilon(\sigma(\cdot))\|_{L^q(0,T;L^q(\Omega_h))}\,
		h^{1-\frac{1}{q}-\frac{1}{r}}\\
		&\leq 
		C\,\|P\|_{L^r(\Omega)}
		\|\uu\|_{B_q^{\beta,\infty}([0,T]\times\Omega)}
		\Big(h^{1-\frac{1}{r}-\frac{1}{q}}\epsilon^{\beta-1}+\epsilon^{\beta-\frac1q}h^{-\frac1r}\Big),
	\end{align*}
	where we used $\mathrm{div}\uu^\epsilon=(\mathrm{div}\uu)^\epsilon=0$ on the support of $\theta_h$ and the scaled trace
	estimate
	\[
	\|\uu^\epsilon(\sigma(\cdot))\|_{L^q(0,T;L^q(\Omega_h))}
	\leq Ch^{\frac1q}\|\uu^\epsilon\|_{L^q(0,T;L^q(\partial\Omega))}
	\leq Ch^{\frac1q}\epsilon^{\beta-\frac1q}\|\uu\|_{B_q^{\beta,\infty}([0,T]\times\Omega)}.
	\]
	The first term on the right-hand side vanishes provided $\nu>\frac{1-\beta}{1-\frac1r-\frac1q}$, and the second one
	vanishes provided $\nu<r\big(\beta-\frac1q\big)$, the latter being automatic since $\frac1r<\beta-\frac1q$ and $\nu<1$.

	
	Hence, to guarantee the vanishing of this term as \(\epsilon,h\to0\), we require
	\[
	1-\frac{1}{r}-\frac{1}{q} > 1-\beta
	\quad\text{i.e.}\quad 
	\frac{1}{r}<\beta-\frac{1}{q}<1-\frac{1}{q}.
	\]
	
	For the nonlinear term \(\uu\cdot\bar{\rho}\), we have
	\[
	\|\uu\cdot\bar{\rho}\|_{L^r(\Omega)}
	\leq \|\uu\|_{L^q(\Omega)}\|\bar{\rho}\|_{L^{p_1}(\Omega)},
	\qquad \frac{1}{r}=\frac{1}{q}+\frac{1}{p_1}.
	\]
	Thus,
	\[
	\frac{1}{p_1}<1-\frac{2}{q}
	\quad\text{i.e.}\quad 
	p_1>1+\frac{2}{q-2}.
	\]
	
	Applying Lemma~\ref{le-f-estimate}, we obtain
	\[
	\frac{n(p-1)+2p}{n(p-1)+2}
	>1+\frac{2}{q-2}
	\quad\text{i.e.}\quad
	(p-1)(q-n-2)>2.
	\]
	
	Similarly, for the term involving \(j\), by the embedding theorem in a bounded domain, we require
	\[
	p_2>r 
	\quad\text{i.e.}\quad 
	\frac{1}{p_2}<\frac{1}{r}<1-\frac{1}{q},
	\]
	which holds provided that
	\[
	\frac{n(p-1)+2p}{(n-1)(p-1)+2p}
	>1+\frac{1}{q-1}
	\quad\text{i.e.}\quad
	(p-1)(q-n-2)>2,
	\]
	the same condition as for \(\bar\rho\).
	
	Finally, for the first term in \eqref{P-elliptic-new}, we need
	\[
	\frac{q}{2}>r 
	\quad\Longrightarrow\quad 
	\frac{2}{q}<1-\frac{1}{q}
	\quad\text{i.e.}\quad 
	q>3.
	\]
	all the above estimates are valid provided the parameters satisfy
	\begin{align*}
		&p>1,\qquad q>n+2,\qquad (p-1)(q-n-2)>2.
	\end{align*}
	Moreover, since 
	\(
	\frac{1}{r}>\frac{1}{p_2}
	\)
	the regularity exponent \(\beta\) must satisfy
	\[
	\beta>\frac{1}{q}+\frac{1}{r}
	>\frac{1}{q}+\frac{1}{p_2}
	>\frac{1}{q}+\frac{(n-1)(p-1)+2p}{n(p-1)+2p}.
	\]
	For term $D_{21}$, we give a control 
	\begin{align*}
		D_{21}
		&=\int_\Omega (P\theta_h)^\epsilon (\mathrm{div}(\theta_h\uu^\epsilon)-\mathrm{div}(\theta_h\uu)^\epsilon) dx\\
		&\leq C\|P\|_{L^r(\Omega)}\|\uu\|_{B_q^{\beta,\infty}}\epsilon^{\beta-1}h^{1-\frac{1}{r}-\frac{1}{q}}\to0,
	\end{align*}
	as $(\epsilon,h)\to0$, where we used integration by parts and the paired form of the commutator estimate in Lemma
	\ref{le-communicate} together with $\frac1q\leq\frac1r$.
	
	Next, we estimate the term $D_1$
	\begin{align*}
		|D_1| 
		&=\Big|-\int_{\Omega}\, (P\,\nabla\theta_h)^\epsilon\cdot(\theta_h\uu)^\epsilon\,dx\Big|\\
		&=\Big|\int_{h\leq d\leq 2h}(P\nabla\theta_h)^\epsilon(\theta_h\uu)^\epsilon dx\Big|\\
		&\leq \frac{C}{h}\|P\|_{L^r(\Omega)}\|\uu\|_{L^q(\Omega_{3h})}h^{1-\frac{1}{r}-\frac{1}{q}}\\
		&\leq C\|P\|_{L^r(\Omega)}\|\uu\|_{B_q^{\beta,\infty}([0,T]\times\Omega)}h^{\beta-\frac{1}{r}-\frac{1}{q}},
	\end{align*}
	where we used $|\nabla\theta_h|\leq\frac Ch$ and the boundary strip estimate
	$\|\uu\|_{L^q(\Omega_{3h})}=\|\uu-\uu(\sigma(\cdot))\|_{L^q(\Omega_{3h})}\leq Ch^\beta\|\uu\|_{B_q^{\beta,\infty}([0,T]\times\Omega)}$.
	This tends to zero because $\beta>\frac1q+\frac1r$.

	Therefore, we obtain 
	\begin{eqnarray}\label{t4}
		\lim_{\epsilon,h\to0}\int_{0}^{T}\int_{\Omega}\phi(t)\,(\nabla P\theta_h)^\epsilon\cdot(\theta_h\uu)^\epsilon dxdt= 0.
	\end{eqnarray}

	\textbf{Estimate of $I_5$.} For the drag force term, we compute
	\begin{align*}
		I_5:=&\int_{0}^{T}\int_\Omega\phi(t)(\theta_h\uu)^\epsilon\cdot\left(\theta_h\int_{\RR}(\uu-\vv)f\,d\vv\right)^\epsilon dxdt\\
		=&\int_{0}^{T}\int_\Omega\phi(t)(\theta_h\uu)^\epsilon\cdot\big[(\theta_h\uu\bar{\rho})^\epsilon-(\theta_h j)^\epsilon\big]\,dxdt
		=:E_1-E_2.
	\end{align*}
	Decompose \(E_1\) as
	\begin{align*}
		E_1
		&=\int_{0}^{T}\int_\Omega
		\phi(t)(\theta_h\uu)^\epsilon\cdot\big[(\theta_h\uu\bar{\rho})^\epsilon-(\theta_h\uu)^\epsilon\bar{\rho}^\epsilon\big]\,dxdt
		+\int_{0}^{T}\int_\Omega\phi(t)|(\theta_h\uu)^\epsilon|^2\bar{\rho}^\epsilon\,dxdt\\
		&=E_{11}+E_{12}.
	\end{align*}
	The commutator term \(E_{11}\) can be estimated by
	\[
	\begin{aligned}
		|E_{11}|
		&\leq \Big|\int_{0}^{T}\int_\Omega
		\phi(t)(\theta_h\uu)^\epsilon\cdot((\theta_h\uu)^\epsilon-(\theta_h\uu))(\bar{\rho}^\epsilon-\bar{\rho})\,dxdt\Big|\\
		&\qquad +\Big|\int_{0}^{T}\int_\Omega \phi(t)(\theta_h\uu)^\epsilon\cdot\omega_\epsilon(\theta_h\uu,\bar{\rho})\,dxdt\Big|
	\end{aligned}
	\]
	where
	\[
	\omega_\epsilon(f,g)(x)
	=\int_{\Omega}m_\epsilon(y)\,[g(x-y)-g(x)]\,[f(x-y)-f(x)]\,dy.
	\]
	Note that 
	\[\|\uu\|_{L^q(0,T;L^q(\Omega))}\leq C,\quad \|\bar{\rho}\|_{L^\infty(0,T; L^{p_1}(\Omega))}\leq C.\, \]
If
\[
\frac{1}{p_1}+\frac{2}{q}\le1\quad\Longleftrightarrow\quad (p-1)(q-n-2)\ge2,
\]
then H\"older's inequality with exponents $(q,q,p_1)$ applies to both terms on the right-hand side of the estimate of $E_{11}$: the first
one is bounded by
\[
C\|(\theta_h\uu)^\epsilon\|_{L^q(0,T;L^q(\Omega))}\,
\|(\theta_h\uu)^\epsilon-\theta_h\uu\|_{L^q(0,T;L^q(\Omega))}\,
\|\bar{\rho}^\epsilon-\bar{\rho}\|_{L^{p_1}((0,T)\times\Omega)}
\leq C\epsilon^{\beta}\|\uu\|_{B_q^{\beta,\infty}}\,o(1)\longrightarrow0,
\]
where the Besov regularity of $\uu$ provides the rate $\epsilon^\beta$,
while $\bar\rho\in L^{p_1}((0,T)\times\Omega)$ with $p_1<\infty$ implies $\|\bar{\rho}^\epsilon-\bar{\rho}\|_{L^{p_1}}\to0$ by the
continuity of translations (no rate is needed here); the commutator term is estimated in the same way,
\[
\Big|\int_{0}^{T}\int_\Omega \phi(t)(\theta_h\uu)^\epsilon\cdot\omega_\epsilon(\theta_h\uu,\bar{\rho})\,dxdt\Big|
\le C\epsilon^{\beta}\|\uu\|_{B_q^{\beta,\infty}}
\Big(\sup_{|y|\le2\epsilon}\|\bar{\rho}(\cdot-y)-\bar{\rho}\|_{L^{p_1}((0,T)\times\Omega)}\Big)\longrightarrow0.
\]
Hence $E_{11}\to0$ as $(\epsilon,h)\to0$.
For the term \(E_{12}\) we split the domain into interior and boundary layers
	\[
	\begin{aligned}
		E_{12}
		&=\int_{0}^{T}\int_{d(x,\partial\Omega)\ge2h}\phi(t)|\uu^\epsilon|^2\bar{\rho}^\epsilon\,dxdt+\int_{0}^{T}\int_{h\le
		d(x,\partial\Omega)\le2h}\phi(t)|(\theta_h\uu)^\epsilon|^2\bar{\rho}^\epsilon\,dxdt.
	\end{aligned}
	\]
	Using the same argument as for term $D_{22}$, we conclude
	\[
	\lim_{\epsilon,h\to0}
	\int_{0}^{T}\int_\Omega \phi(t)(\theta_h\uu)^\epsilon\cdot(\theta_h\uu\bar{\rho})^\epsilon\,dxdt
	=\int_{0}^{T}\int_\Omega \phi(t)\uu\cdot(\uu\bar{\rho})\,dxdt.
	\]
	
	The term \(E_2\) is handled analogously. Under the condition
	\[
	\frac{1}{p_2}+\frac{1}{q}\le1\quad\Longleftrightarrow\quad
	(p-1)(q-n-2)>2,
	\]
	one obtains
	\[
	\lim_{\epsilon,h\to0}
	\int_{0}^{T}\int_\Omega \phi(t)(\theta_h\uu)^\epsilon\cdot (\theta_hj)^\epsilon\,dxdt
	=\int_{0}^{T}\int_\Omega \phi(t)\uu\cdot j\,dxdt.
	\]
	Therefore
	\begin{equation}\label{t5_refined}
		\lim_{\epsilon,h\to0}
		\int_{0}^{T}\int_\Omega\phi(t)(\theta_h\uu)^\epsilon\cdot\left(\theta_h\int_{\RR}(\uu-\vv)f\,d\vv\right)^\epsilon dxdt
		=\int_{0}^{T}\int_\Omega\phi(t)\uu\cdot\left(\int_{\RR}(\uu-\vv)f\,d\vv\right)dxdt.
	\end{equation}

	Finally, we estimate the contribution coming from the Vlasov equation.
Let $\chi_R \in C_c^\infty(\RR)$ be a standard cut-off function satisfying
\begin{eqnarray*}
	\chi_R(\vv)=
	\begin{cases}
		1,\quad\text{for } |\vv|\le R,\\
		0,\quad\text{for } |\vv|\ge 2R,
	\end{cases}
	\quad|\nabla_\vv \chi_R|\leq \frac{C}{R}.
\end{eqnarray*}

We take the admissible test function
\begin{align*}
    \Phi_{h,R}(x,\vv,t)=\phi(t)\theta_h(x)\Big(1+\frac{|\vv|^2}{2}\Big)\chi_R(\vv).
\end{align*}
We plug $\Phi_{h,R}$ into the weak formulation of the Vlasov equation:
\begin{align*}
\int_0^T\int_{\Omega}\int_{\RR}
    f\big(\partial_t\Phi_{h,R} + \vv\cdot\nabla_x\Phi_{h,R}
    + ( \uu - \vv )\cdot\nabla_\vv \Phi_{h,R} \big)\,d\vv\,dxdt = 0 .
\end{align*}

A direct computation gives
\begin{align*}
    \partial_t\Phi_{h,R} &= \phi'(t)\theta_h(x)\Big(1+\frac{|\vv|^2}{2}\Big)\chi_R(\vv),\\
	\nabla_x\Phi_{h,R}   &=\phi(t)\nabla_x\theta_h(x)\Big(1+\frac{|\vv|^2}{2}\Big)\chi_R(\vv),\\
    \nabla_\vv\Phi_{h,R} &= \phi(t)\theta_h(x)\Big(\vv\chi_R(\vv)
    +\Big(1+\frac{|\vv|^2}{2}\Big)\nabla_\vv\chi_R(\vv)\Big).
\end{align*}

Therefore
\begin{align*}
&\int_0^T\phi'(t)
    \left(\int_{\Omega}\int_{\RR} 
        f\theta_h(x)\Big(1+\frac{|\vv|^2}{2}\Big)\chi_R(\vv)\,d\vv\,dx\right) dt \\
&\qquad+
\int_{0}^{T}\phi(t)\left(\int_\Omega\int_{\RR}\vv f \nabla_x\theta_h(x)\Big(1+\frac{|\vv|^2}{2}\Big)\chi_R(\vv)d\vv dx\right)dt\\
&\qquad+
\int_0^T\phi(t)
    \left(\int_{\Omega}\int_{\RR}
        (\uu-\vv)\cdot
        \Big(1+\frac{|\vv|^2}{2}\Big)\nabla_\vv\chi_R(\vv)\, \theta_h(x)f
        \,d\vv\,dx\right) dt \\
&\qquad+
\int_0^T\phi(t)
    \left(\int_{\Omega}\int_{\RR}
        (\uu-\vv)\cdot\vv\, \theta_h(x)f \chi_R(\vv)\,d\vv\,dx\right) dt
:=\sum_{i=1}^{4} G_i=0 .
\end{align*}

\textbf{Estimate of $G_1$.}
Using the finite kinetic energy of $f$, the term $G_1$ could be controlled, i.e.
\begin{align}\label{G_1}
	\lim_{R\to\infty}\lim_{h\to0}\int_0^T\phi'(t)
    \left(\int_{\Omega}\int_{\RR} 
        f\theta_h(x)\Big(1+\frac{|\vv|^2}{2}\Big)\chi_R(\vv)\,d\vv\,dx\right) dt\notag\\
	=\int_0^T\phi'(t)
    \left(\int_{\Omega}\int_{\RR} 
        f\Big(1+\frac{|\vv|^2}{2}\Big)\,d\vv\,dx\right) dt.
\end{align}

\textbf{Estimate of $G_2$.} We define the truncated macroscopic momentum flux
\begin{align*}
    M_R(x)=\int_{\RR} \vv \left(1+\frac{|\vv|^2}{2}\right)\chi_R(\vv) f(x,\vv) \, d\vv.
\end{align*}
Since the kernel $\vv(1+\frac{|\vv|^2}{2})\chi_R(\vv)$ is bounded and supported in $\{|\vv|\leq 2R\}$,
the proof of Lemma \ref{le-MR} applies verbatim, so that $M_R\in B_p^{\alpha,\infty}(\Omega)$ with
\begin{align*}
    \|M_R\|_{B_p^{\alpha,\infty}(\Omega)}\leq C_R\|f\|_{B_p^{\alpha,\infty}(\Omega\times\RR)}.
\end{align*}
Moreover, $|M_R(x)| \leq C_R \bar{\rho}(x)$ and $\bar\rho\in L^1(\Omega)$;
hence $M_R\in L^1(\Omega)\cap B_p^{\alpha,\infty}(\Omega)$, uniformly with respect to $h$. Then
\begin{align*}
    &\int_{0}^{T}\phi(t)\left(\int_\Omega\int_{\RR}\vv f \nabla_x\theta_h(x)\Big(1+\frac{|\vv|^2}{2}\Big)\chi_R(\vv)\, d\vv \, dx\right)dt\\
    &\qquad=\int_{0}^{T}\phi(t)\left(\int_\Omega\nabla_x\theta_h(x)\cdot M_R(x)\, dx\right) dt\\
    &\qquad=-\frac{1}{h}\int_{0}^{T}\phi(t)\left(\int_{h\leq d(x,\partial\Omega)\leq 2h}
    M_R(x)\cdot\nu(\sigma_x)\,\kappa'\Big(\frac{d(x,\partial\Omega)}{h}\Big)\, dx\right)dt,
\end{align*}
where $\nu(\sigma_x)$ denotes the outward unit normal to $\partial\Omega$ at the nearest point $\sigma_x$,
since $\nabla d(x)=-\nu(\sigma(x))$.

Thanks to the trace condition $\alpha p>1$ assumed in Theorem \ref{renormalization},
Lemmas \ref{le-trace} and \ref{le-layer} apply componentwise to $M_R$: the trace $\operatorname{tr}(M_R)\in L^p(\partial\Omega)$ is well
defined and
\begin{align*}
    \lim_{h\to 0}\frac{1}{h}\int_{h\leq d(x,\partial\Omega)\leq 2h} M_R(x)\,\kappa'\Big(\frac{d(x,\partial\Omega)}{h}\Big)\, dx
    =\int_{\partial\Omega} \operatorname{tr}(M_R)(x)\, d\sigma_x.
\end{align*}
Furthermore, the specular reflection boundary condition \eqref{Boundary conditions}, understood in the trace sense,
implies that $\operatorname{tr}(M_R)$ is tangential to $\partial\Omega$:
indeed, the change of variables $\vv\mapsto\vv^*$ in the integral shows that $M_R$ coincides with its reflection in the normal direction,
so its normal component vanishes. Therefore $\operatorname{tr}(M_R)\cdot\nu=0$ on $\partial\Omega$, and
\begin{align}\label{G_2}
    \lim_{R\to\infty}\lim_{h\to 0}\int_{0}^{T}\phi(t)\left(\int_\Omega\int_{\RR}\vv f
    \nabla_x\theta_h(x)\Big(1+\frac{|\vv|^2}{2}\Big)\chi_R(\vv)\, d\vv \, dx\right)dt=0.
\end{align}

\textbf{Estimate of $G_3$.} We decompose $G_3$ as follows:
\begin{align*}
	G_3
	&=\int_0^T\phi(t)
    \left(\int_{\Omega}\int_{\RR}
        (\uu-\vv)\cdot
        \Big(1+\frac{|\vv|^2}{2}\Big)\nabla_\vv\chi_R(\vv)\, \theta_h(x)f
        \,d\vv\,dx\right) dt\\
	&=\int_0^T\int_\Omega\int_{\RR}\phi(t)\theta_h f\uu\left(1+\frac{|\vv|^2}{2}\right)\nabla_\vv\chi_R(\vv) dxd\vv dt\\
	&\quad-\int_0^T\int_\Omega\int_{\RR}\phi(t)\theta_h f\vv\left(1+\frac{|\vv|^2}{2}\right)\nabla_\vv\chi_R(\vv) dxd\vv dt\\
	&:= G_{31}-G_{32}.
\end{align*}
For term $G_{31}$, note that 
\begin{align*}
	|\phi(t)\theta_h f\uu\left(1+\frac{|\vv|^2}{2}\right)\nabla_\vv\chi_R(\vv)|\leq \|\phi\|_{C^0}|\vv |f |\uu|\|\chi_R\|_{C^1}.
\end{align*}
Using Lebesgue dominated convergence theorem, if 
\begin{align*}
	\frac{1}{p_2}+\frac{1}{q}\le1\quad\Longleftrightarrow\quad
	(p-1)(q-n-2)>2,
\end{align*}
we have 
\begin{align*}
	\lim_{R\to\infty}\lim_{h\to0}\int_0^T\int_\Omega\int_{\RR}\phi(t)\theta_h f\uu\left(1+\frac{|\vv|^2}{2}\right)\nabla_\vv\chi_R(\vv) dxd\vv
	dt=0.
\end{align*}
Then for term $G_{32}$, notice that 
\begin{align*}
	\int_{\RR}f\vv\left(1+\frac{|\vv|^2}{2}\right)\nabla_\vv\chi_R(\vv)d\vv
	&\leq\frac{C}{R}\int_{R\leq|\vv|\leq 2R}\left(1+\frac{|\vv|^2}{2}\right)f\vv d\vv\\
	&\leq2C\int_{R\leq|\vv|\leq 2R}\left(1+\frac{|\vv|^2}{2}\right)f d\vv.
\end{align*}
Since $\int_{\RR}f\left(1+\frac{|\vv|^2}{2}\right)d\vv\in L^1((0,T)\times\Omega)$, it is easy to derive 
\begin{align*}
	\lim_{R\to\infty}\lim_{h\to0}\int_0^T\int_\Omega\int_{\RR}\phi(t)\theta_h f\vv\left(1+\frac{|\vv|^2}{2}\right)\nabla_\vv\chi_R(\vv) dxd\vv
	dt=0.
\end{align*}
Thus, we obtain 
\begin{align}\label{G_3}
	\lim_{R\to\infty}\lim_{h\to0}\int_0^T\phi(t)
    \left(\int_{\Omega}\int_{\RR}
        (\uu-\vv)\cdot
        \Big(1+\frac{|\vv|^2}{2}\Big)\nabla_\vv\chi_R(\vv)\, \theta_h(x)f
        \,d\vv\,dx\right) dt=0.
\end{align}

\textbf{Estimate of $G_4$.} The first of $G_4$ is similar to the estimate $G_{31}$,
and using the finite kinetic energy of $f$ again, one obtains easily
\begin{align}\label{G_4}
	\lim_{R\to\infty}\lim_{h\to0}\int_0^T
    \int_{\Omega}\int_{\RR}
        \phi(t)(\uu-\vv)\cdot\vv\theta_h(x)f \chi_R(\vv)d\vv dx dt=\int_0^T\phi(t)
    \left(\int_{\Omega}\int_{\RR}
        (\uu-\vv)\cdot\vv f d\vv dx\right) dt.
\end{align}

Thanks to the estimates \eqref{G_1}-\eqref{G_4}, letting $h\to 0$ first and then $R\to\infty$, we obtain
\begin{align*}
&\int_0^T\phi'(t)
    \left(\int_{\Omega}\int_{\mathbb{R}^n} 
        f\Big(1+\frac{|\vv|^2}{2}\Big)\,d\vv\,dx\right) dt \\
&\qquad+
\int_0^T\phi(t)
    \left(\int_{\Omega}\int_{\mathbb{R}^n}
        (\uu-\vv)\cdot\vv\, f \,d\vv\,dx\right) dt
=0 .
\end{align*}

We now rewrite the second term. Using the identity
\begin{align*}
(\uu-\vv)\cdot\uu 
    = (\uu-\vv)\cdot\vv + |\uu-\vv|^2 ,
\end{align*}
we obtain
\begin{align*}
\int_{\Omega}\int_{\RR}
    (\uu-\vv)\cdot\vv\, f\,d\vv\,dx
    &=
\int_{\Omega}\int_{\RR}
    (\uu-\vv)\cdot\uu\, f\,d\vv\,dx
    -
\int_{\Omega}\int_{\RR}
    |\uu-\vv|^2 f\,d\vv\,dx .
\end{align*}

Hence the weak identity becomes
\begin{align*}
&\int_0^T\phi'(t)
    \left(\int_{\Omega}\int_{\RR} 
        f\Big(1+\frac{|\vv|^2}{2}\Big)\,d\vv\,dx\right) dt  \\
&\qquad+
\int_0^T\phi(t)
    \left(
        \int_{\Omega}\int_{\RR}
        (\uu-\vv)\cdot\uu\, f\, d\vv\, dx
        -
        \int_{\Omega}\int_{\RR}
        |\uu-\vv|^2 f\, d\vv\, dx
    \right) dt
=0 .
\end{align*}

Finally, combining this identity with the limit obtained in \eqref{t5_refined},
we conclude that
\begin{align}\label{4.13}
\int_0^T \phi'(t) F(f(t))\,dt
    &+
\int_0^T \phi(t)
        \int_{\Omega} \uu\cdot
            \Big(\int_{\RR} (\uu-\vv) f\,d\vv\Big)\,dxdt
    \notag\\&\qquad=
\int_0^T \phi(t)
        \int_{\Omega}\int_{\RR}
            |\uu-\vv|^2 f \,d\vv\,dxdt ,
\end{align}
where 
\begin{align*}
F(f(t))=\int_{\Omega}\int_{\RR}
        f(t,x,\vv)\Big(1+\frac{|\vv|^2}{2}\Big)\,d\vv\,dx.
\end{align*}

	
	Finally we extend the identity to test functions that do not vanish at $t=0$.
	Under the regularity assumptions of Theorem \ref{renormalization}, the following time continuity properties hold:
	\begin{itemize}
		\item[(i)] since $\beta>\frac1q$,
		the space-time Besov space $B_q^{\beta,\infty}([0,T]\times\Omega)$ embeds into $C([0,T];L^q(\Omega))$ by the Morrey-type embedding in the
		time variable, so that $\uu\in C([0,T];L^q(\Omega))$;
		as $q>n+2$, the map $t\mapsto\int_\Omega\frac12|\uu(t,x)|^2\,dx$ is continuous on $[0,T]$;
		\item[(ii)] from the Vlasov equation and the bound $f(1+\frac{|\vv|^2}{2})\in L^\infty(0,T;L^1(\Omega\times\RR))$,
		the map $t\mapsto\int_{\Omega\times\RR}f(t)\psi\,dx d\vv$ is continuous for every compactly supported test function $\psi$;
		the strong energy inequality in Definition \ref{def weak1} gives
		$\limsup_{t\to s^+}E(t)\le E(s)$ for almost every $s\in[0,T)$ including $s=0$,
		while the weak lower semicontinuity of the kinetic energy gives
		$F(f(s))\le\liminf_{t\to s^+}F(f(t))$;
		combined with the continuity of the fluid energy from (i), this yields
		\[
		\limsup_{t\to s^+}F(f(t))\le F(f(s))\le\liminf_{t\to s^+}F(f(t)),
		\qquad F(f(t)):=\int_{\Omega}\int_{\RR}f(t)\Big(1+\frac{|\vv|^2}{2}\Big)d\vv\,dx,
		\]
		so that $F(f(t))$ is right-continuous at almost every $s\in[0,T)$, in particular at $s=0$.
	\end{itemize}
	Consequently, after a possible redefinition on a time null set,
	$t\mapsto E(t):=\int_\Omega\frac12|\uu|^2\,dx+F(f(t))$ coincides with an absolutely continuous function on $[0,T]$,
	and approximating $\mathbf{1}_{[0,t]}$ by test functions $\phi\in\mathcal{D}(-1,+\infty)$ in the identity obtained from
	\eqref{t1}--\eqref{4.13}
	yields $E(0)-E(t)=\int_0^t\big(\mu\int_\Omega|\nabla\uu|^2dx+\int_{\Omega\times\RR}f|\uu-\vv|^2d\vv dx\big)ds$ for every $t\in[0,T]$;
	the detailed justification of this extension step is given in Step 6 of Section \ref{se-pro-th2}.

	Combining \eqref{t1}--\eqref{4.13} and passing to the limit \(\epsilon,h\to0\),
	we obtain the global energy equality for the incompressible Vlasov--Navier--Stokes system: for every $t\in[0,T]$,
	\[
	\begin{aligned}
		&\int_\Omega \tfrac{1}{2}|\uu(t)|^2\,dx
		+\int_{\Omega}\int_{\RR} f(t)\Big(1+\tfrac{|\vv|^2}{2}\Big)\,d\vv\,dx
		+\int_{0}^{t}\int_\Omega \mu|\nabla\uu|^2\,dxd\tau \\
		&\qquad+\int_{0}^{t}\int_{\Omega\times\RR} f|\uu-\vv|^2\,d\vv\,dxd\tau
		= \int_\Omega \tfrac{1}{2}|\uu_0|^2\,dx
		+\int_{\Omega}\int_{\RR} f_0\Big(1+\tfrac{|\vv|^2}{2}\Big)\,d\vv\,dx.
	\end{aligned}
	\]

\end{proof}

\section{Proof of Theorem \ref{CNSV thm}}\label{se-pro-th2}
To derive the renormalization property in the compressible case,
we establish the following commutator estimates similarly to the
incompressible case.

\begin{lemma}\label{le-error-CNSV}
Let $(\rho,\uu,f)$ satisfy the regularity hypotheses in
Theorem~\ref{CNSV thm}. Then
\begin{align}
	&\|\nabla_x(\vv f)^{\epsilon,\delta,\eta}-\vv^\delta\cdot\nabla_x f^{\epsilon,\delta,\eta}\|_{L^1(0,T;L^p(\Omega\times\RR))}\notag\\
	&\quad\leq C\big(\epsilon^{\alpha-1}\delta\wedge\epsilon^{-1}\delta^{\alpha+1}\big)\W_f(\epsilon,\delta),\\
	&\big\|\nabla_{\vv}\cdot\big((f\rho \uu)^{\epsilon,\delta,\eta}
	-f^{\epsilon,\delta,\eta}\rho^{\epsilon,\eta}\uu^{\epsilon,\eta}\big)\big\|_{L^{s_3}(0,T;L^{s_3}(\Omega;L^p(\RR)))}\notag\\
	&\quad\leq C\Big(\epsilon^{\tau}\big(\epsilon^{\alpha}\delta^{-1}\wedge\delta^{\alpha-1}\big)+\epsilon^{\tau+\beta}\delta^{\alpha-1}\Big)
	\|f\|_{B_p^{\alpha,\infty}([0,T]\times\Omega\times\RR)},\\
	&\big\|\nabla_{\vv}\cdot\big((f\rho \vv)^{\epsilon,\delta,\eta}
	-f^{\epsilon,\delta,\eta}\rho^{\epsilon,\eta}\vv^\delta\big)\big\|_{L^{s_1}(0,T;L^{s_1}(\Omega;L^p(\RR)))}\notag\\
	&\quad\leq C\Big(\epsilon^{\tau}\big(\epsilon^{\alpha}\delta^{-1}\wedge\delta^{\alpha-1}\big)+\delta^{\alpha}\Big)
	\|f\|_{B_p^{\alpha,\infty}([0,T]\times\Omega\times\RR)}.
\end{align}
Here we denote $\frac{1}{p}+\frac{1}{r}=\frac{1}{s_1}<1$,
$\frac{1}{q}+\frac{1}{r}=\frac{1}{s_2}<1$ and $\frac{1}{p}+\frac{1}{q}+\frac{1}{r}=\frac{1}{s_3}<1$.
In these two estimates the constant $C$ depends on
$\|\rho\|_{L^\infty((0,T)\times\Omega)}$,
$\|\rho\uu\|_{B_r^{\tau,\infty}([0,T]\times\Omega)}$ and
$\|\uu\|_{B_q^{\beta,\infty}([0,T]\times\Omega)}$.
In particular, taking $\delta=\epsilon$, the right-hand sides are of order
$\epsilon^{\alpha}$, $\epsilon^{\alpha+\tau-1}+\epsilon^{\alpha+\beta+\tau-1}$ and
$\epsilon^{\alpha+\tau-1}+\epsilon^{\alpha}$ (each multiplied by
$\|f\|_{B_p^{\alpha,\infty}([0,T]\times\Omega\times\RR)}$), respectively,
and tend to zero provided
$\alpha+\tau>1$ and $\alpha+\beta+\tau>1$.
\end{lemma}

\begin{proof}
The first estimate is identical to \eqref{eq-comm-1} in Lemma \ref{le-error}:
the transport commutator does not involve the drag force,
so the density $\rho$ plays no role in it and the proof of
Lemma \ref{le-error} applies verbatim.
We therefore only estimate the commutator generated by the drag force term
$\nabla_{\vv}\cdot\big(f\rho(\uu-\vv)\big)$
in the renormalization procedure.
Compared with the incompressible case, the only new difficulty
comes from the presence of the density $\rho$,
which produces triple products.

We consider the commutator generated by the drag force term
$\nabla_{\vv}\cdot(f\rho(\uu-\vv))$,
namely
\begin{align*}
\mathcal{R}_{\epsilon,\delta,\eta}
:=
\nabla_{\vv}\cdot
\Big(
(f\rho(\uu-\vv))^{\epsilon,\delta,\eta}
-
f^{\epsilon,\delta,\eta}\rho^{\epsilon,\eta}(\uu-\vv)^{\epsilon,\eta,\delta}
\Big).
\end{align*}
We decompose it into the fluid velocity part and the particle velocity part as
\begin{align*}
\mathcal{R}_{\epsilon,\delta,\eta}
=
\mathcal{R}^{\uu}_{\epsilon,\delta,\eta}
-
\mathcal{R}^{\vv}_{\epsilon,\delta,\eta},
\end{align*}
where
\begin{align*}
\mathcal{R}^{\uu}_{\epsilon,\delta,\eta}
&:=
\nabla_{\vv}\cdot
\Big(
(f\rho \uu)^{\epsilon,\delta,\eta}
-
f^{\epsilon,\delta,\eta}\rho^{\epsilon,\eta}\uu^{\epsilon,\eta}
\Big),\\
\mathcal{R}^{\vv}_{\epsilon,\delta,\eta}
&:=
\nabla_{\vv}\cdot
\Big(
(f\rho \vv)^{\epsilon,\delta,\eta}
-
f^{\epsilon,\delta,\eta}\rho^{\epsilon,\eta}\vv^\delta
\Big).
\end{align*}

\textbf{Estimate of the fluid velocity part $\mathcal{R}^{\uu}_{\epsilon,\delta,\eta}$.}

We decompose
\begin{align*}
&(f\rho \uu)^{\epsilon,\delta,\eta}
-
f^{\epsilon,\delta,\eta}\rho^{\epsilon,\eta}\uu^{\epsilon,\eta}
\\
&=
\big((\rho \uu)f\big)^{\epsilon,\delta,\eta}
-
(\rho \uu)^{\epsilon,\eta} f^{\epsilon,\delta,\eta}
+
\big(
(\rho \uu)^{\epsilon,\eta}
-
\rho^{\epsilon,\eta}\uu^{\epsilon,\eta}
\big)
f^{\epsilon,\delta,\eta}
\\
&= I_{11} + I_{12}.
\end{align*}

For $I_{11}$, arguing exactly as in the incompressible case
(with $\uu$ replaced by $\rho \uu$, so that $\beta$ is replaced by $\tau$), we obtain
\begin{align*}
\|\nabla_\vv\cdot I_{11}\|_{L^{s_1}(0,T;L^{s_1}(\Omega;L^p(\RR)))}
\leq
C\epsilon^{\tau}\big(\epsilon^{\alpha}\delta^{-1}\wedge\delta^{\alpha-1}\big)
\|f\|_{B_p^{\alpha,\infty}([0,T]\times\Omega\times\RR)},
\end{align*}
where $C$ depends on $\|\rho\uu\|_{B_r^{\tau,\infty}([0,T]\times\Omega)}$ and $\frac{1}{p}+\frac{1}{r}=\frac{1}{s_1}<1.$

For $I_{12}$,
using the Besov regularity of $\rho$ and $\uu$ and H\"older's inequality with $\frac{1}{q}+\frac{1}{r}=\frac{1}{s_2}<1$, we infer
\begin{align*}
\|(\rho \uu)^{\epsilon,\eta}
-
\rho^{\epsilon,\eta}\uu^{\epsilon,\eta}\|_{L^{s_2}((0,T)\times\Omega)}
\le
C\,\epsilon^{\tau+\beta}
\|\rho\|_{B^{\tau,\infty}_r([0,T]\times\Omega)}
\|\uu\|_{B^{\beta,\infty}_q([0,T]\times\Omega)}.
\end{align*}
Therefore,
\begin{align*}
\|\nabla_\vv\cdot I_{12}\|_{L^{s_3}(0,T;L^{s_3}(\Omega;L^p(\RR)))}
&\leq
\|(\rho \uu)^{\epsilon,\eta}
-
\rho^{\epsilon,\eta}\uu^{\epsilon,\eta}\|_{L^{s_2}((0,T)\times\Omega)}\|\nabla_{\vv}f^{\epsilon,\delta,\eta}\|_{L^p(0,T;B_p^{\alpha,\infty}(\Omega\times\RR))}
\\
&\le
C\epsilon^{\tau+\beta}\delta^{\alpha-1}\|f\|_{B_p^{\alpha,\infty}([0,T]\times\Omega\times\RR)},
\end{align*}
where $\frac{1}{p}+\frac{1}{q}+\frac{1}{r}=\frac{1}{s_3}<1$.

Collecting the above estimates, we conclude that
\[
\|\mathcal{R}^\uu_{\epsilon,\delta,\eta}\|_{L^{s_3}(0,T;L^{s_3}(\Omega;L^p(\RR)))}
\leq
C\Big(\epsilon^{\tau}\big(\epsilon^{\alpha}\delta^{-1}\wedge\delta^{\alpha-1}\big)+\epsilon^{\tau+\beta}\delta^{\alpha-1}\Big)
\|f\|_{B_p^{\alpha,\infty}([0,T]\times\Omega\times\RR)}
.
\]

\textbf{Estimate of the particle velocity part $\mathcal{R}^{\vv}_{\epsilon,\delta,\eta}$.}

We similarly decompose
\begin{align*}
&(f\rho \vv)^{\epsilon,\delta,\eta}
-
\rho^{\epsilon,\eta} f^{\epsilon,\delta,\eta} \vv^\delta
\\
&=\big(f\rho \vv\big)^{\epsilon,\delta,\eta}
-
(f\vv)^{\epsilon,\delta,\eta} \rho^{\epsilon,\eta}
+
\big(
(f\vv)^{\epsilon,\delta,\eta}
-
f^{\epsilon,\delta,\eta}\vv^{\delta}
\big)
\rho^{\epsilon,\eta}\\
&= I_{21} + I_{22}.
\end{align*}

We first consider $I_{21}$.
Applying $\nabla_\vv$ and using Minkowski's inequality, we obtain
\begin{align*}
&\|\nabla_\vv\cdot I_{21}\|_{L^{s_1}(0,T;L^{s_1}(\Omega;L^p(\RR)))}\\
&\leq\|\nabla_\vv\cdot\omega_{\epsilon,\eta}(\rho,(f\vv)^\delta)\|_{L^{s_1}(0,T;L^{s_1}(\Omega;L^p(\RR)))}\\
&\quad+
\|\nabla_\vv\cdot((f\vv)^\delta-(f\vv)^{\epsilon,\delta,\eta})(\rho-\rho^{\epsilon,\eta})\|_{L^{s_1}(0,T;L^{s_1}(\Omega;L^p(\RR)))}\\
&\leq\|\nabla_\vv\cdot\omega_{\epsilon}(\rho,(f\vv)^\delta)\|_{L^{s_1}(0,T;L^{s_1}(\Omega;L^p(\RR)))}\\
&\quad+
\|\nabla_\vv\cdot((f\vv)^\delta-(f\vv)^{\epsilon,\delta})(\rho-\rho^{\epsilon})\|_{L^{s_1}(0,T;L^{s_1}(\Omega;L^p(\RR)))}\\
&\leq
C\epsilon^{\tau}\big(\epsilon^{\alpha}\delta^{-1}\wedge\delta^{\alpha-1}\big)\|\rho\|_{B_r^{\tau,\infty}([0,T]\times\Omega)}\|f\|_{B_p^{\alpha,\infty}([0,T]\times\Omega\times\RR)},
\end{align*}
where in the last step we used the two-route estimate of Lemma \ref{le-diff} for the
$\vv$-derivative of $(f\vv)^\delta$ (the extra factor $(1+|\vv|)$ produced by differentiating
$(f\vv)^\delta$ in $\vv$ is harmless, since all commutators are eventually tested against the
compactly supported function $\Phi$).

For $I_{22}$, using Lemma \ref{le-error} we similarly obtain
\begin{align*}
\|\nabla_\vv\cdot I_{22}\|_{L^{s_1}(0,T;L^{s_1}(\Omega;L^p(\RR)))}
&\leq
\|\rho^{\epsilon,\eta}\|_{L^r((0,T)\times\Omega)}\|\nabla_\vv\cdot\big((f\vv)^{\epsilon,\delta,\eta}
-
f^{\epsilon,\delta,\eta}\vv^{\delta}\big)\|_{L^p((0,T)\times\Omega\times\RR)}
\\
&\leq\|\rho^{\epsilon}\|_{L^r((0,T)\times\Omega)}\|\nabla_\vv\cdot\big((f\vv)^{\epsilon,\delta}
-
f^{\epsilon,\delta}\vv^{\delta}\big)\|_{L^p((0,T)\times\Omega\times\RR)}\\
&\leq
C \delta^{\alpha}\|\rho\|_{B_r^{\tau,\infty}([0,T]\times\Omega)}\|f\|_{B_p^{\alpha,\infty}([0,T]\times\Omega\times\RR)}.
\end{align*}

Collecting the above estimates, we conclude that
\[
\|\mathcal{R}^\vv_{\epsilon,\delta,\eta}\|_{L^{s_1}(0,T;L^{s_1}(\Omega;L^p(\RR)))}
\leq
C\Big(\epsilon^{\tau}\big(\epsilon^{\alpha}\delta^{-1}\wedge\delta^{\alpha-1}\big)+\delta^{\alpha}\Big)
\|f\|_{B_p^{\alpha,\infty}([0,T]\times\Omega\times\RR)}
.
\]

\end{proof}

Now we start the proof of Theorem \ref{CNSV thm}.
\begin{proof}[Proof]
	We first show the renormalization property \eqref{eq-renorm-2}, following the same line as in the incompressible case.
	For any test function $\Theta\in\D((0,T)\times\Omega\times\RR)$, the weak formulation of the Vlasov equation reads
	\begin{equation*}
		\int_{0}^{T}\int_\Omega\int_{\RR}f(\partial_t\Theta+\vv\nabla_x\Theta+\rho(\uu-\vv)\nabla_\vv\Theta)dxd\vv dt=0.
	\end{equation*}
	Choosing the same test function as in the incompressible case,
	$\Theta=\theta_h(\mathcal{R}'(f_h^{\epsilon,\delta,\eta})\Phi)^{\epsilon,\delta,\eta}$
	with $f_h:=\theta_h f$, $\Phi\in \D((0,T)\times\Omega\times\RR)$
	and $\mathcal{R}\in \mathcal{S}_k(\R^+;\R^+)$, $k\in\big[2,2p\big(1-\frac1q-\frac1r\big)\big]$,
	and integrating by parts, we arrive at the analogue of the identity obtained in the incompressible case,
	with the drag terms carrying the additional factor $\rho^{\epsilon,\eta}$
	and with the same boundary term involving $\nabla_x\theta_h$.
	This boundary term vanishes identically for all sufficiently small $h$ and $\epsilon$
	by the same support-separation argument as before:
	the spatial support of $(\mathcal{R}'(f_h^{\epsilon,\delta,\eta})\Phi)^{\epsilon,\delta,\eta}$ is contained in
	$\{d(x,\partial\Omega)\ge d_0/2\}$ for all $\epsilon<d_0/4$,
	while $\nabla_x\theta_h$ is supported in $\{h\le d(x,\partial\Omega)\le 2h\}$, so the two supports are disjoint.
	Consequently, for all $h,\epsilon<d_0/8$ the function $f_h$ can be replaced by $f$ in the remaining integrands;
	moving all commutators to the right-hand side and applying Lemma \ref{le-error-CNSV}, we obtain
	\begin{align*}
		&\biggl| \int_{0}^{T} \int_{\Omega} \int_{\RR} \mathcal{R}(f^{\epsilon,\delta,\eta}) (\partial_t \Phi + \mathbf{v}^\delta \cdot \nabla_x
		\Phi + \rho^{\epsilon,\eta}(\mathbf{u}^{\epsilon,\eta} - \mathbf{v}^\delta) \cdot \nabla_{\mathbf{v}} \Phi) \\
		&\quad + \Phi(\nabla_\vv\cdot\mathbf{v}^\delta) (\rho^{\epsilon,\eta}f^{\epsilon,\delta,\eta} \mathcal{R}'(f^{\epsilon,\delta,\eta}) -
		\rho^{\epsilon,\eta}\mathcal{R}(f^{\epsilon,\delta,\eta})) \, d\mathbf{v} \, dx \, dt \biggr|\\
		&\leq C\Big(\big(\epsilon^{\alpha-1}\delta\wedge\epsilon^{-1}\delta^{\alpha+1}\big)\Big)\W_f(\epsilon,\delta)\\
		&\quad+C\Big(\epsilon^{\tau}\big(\epsilon^{\alpha}\delta^{-1}\wedge\delta^{\alpha-1}\big)
		+\epsilon^{\tau+\beta}\delta^{\alpha-1}+\delta^{\alpha}\Big)
		\|f\|_{B_p^{\alpha,\infty}([0,T]\times\Omega\times\RR)}.
	\end{align*}
	Here the range of $k$ comes from pairing $\mathcal{R}'(f^{\epsilon,\delta,\eta})\Phi$ against the commutators in
	$L^p$, $L^{s_1}$ and $L^{s_3}$: the sharpest requirement is $\big(\frac{k}{2}-1\big)s_3'\leq p$, that is,
	$k\leq 2p\big(1-\frac1q-\frac1r\big)$.
	Choosing $\delta=\epsilon$ and using that $\alpha+\tau>1$ and $\alpha+\beta+\tau>1$,
	the right-hand side tends to zero as $\epsilon,\delta\to0$,
	which gives the renormalization property \eqref{eq-renorm-2}.
\end{proof}

In the remainder of this section, for notational simplicity, we write
\begin{align*}
    f^\epsilon(x,t)=m_\epsilon *_x m_\epsilon *_t f(x,t)
\end{align*}
for the joint space-time mollification of a function $f(x,t)$; that is,
the time mollification parameter $\eta$ of Section \ref{se-pre} is taken
equal to the space parameter $\epsilon$.
For the commutator estimates, we introduce
\begin{align*}
    \omega_\epsilon(f,g)(x,t)
    = \int_{\epsilon}^{2\epsilon} \int_\Omega m_\epsilon(y)m_\epsilon(\tau)\, \Delta_\epsilon f \cdot \Delta_\epsilon g \,dy\,d\tau,
\end{align*}
where $\Delta_\epsilon f:=f(x-y,t-\tau)-f(x,t)$, and $\Delta_\epsilon g$ is defined analogously.

Here we introduce a new function $\Phi^{h,\epsilon}=\theta_h((\phi(t)\theta_h\uu)^\epsilon)^\epsilon$,
where $\theta_h(x)$ is the cutoff function mentioned before. Note that $\phi(t)\in\mathcal{D}(0,+\infty)$ is a test function,
where $\mathcal{D}(0,+\infty)$ denotes the class of compactly supported smooth functions in $(0,+\infty)$.
In particular, $\phi$ vanishes near $t=0$.
However, it is needed to extend the result for $\phi(t)\in\mathcal{D}(-1,+\infty)$.

Using $\Phi^{h,\epsilon}$ to test the second equation of \eqref{CNSVsystem}, one can deduce that
\begin{eqnarray*}
	\int_{0}^{T}\int_\Omega \Phi^{h,\epsilon}\left((\rho\uu)_t + \nabla_x\cdot(\rho\uu\otimes\uu) + \nabla_x P(\rho) - \mu \Delta_x
	\uu-\lambda\nabla_x \mathrm{div}\uu  \displaystyle{+\int_{\RR} \rho(\uu - \vv) f d\vv}\right)dxdt=0,
\end{eqnarray*}

which in turn yields
\begin{align}\label{CNSV proof Main}
	\int_{0}^{T}\int_{\Omega} \phi(t)(\theta_h \uu)^\epsilon \Bigg[
	& (\theta_h\rho\uu)_t 
	+ \theta_h\nabla_x\cdot(\rho\uu\otimes\uu) 
	+ \theta_h\nabla_x P(\rho) \notag \\
	& - \mu \theta_h\Delta_x \mathbf{u}
	- \lambda\theta_h\nabla_x \mathrm{div}\uu
	+ \theta_h\int_{\mathbb{R}^n} \rho(\mathbf{u} - \mathbf{v}) f \, d\mathbf{v}
	\Bigg]^\epsilon dxdt = 0.
\end{align}

\begin{proof}[Proof]
	\textbf{Step 1.}
	The first term in \eqref{CNSV proof Main} shows that 
	\begin{align}\label{A}
		\int_{0}^{T}\int_\Omega \phi(t)(\theta_h\uu)^\epsilon (\theta_h\rho\uu)^\epsilon_t dxdt
		&=\int_{0}^{T}\int_\Omega
		\phi(t)(\theta_h\uu)^\epsilon\left((\theta_h\rho\uu)^\epsilon_t-(\rho^\epsilon(\theta_h\uu)^\epsilon)_t\right)dxdt\notag\\
		&\quad+ \int_{0}^{T}\int_\Omega\phi(t)(\theta_h\uu)^\epsilon (\rho^\epsilon(\theta_h\uu)^\epsilon)_t dxdt\notag\\
		&= A +\int_{0}^{T}\int_\Omega\phi(t)\rho_t^\epsilon|(\theta_h\uu)^\epsilon|^2 dxdt\notag\\
		&\quad+\int_{0}^{T}\int_\Omega\phi(t)\rho^\epsilon
		\partial_t\frac{|(\theta_h\uu)^\epsilon|^2}{2}dxdt.
	\end{align}


	Similarly, the second term in \eqref{CNSV proof Main} gives us
	\begin{align}\label{B}
		&\int_{0}^{T}\int_\Omega \phi(t)(\theta_h \uu)^\epsilon 
		\, (\theta_h\mathrm{div}(\rho\uu\otimes\uu))^\epsilon \, dxdt\notag\\
		&=-\int_{0}^{T}\int_\Omega \phi(t)(\theta_h\uu)^\epsilon(\nabla\theta_h(\rho\uu\otimes\uu))^\epsilon dxdt
		-\int_{0}^{T}\int_\Omega \phi(t)\nabla(\theta_h\uu)^\epsilon(\theta_h(\rho\uu\otimes\uu))^\epsilon dxdt\notag\\
		&=\int_{0}^{T}\int_\Omega
		\phi(t)(\theta_h\uu)^\epsilon(\nabla\theta_h(\rho\uu\otimes\uu)^\epsilon\notag\\
		&\qquad-(\nabla\theta_h(\rho\uu\otimes\uu))^\epsilon)dxdt\notag\\
		&\quad+\int_{0}^{T}\int_\Omega
		\phi(t)(\theta_h\uu)^\epsilon\nabla\theta_h(((\rho\uu)^\epsilon\otimes\uu^\epsilon)-(\rho\uu\otimes\uu)^\epsilon)dxdt\notag\\
		&\quad-\int_{0}^{T}\int_\Omega \phi(t)(\theta_h\uu)^\epsilon\nabla\theta_h((\rho\uu)^\epsilon\otimes\uu^\epsilon)dxdt\notag\\
		&\quad+\int_{0}^{T}\int_\Omega \phi(t)\nabla(\theta_h\uu)^\epsilon
		((\rho\uu)^\epsilon\otimes(\theta_h\uu)^\epsilon-(\rho\uu\otimes\theta_h\uu)^\epsilon)dxdt-\int_{0}^{T}\int_\Omega
		\phi(t)\nabla(\theta_h\uu)^\epsilon(\rho\uu)^\epsilon\otimes(\theta_h\uu)^\epsilon dxdt\notag\\
		&=B_1+B_2+B_3\notag\\
		&\quad-\int_{0}^{T}\int_\Omega
		\phi(t)(\theta_h\uu)^\epsilon\nabla\theta_h((\rho\uu)^\epsilon\otimes\uu^\epsilon)dxdt-\int_{0}^{T}\int_\Omega
		\phi(t)\nabla(\theta_h\uu)^\epsilon(\rho\uu)^\epsilon\otimes(\theta_h\uu)^\epsilon dxdt.
	\end{align}

	Combining \eqref{A}-\eqref{B}, the first two terms in \eqref{CNSV proof Main} are given by 
	
	\begin{eqnarray}\label{8.5}
		\int_{0}^{T}\!\!\int_\Omega \phi(t)\left(\frac{1}{2}\rho^\epsilon|(\theta_h\uu)^\epsilon|^2\right)_t dxdt + A+
		B_1+B_2+B_3-\int_{0}^{T}\!\!\int_\Omega \phi(t)(\theta_h\uu)^\epsilon\nabla\theta_h((\rho\uu)^\epsilon\otimes\uu^\epsilon)dxdt.
	\end{eqnarray}
	Here the term $\int_{0}^{T}\int_\Omega\phi(t)\rho_t^\epsilon|(\theta_h\uu)^\epsilon|^2/2\,dxdt$ from \eqref{A}
	has been cancelled against the last integral in \eqref{B}
	through the mollified mass equation $\partial_t\rho^\epsilon+\mathrm{div}(\rho\uu)^\epsilon=0$.

	\textbf{Step 2.}
	Since 
	\begin{align*}
		(fg)^\epsilon-f^\epsilon g^\epsilon=(f^\epsilon-f)(g-g^\epsilon)+\omega_\epsilon (f,g),
	\end{align*}
	term $A$ could be controlled by 
	\begin{align*}
		|A|
		&\leq\left|\int_{0}^{T}\int_\Omega
		\phi(t)(\theta_h\uu)^\epsilon\partial_t\left[(\rho^\epsilon-\rho)((\theta_h\uu)^\epsilon-(\theta_h\uu))\right]dxdt\right|+\left|\int_{0}^{T}\int_\Omega\phi(t)(\theta_h\uu)^\epsilon\partial_t\omega_\epsilon(\rho,(\theta_h\uu))dxdt\right|\\
		&\leq\left|\int_{0}^{T}\int_\Omega \partial_t\phi(t)(\rho^\epsilon-\rho)((\theta_h\uu)^\epsilon-(\theta_h\uu))(\theta_h\uu)^\epsilon
		dxdt\right|\\
		&\quad+\left|\int_{0}^{T}\int_\Omega\phi(t)(\rho^\epsilon-\rho)((\theta_h\uu)^\epsilon-(\theta_h\uu))\partial_t(\theta_h\uu)^\epsilon
		dxdt\right|\\
		&\quad+\left|\int_{0}^{T}\int_\Omega\partial_t\phi(t)\omega_\epsilon(\rho,(\theta_h\uu))(\theta_h\uu)^\epsilon
		dxdt\right|+\left|\int_{0}^{T}\int_\Omega\phi(t)\omega_\epsilon(\rho,(\theta_h\uu))\partial_t(\theta_h\uu)^\epsilon dxdt\right|\\
		&\leq C\|\phi(t)\|_{C^1}\|\rho\|_{B_r^{\tau,\infty}([0,T]\times\Omega)}\|\uu\|_{B_q^{\beta,\infty}([0,T]\times\Omega)}^2
		\epsilon^{\tau+\beta}\\
		&\quad+
		C\|\phi\|_{C^0}\|\rho\|_{B_r^{\tau,\infty}([0,T]\times\Omega)}\|\uu\|_{B_q^{\beta,\infty}([0,T]\times\Omega)}\epsilon^{\tau+2\beta-1}\to0
	\end{align*}
	as $\epsilon\to0$, if $\tau+2\beta-1>0$.
	
	Note that both products in $B_1$ are supported on the strip $\{h\le d(x,\partial\Omega)\le 2h\}$ (the second one
	up to the $O(\epsilon)$-enlargement caused by the mollification, with $\epsilon\le h$). Therefore, with
	$\frac1b=1-\frac2q-\frac1r$ and $\frac1s=\frac1q+\frac1r$,
	\begin{align*}
		|B_1|
		&=\left|\int_{0}^{T}\int_\Omega
		\phi(t)(\theta_h\uu)^\epsilon(\nabla\theta_h(\rho\uu\otimes\uu)^\epsilon-(\nabla\theta_h(\rho\uu\otimes\uu))^\epsilon)dxdt\right|\\
		&\leq C\|\phi\|_{C^0}\|(\theta_h\uu)^\epsilon\|_{L^q(0,T;L^q(\Omega_{3h}))}
		\Big(\|\nabla\theta_h\|_{L^b(0,T;L^b(\Omega))}\|(\rho\uu\otimes\uu)^\epsilon\|_{L^s(0,T;L^s(\Omega))}\notag\\
		&\qquad+\|(\nabla\theta_h(\rho\uu\otimes\uu))^\epsilon\|_{L^{q'}(0,T;L^{q'}(\Omega))}\Big)\\
		&\leq C\|\phi\|_{C^0}\|\rho\|_{L^\infty((0,T)\times\Omega)}
		\|\uu\|_{B_q^{\beta,\infty}([0,T]\times\Omega)}^3\,h^{\beta-\frac2q-\frac1r}\to0,
	\end{align*}
	where, for the first summand, H\"older's inequality is used with exponents $(q,b,s)$ such that
	$\frac1q+\frac1b+\frac1s=1$; for the second one, $\frac1q+\frac1{q'}=1$, Jensen's inequality
	$\|(\nabla\theta_h(\rho\uu\otimes\uu))^\epsilon\|_{L^{q'}}\le\|\nabla\theta_h(\rho\uu\otimes\uu)\|_{L^{q'}}
	\le\|\nabla\theta_h\|_{L^b}\|\rho\uu\otimes\uu\|_{L^s}$ since $\frac1b+\frac1s=1-\frac1q=\frac1{q'}$.
	We also used $|\nabla\theta_h|\le C/h$ (hence $\|\nabla\theta_h\|_{L^b}\le Ch^{-1+\frac1b}=Ch^{-\frac2q-\frac1r}$),
	$\|\rho\uu\otimes\uu\|_{L^s}\le\|\rho\uu\|_{L^r}\|\uu\|_{L^q}\le C$,
	and the boundary strip estimate
	$\|\uu\|_{L^q(\Omega_{3h})}\le Ch^\beta\|\uu\|_{B_q^{\beta,\infty}([0,T]\times\Omega)}$, which follows from
	$\uu=0$ on $\partial\Omega$. This term therefore vanishes as soon as $\beta>\frac2q+\frac1r$, without
	consuming any $\nu$-window.
	Meanwhile
	\begin{align*}
		(\rho\uu\otimes\uu)^\epsilon - ((\rho\uu)^\epsilon\otimes\uu^\epsilon)
		= ((\rho\uu)^\epsilon - \rho\uu)\otimes(\uu^\epsilon - \uu)
		+ \!\!\int_{\epsilon}^{2\epsilon}\int_\Omega 
		m_\epsilon(y)m_\epsilon(\tau)\Delta_{\epsilon}(\rho\uu)\otimes
		\Delta_{\epsilon}\uu\,dy\,d\tau.
	\end{align*}
	Similarly, for $B_2$, 
	\begin{align*}
		|B_2|
		&= \left|\int_{0}^{T}\int_\Omega
		\phi(t)(\theta_h\uu)^\epsilon\nabla\theta_h(((\rho\uu)^\epsilon\otimes\uu^\epsilon)-(\rho\uu\otimes\uu)^\epsilon)dxdt\right|\\
		&\leq \left|\int_{0}^{T}\int_\Omega\phi(t)((\rho\uu)^\epsilon-\rho\uu)\otimes(\uu^\epsilon-\uu)\nabla\theta_h (\theta_h\uu)^\epsilon
		dxdt\right|\\
		&\quad+\left|\int_{0}^{T}\int_\Omega\phi(t)\left(\int_{\epsilon}^{2\epsilon}\int_\Omega 
		m_\epsilon(y)m_\epsilon(\tau)\Delta_{\epsilon}(\rho\uu)\otimes
		\Delta_{\epsilon}\uu\,dy\,d\tau\right)\nabla\theta_h (\theta_h\uu)^\epsilon dxdt\right|\\
		&\leq
		C\|\phi\|_{C^0}\|\rho\uu\|_{B_r^{\tau,\infty}([0,T]\times\Omega)}\|\uu\|_{B_q^{\beta,\infty}([0,T]\times\Omega)}\|\uu\|_{L^q(0,T;L^q(\Omega))}\epsilon^{\tau+\beta}h^{-\frac2q-\frac1r}\to0,
	\end{align*}
	as $(\epsilon,h)\to0$, if $\nu<\frac{\tau+\beta}{\frac2q+\frac1r}$ (note that $\frac2q+\frac1r<\frac{2}{n+2}\le\frac12<1$,
	since the hypothesis $\big(1-\frac1p\big)\big(\frac{rq}{r+q}-n\big)>2$ implies $\frac1q+\frac1r<\frac1{n+2}$).
	Here we used the H\"older inequality with $\nabla\theta_h$ in $L^{\frac{qr}{qr-2r-q}}(\Omega)$ on the boundary strip. Since $\nabla\uu$ is bounded in $L^2(0,T;L^2(\Omega))$, we have 
	\begin{align*}
		|B_{3}|
		&=\left|\int_{0}^{T}\int_\Omega \phi(t)\nabla(\theta_h\uu)^\epsilon
		((\rho\uu)^\epsilon\otimes(\theta_h\uu)^\epsilon-(\rho\uu\otimes\theta_h\uu)^\epsilon)dxdt\right|\\
		&\leq C\|\phi\|_{C^0}
		\big\|(\rho\uu)^\epsilon\otimes(\theta_h\uu)^\epsilon-(\rho\uu\otimes\theta_h\uu)^\epsilon\big\|_{L^2(0,T;L^2(\Omega))}
		\big\|\nabla(\theta_h\uu)^\epsilon\big\|_{L^2(0,T;L^2(\Omega))}\\
		&\leq
		C\|\phi\|_{C^0}\|\rho\uu\|_{B_r^{\tau,\infty}([0,T]\times\Omega)}\|\uu\|_{B_q^{\beta,\infty}([0,T]\times\Omega)}\,
		\epsilon^{\tau+\beta}\big(1+h^{\beta-\frac12-\frac1q}\big)\to0,
	\end{align*}
	as $(\epsilon,h)\to0$ if $\nu<\frac{\tau+\beta}{\frac12+\frac1q-\beta}$, where we used
	$L^s(\Omega)\hookrightarrow L^2(\Omega)$ with $\frac1s=\frac1r+\frac1q<\frac1n\leq\frac12$ (a consequence of the
	hypotheses of Theorem \ref{CNSV thm}) and the bound
	$\|\nabla(\theta_h\uu)^\epsilon\|_{L^2(0,T;L^2(\Omega))}\leq C\big(1+h^{\beta-\frac12-\frac1q}\big)$
	established in the estimate of $B_{21}$ in the incompressible case. For the last term in \eqref{8.5}, we estimate as follows
	\begin{align*}
		&\left|-\int_{0}^{T}\!\!\int_\Omega \phi(t)(\theta_h\uu)^\epsilon\nabla\theta_h((\rho\uu)^\epsilon\otimes\uu^\epsilon)dxdt\right|\\
		&\quad\leq\left|\int_{0}^{T}\!\!\int_{\Omega}
		\phi(t)(\theta_h\uu)^\epsilon\nabla\theta_h((\rho\uu)^\epsilon\otimes(\uu^\epsilon(x)-\uu^\epsilon(\sigma(x))))dxdt\right|\\
		&\qquad+\left|\int_{0}^{T}\!\!\int_{\Omega}
		\phi(t)(\theta_h\uu)^\epsilon\nabla\theta_h((\rho\uu)^\epsilon\otimes\uu^\epsilon(\sigma(x)))dxdt\right|\\
		&\quad\leq
		C\|\phi\|_{C^0}\|\uu\|_{L^q(0,T;L^q(\Omega))}\|\rho\uu\|_{L^r(0,T;L^r(\Omega))}\|\uu\|_{B_q^{\beta,\infty}([0,T]\times\Omega)}
		\Big(\epsilon^{\beta-1}h^{1-\frac{2}{q}-\frac{1}{r}}+\epsilon^{\beta-\frac1q}h^{-\frac1q-\frac1r}\Big)\to0,
	\end{align*}
	
	where we used
	\[
	\|\uu^\epsilon-\uu^\epsilon(\sigma(x))\|_{L^q(0,T;L^q(\Omega))}
	\le C(q,\Omega)\,h\,\|\nabla\uu^\epsilon\|_{L^q(0,T;L^q(\Omega))}
	\]
	and the scaled trace estimate
	$\|\uu^\epsilon(\sigma(\cdot))\|_{L^q(0,T;L^q(\Omega_h))}\leq
	Ch^{\frac1q}\epsilon^{\beta-\frac1q}\|\uu\|_{B_q^{\beta,\infty}([0,T]\times\Omega)}$ as in the incompressible case.
	Choosing $h=\epsilon^\nu$ with
	\[
	\frac{qr(1-\beta)}{qr-2r-q}<\nu<\frac{r(q\beta-1)}{q+r},
	\]
	(this interval is nonempty exactly when $\beta>\frac2q+\frac1r$),
	to guarantee the vanishing as $(\epsilon,h)\to0$ we need 
	\begin{align*}
		1> \beta>\frac{2}{q}+\frac{1}{r}.
	\end{align*}
	
	
	
	Passing to the limit  in \eqref{8.5}, we obtain
	\begin{eqnarray}\label{eq-step12-limit}
		\lim_{\epsilon,h\to0}\int_{0}^{T}\int_{\Omega} \phi(t)(\theta_h \uu)^\epsilon \Bigg[\theta_h(\rho\uu)_t +
		\theta_h\nabla_x\cdot(\rho\uu\otimes\uu)\Bigg]^\epsilon dxdt=\int_{0}^{T}\int_\Omega \phi(t)\left(\tfrac{1}{2}\rho|\uu|^2\right)_t \,dxdt.
	\end{eqnarray}

	\textbf{Step 3.} Next, the fourth term in \eqref{CNSV proof Main} gives
	\begin{align}\label{8.6}
		&\mu\int_{0}^{T}\int_\Omega \phi(t)(\theta_h \uu)^\epsilon (\theta_h\Delta_x \uu)^\epsilon \,dxdt\notag\\
		&\quad= -\mu\int_{0}^{T}\int_\Omega \phi(t)\nabla(\theta_h\uu)^\epsilon:(\theta_h\nabla\uu)^\epsilon \,dxdt
		-\mu\int_{0}^{T}\int_\Omega \phi(t)(\theta_h\uu)^\epsilon(\nabla\theta_h\nabla\uu)^\epsilon \,dxdt.
	\end{align}
	We note that the structure of \eqref{8.6} is the same as in the proof of \eqref{t3}, so we use the same argument to obtain the result
	\begin{eqnarray}
		\lim_{\epsilon,h\to0}-\mu\int_{0}^{T}\int_\Omega\phi(t)(\theta_h\uu)^\epsilon\cdot(\theta_h\Delta_x\uu)^\epsilon
		dxdt=\mu\int_{0}^{T}\int_{\Omega}\phi(t)|\nabla\uu|^2dxdt.
	\end{eqnarray}

	The estimate for the fifth term in \eqref{CNSV proof Main} is analogous to that for the fourth term, and thus we obtain
	\begin{equation}\label{eq-lambda-limit}
		\lim_{\epsilon,h\to0}-\lambda\int_{0}^{T}\int_\Omega \phi(t)(\theta_h \uu)^\epsilon (\theta_h\nabla_x\mathrm{div}\,\uu)^\epsilon \,dxdt
		=\lambda\int_{0}^{T}\int_\Omega \phi(t)\,|\mathrm{div}\,\uu|^2 \,dxdt.
	\end{equation}

	\textbf{Step 4.}We estimate the third term in \eqref{CNSV proof Main}, i.e. the pressure term as follows
	
	\begin{align}\label{pressure1}
		\int_{0}^{T}\int_\Omega\phi(t)(\theta_h\uu)^\epsilon (\theta_h\nabla P(\rho))^\epsilon dxdt
		&=\int_{0}^{T}\int_\Omega \phi(t)(\theta_h\uu)^\epsilon((\theta_h\nabla P(\rho))^\epsilon-\theta_h(\nabla P(\rho))^\epsilon) dxdt\notag\\
		&\quad+\int_{0}^{T}\int_\Omega\phi(t)(\theta_h\uu)^\epsilon\theta_h\nabla(P(\rho)^\epsilon-P(\rho^\epsilon)) dxdt\notag\\
		&\quad+ \int_{0}^{T}\int_\Omega\phi(t)\theta_h\nabla P(\rho^\epsilon)((\theta_h\uu)^\epsilon-\theta_h\uu^\epsilon) dxdt\notag\\
		&\quad+\int_{0}^{T}\int_\Omega \phi(t)\theta_h\uu^\epsilon\theta_h\nabla P(\rho^\epsilon) dxdt\notag\\
		&=D_1+D_2+D_3+\gamma\int_{0}^{T}\int_\Omega\phi(t)\theta_h^2\uu^\epsilon(\rho^\epsilon)^{\gamma-1}\nabla\rho^\epsilon dxdt.
	\end{align}
	We now treat the last term of \eqref{pressure1}.
	Integrating by parts in the space variable gives
	\begin{align}\label{pressure2}
		&\gamma\int_{0}^{T}\int_\Omega\phi(t)\theta_h^2\uu^\epsilon(\rho^\epsilon)^{\gamma-1}\nabla\rho^\epsilon dxdt
		=\int_{0}^{T}\int_\Omega\phi(t)\theta_h^2\uu^\epsilon\cdot\nabla P(\rho^\epsilon)\,dxdt\notag\\
		&\quad=-\int_{0}^{T}\int_\Omega\phi(t)\theta_h^2P(\rho^\epsilon)\,\mathrm{div}\,\uu^\epsilon\,dxdt
		-2\int_{0}^{T}\int_\Omega\phi(t)\theta_h\nabla\theta_h\cdot\uu^\epsilon P(\rho^\epsilon)\,dxdt.
	\end{align}
	For the first term on the right-hand side of \eqref{pressure2}, adding and subtracting the mollified term
	$(P(\rho)\mathrm{div}\,\uu)^\epsilon$ gives
	\begin{align}\label{pressure3}
		-\int_{0}^{T}\int_\Omega\phi(t)\theta_h^2P(\rho^\epsilon)\,\mathrm{div}\,\uu^\epsilon\,dxdt
		=D_4-\int_{0}^{T}\int_\Omega\phi(t)\theta_h^2(P(\rho)\mathrm{div}\,\uu)^\epsilon\,dxdt,
	\end{align}
	where
	\begin{align*}
		D_4:=-\int_{0}^{T}\int_\Omega\phi(t)\theta_h^2
		\Big[P(\rho^\epsilon)\,\mathrm{div}\,\uu^\epsilon-(P(\rho)\mathrm{div}\,\uu)^\epsilon\Big]dxdt.
	\end{align*}
	Since $\uu\in L^2(0,T;H_0^1(\Omega))$ and $0<\underline{\rho}\leq\rho\leq\overline{\rho}$,
	the DiPerna--Lions renormalization theory \cite{diperna1989ons} (cf. also Step 6a below),
	applied to the mass equation $\eqref{CNSVsystem}_1$, shows that $H(\rho):=\frac{\rho^\gamma}{\gamma-1}$ satisfies
	the renormalized continuity equation
	\begin{align}\label{eq-renorm-mass}
		\partial_t H(\rho)+\mathrm{div}(H(\rho)\uu)+P(\rho)\,\mathrm{div}\,\uu=0
	\end{align}
	in the sense of distributions, because $\rho H'(\rho)-H(\rho)=P(\rho)$.
	Mollifying both sides of \eqref{eq-renorm-mass} and testing against $\phi(t)\theta_h^2$, we obtain
	\begin{align}\label{pressure4}
		-\int_{0}^{T}\int_\Omega\phi(t)\theta_h^2(P(\rho)\mathrm{div}\,\uu)^\epsilon\,dxdt\notag
		&=\int_{0}^{T}\int_\Omega\phi(t)\theta_h^2\partial_t H(\rho)^\epsilon\,dxdt\notag\\
		&\quad+\int_{0}^{T}\int_\Omega\phi(t)\theta_h^2\,\mathrm{div}(H(\rho)\uu)^\epsilon\,dxdt\notag\\
		&=-\int_{0}^{T}\phi'(t)\int_\Omega\theta_h^2 H(\rho)^\epsilon\,dxdt\notag\\
		&\quad-2\int_{0}^{T}\int_\Omega\phi(t)\theta_h\nabla\theta_h\cdot(H(\rho)\uu)^\epsilon\,dxdt.
	\end{align}
	Thanks to \eqref{pressure1}--\eqref{pressure4}, we obtain
	\begin{align}\label{pressure5}
		\int_{0}^{T}\int_\Omega\phi(t)(\theta_h\uu)^\epsilon (\theta_h\nabla P(\rho))^\epsilon dxdt
		&=D_1+D_2+D_3+D_4
		-\int_{0}^{T}\phi'(t)\int_\Omega\theta_h^2 H(\rho)^\epsilon\,dxdt\notag\\
		&\quad-2\int_{0}^{T}\int_\Omega\phi(t)\theta_h\nabla\theta_h\cdot(H(\rho)\uu)^\epsilon\,dxdt\notag\\
		&\quad-2\int_{0}^{T}\int_\Omega\phi(t)\theta_h\nabla\theta_h\cdot\uu^\epsilon P(\rho^\epsilon)\,dxdt.
	\end{align}
	For term $D_1$, we have 
	\begin{align*}
		|D_1|
		&=\left|\int_{0}^{T}\int_\Omega \phi(t)(\theta_h\uu)^\epsilon((\theta_h\nabla P(\rho))^\epsilon-\theta_h(\nabla P(\rho))^\epsilon)
		dxdt\right|\\
		&\leq\left|\int_{0}^{T}\int_\Omega \phi(t)(\theta_h\uu)^\epsilon((\nabla\theta_h P)^\epsilon-\nabla\theta_h P^\epsilon) dxdt\right|\\
		&\quad+\left|\int_{0}^{T}\int_\Omega\phi(t)\nabla(\theta_h\uu)^\epsilon((\theta_h P)^\epsilon-\theta_h P^\epsilon) dxdt\right|\\
		&\leq C\|\phi\|_{C^0}\|\uu\|_{L^q(0,T;L^q(\Omega))}\|P\|_{C^0}\frac{\epsilon}{h}+
		C\|\phi\|_{C^0}\|\nabla\uu\|_{L^2(0,T;L^2(\Omega))}\|P\|_{C^0}\epsilon h^\frac{1}{2}\to0,
	\end{align*}
	as $(\epsilon,h)\to0$.
	For term $D_2$, since $P\in C[\underline{\rho},\overline{\rho}]$,
	$\uu\in B_q^{\beta,\infty}([0,T]\times\Omega)$ with $q>2$ and $\nabla\uu$ is bounded in $L^2(0,T;L^2(\Omega))$, 
	and thus 
	\begin{align*}
		\lim_{\epsilon,h\to0}\int_{0}^{T}\int_\Omega\phi(t)(\theta_h\uu)^\epsilon\theta_h\nabla(P(\rho)^\epsilon-P(\rho^\epsilon)) dxdt=0.
	\end{align*}
	Using integrating by parts, we give a control on $D_3$
	\begin{align*}
		|D_3|
		&=\left|\int_{0}^{T}\int_\Omega \phi(t)\theta_h\nabla P(\rho^\epsilon)((\theta_h\uu)^\epsilon-\theta_h\uu^\epsilon)dxdt\right|\\
		&\leq\left|\int_{0}^{T}\int_\Omega \phi(t)\theta_h P\mathrm{div}((\theta_h\uu)^\epsilon-\theta_h\uu^\epsilon)dxdt\right|\\
		&\quad+\left|\int_{0}^{T}\int_\Omega\phi(t)\nabla\theta_h P((\theta_h\uu)^\epsilon-\theta_h\uu^\epsilon)dxdt\right|\\
		&\leq C\|\phi\|_{C^0}\|P\|_{C^0}\|\nabla\uu\|_{L^2(0,T;L^2(\Omega))}\,\epsilon h^{-\frac12}+
		C\|\phi\|_{C^0}\|P\|_{C^0}\|\uu\|_{L^q(0,T;L^q(\Omega))}\,\epsilon h^{-1-\frac1q}\to0,
	\end{align*}
	as $(\epsilon,h)\to0$ provided $\nu<\frac{q}{q+1}$. Here we used that the commutator
	$(\theta_h\uu)^\epsilon-\theta_h\uu^\epsilon$ is supported in a strip of width $3h$ near $\partial\Omega$ and satisfies
	$\|(\theta_h\uu)^\epsilon-\theta_h\uu^\epsilon\|_{L^q(0,T;L^q(\Omega))}\leq
	C\frac{\epsilon}{h}\|\uu\|_{L^q(0,T;L^q(\Omega))}$, since
	$|\theta_h(x-y)-\theta_h(x)|\leq C\frac{|y|}{h}$.

	For $D_4$, since $P\in C^1[\underline{\rho},\overline{\rho}]$ and $\rho\in B_r^{\tau,\infty}([0,T]\times\Omega)$,
	the composition $P(\rho)$ belongs to $B_r^{\tau,\infty}([0,T]\times\Omega)$ as well.
	Splitting the integrand into
	\begin{align*}
		P(\rho^\epsilon)\,\mathrm{div}\,\uu^\epsilon-(P(\rho)\mathrm{div}\,\uu)^\epsilon
		&=\big[P(\rho^\epsilon)-(P(\rho))^\epsilon\big]\,\mathrm{div}\,\uu^\epsilon\\
		&\quad+\big[(P(\rho))^\epsilon(\mathrm{div}\,\uu)^\epsilon-(P(\rho)\mathrm{div}\,\uu)^\epsilon\big],
	\end{align*}
	we deduce from Lemma \ref{le-communicate}, the Lipschitz continuity of $P$ on $[\underline{\rho},\overline{\rho}]$
	and the mollification error estimate that
	\begin{align*}
		|D_4|
		&\leq C\|\phi\|_{C^0}
		\big\|P(\rho^\epsilon)-(P(\rho))^\epsilon\big\|_{L^r((0,T)\times\Omega)}
		\|\mathrm{div}\,\uu\|_{L^2(0,T;L^2(\Omega))}\\
		&\quad+C\|\phi\|_{C^0}
		\big\|(P(\rho))^\epsilon(\mathrm{div}\,\uu)^\epsilon-(P(\rho)\mathrm{div}\,\uu)^\epsilon\big\|_{L^1((0,T)\times\Omega)}\\
		&\leq C\|\phi\|_{C^0}\|P(\rho)\|_{B_r^{\tau,\infty}([0,T]\times\Omega)}
		\|\nabla\uu\|_{L^2(0,T;L^2(\Omega))}\,\epsilon^\tau\to0,
		\quad (\epsilon,h)\to0,
	\end{align*}
	where the H\"older inequality is used with $\frac1r+\frac12\leq1$, which follows from $r>n\geq2$;
	note that $\frac{rq}{r+q}<r$, so the condition $\frac{rq}{r+q}>n$ indeed implies $r>n$.
	For the penultimate term in \eqref{pressure5}, since $H(\rho)$ is bounded in $L^\infty((0,T)\times\Omega)$,
	$\uu=0$ on $\partial\Omega$ and the convolution enlarges the support by at most $\epsilon\leq h$,
	the boundary-layer estimate gives
	\begin{align*}
		\left|\int_{0}^{T}\int_\Omega\phi(t)\theta_h\nabla\theta_h\cdot(H(\rho)\uu)^\epsilon\,dxdt\right|
		&\leq \frac{C}{h}\|\phi\|_{C^0}\|H(\rho)\|_{L^\infty((0,T)\times\Omega)}
		\|\uu\|_{L^1(0,T;L^1(\Omega_{3h}))}\\
		&\leq \frac{C}{h}\|\phi\|_{C^0}\|H(\rho)\|_{L^\infty((0,T)\times\Omega)}
		\|\uu\|_{L^q(0,T;L^q(\Omega_{3h}))}\,h^{1-\frac1q}\\
		&\leq C\|\phi\|_{C^0}\|H(\rho)\|_{L^\infty((0,T)\times\Omega)}
		\|\uu\|_{B_q^{\beta,\infty}([0,T]\times\Omega)}\,h^{\beta-\frac1q}\to0,
	\end{align*}
	provided $\beta>\frac1q$, where we used
	$\|\uu\|_{L^q(\Omega_{3h})}=\|\uu-\uu(\sigma(\cdot))\|_{L^q(\Omega_{3h})}
	\leq Ch^\beta\|\uu\|_{B_q^{\beta,\infty}([0,T]\times\Omega)}$,
	which follows from $\uu=0$ on $\partial\Omega$ and $|x-\sigma(x)|\leq 3h$ on $\Omega_{3h}$,
	$\sigma(x)$ being the nearest-point projection onto $\partial\Omega$.
	For the last term of \eqref{pressure5}, the same boundary-layer argument yields
	\begin{align*}
		\left|\int_{0}^{T}\int_\Omega\phi(t)\theta_h\nabla\theta_h\cdot\uu^\epsilon P(\rho^\epsilon)\,dxdt\right|
		&\leq \frac{C}{h}\|\phi\|_{C^0}\|P\|_{C^0}\|\uu\|_{L^1(0,T;L^1(\Omega_{3h}))}\\
		&\leq C\|\phi\|_{C^0}\|P\|_{C^0}\|\uu\|_{B_q^{\beta,\infty}([0,T]\times\Omega)}\,h^{\beta-\frac1q}\to0,
		\quad h\to0,
	\end{align*}
	provided $\beta>\frac{1}{q}$.
	Thanks to the continuity of $P(\rho)$ again, passing to the limit $(\epsilon,h)\to0$ in \eqref{pressure5},
	consequently, we obtain
	\begin{align}
		\lim_{\epsilon,h\to0}\int_{0}^{T}\int_\Omega\phi(t)(\theta_h\uu)^\epsilon (\theta_h\nabla P(\rho))^\epsilon dxdt
		=-\int_{0}^{T}\phi'(t)\int_\Omega H(\rho)\,dxdt
		=\int_{0}^{T}\int_\Omega\phi(t)\frac{d}{dt}\frac{\rho^\gamma}{\gamma-1}\,dxdt.
	\end{align}

	\textbf{Step 5.} Finally, we estimate the last term in \eqref{CNSV proof Main}
	\begin{align}
		\int_{0}^{T}\int_\Omega 
		\phi(t)(\theta_h\uu)^\epsilon 
		\left(\theta_h\int_{\RR}\rho(\uu-\vv)f\,d\vv\right)^\epsilon 
		dxdt
		&= \int_{0}^{T}\int_\Omega 
		\phi(t)(\theta_h\uu)^\epsilon 
		\left(\theta_h\int_{\RR}\rho\uu f\,d\vv\right)^\epsilon 
		dxdt  \notag \\
		&\quad - \int_{0}^{T}\int_\Omega 
		\phi(t)(\theta_h\uu)^\epsilon 
		\left(\theta_h\int_{\RR}\rho\vv f\,d\vv\right)^\epsilon 
		dxdt  \notag \\
		&= \int_{0}^{T}\int_\Omega 
		\phi(t)(\theta_h\uu)^\epsilon 
		(\theta_h\rho \uu \bar{\rho})^\epsilon 
		dxdt \notag\\
		&\quad- \int_{0}^{T}\int_\Omega
		\phi(t)(\theta_h\uu)^\epsilon (\theta_h\rho j)^\epsilon 
		dxdt , \notag\\
		&=F_1+F_2.
	\end{align}
	
	We decompose $F_1$ as 
	\begin{align*}
		\int_{0}^{T}\int_\Omega 
		\phi(t)(\theta_h\uu)^\epsilon 
		(\theta_h\rho \uu \bar{\rho})^\epsilon 
		dxdt
		&=\int_{0}^{T}\int_\Omega 
		\phi(t)(\theta_h\uu)^\epsilon 
		[(\theta_h\uu\rho \bar{\rho})^\epsilon -(\theta_h\uu\rho)^\epsilon\bar{\rho}^\epsilon]
		dxdt\\
		&\quad+\int_{0}^{T}\int_\Omega 
		\phi(t)(\theta_h\uu)^\epsilon 
		\Big[(\theta_h\uu\rho)^\epsilon-(\theta_h\uu)^\epsilon\rho^\epsilon\Big]\bar{\rho}^\epsilon dxdt\\
		&\quad+\int_{0}^{T}\int_\Omega\phi(t)(\theta_h\uu)^\epsilon\cdot((\theta_h\uu)^\epsilon\rho^\epsilon\bar{\rho}^\epsilon)dxdt\\
		&=F_{11}+F_{12} +  \int_{0}^{T}\int_\Omega\phi(t)((\theta_h\uu)^\epsilon)^2\cdot(\rho^\epsilon\bar{\rho}^\epsilon)dxdt. 
	\end{align*}
	For term $F_{11}$ and term $F_{12}$,
	using $\uu\in B^{\beta,\infty}_q([0,T]\times\Omega)$ and $\rho,\rho\uu\in B^{\tau,\infty}_r([0,T]\times\Omega)$, we have 
	\begin{align*}
		|F_{11}|&\leq
		C\|\phi\|_{C^0}\|\uu\|_{L^q(0,T;L^q(\Omega))}\|\rho\uu\|_{B_r^{\tau,\infty}([0,T]\times\Omega)}\|\bar{\rho}\|_{L^{p_1}(\Omega)}\epsilon^\tau,\\
		|F_{12}|&\leq
		C\|\phi\|_{C^0}\|\uu\|_{L^q(0,T;L^q(\Omega))}\|\rho\|_{B_r^{\tau,\infty}([0,T]\times\Omega)}\|\uu\|_{B_q^{\beta,\infty}([0,T]\times\Omega)}\|\bar{\rho}\|_{L^{p_1}(\Omega)}\epsilon^{\tau+\beta},
	\end{align*}
	provided 
	\begin{align*}
		\frac{1}{p_1}+\frac{1}{q}+\frac{1}{r}\le1\quad&\Longleftrightarrow\quad (1-\frac{1}{p})(\frac{2rq}{r+q}-n)>2,\\
		\frac{1}{p_1}+\frac{2}{q}+\frac{1}{r}\le1\quad&\Longleftrightarrow\quad
		(1-\frac{1}{p})(\frac{2rq}{2r+q}-n)>2.
	\end{align*}
	Thus, we obtain 
	\begin{align}\label{new bar rho}
		\lim_{\epsilon,h\to0}\int_{0}^{T}\int_\Omega 
		\phi(t)(\theta_h\uu)^\epsilon 
		(\theta_h\rho \uu \bar{\rho})^\epsilon 
		dxdt=\int_{0}^{T}\int_\Omega 
		\phi(t)\, \uu
		(\rho \uu \bar{\rho})
		dxdt.
	\end{align}
	
	Similarly, we decompose $F_2$ as 
	\begin{align*}
		\int_{0}^{T}\int_\Omega \phi(t)(\theta_h\uu)^\epsilon (\theta_h\rho j)^\epsilon dxdt
		&=\int_{0}^{T}\int_\Omega \phi(t)(\theta_h\uu)^\epsilon [(\theta_h\rho j)^\epsilon-(\theta_h\rho)^\epsilon j^\epsilon] dxdt\\
		&\quad+\int_{0}^{T}\int_\Omega \phi(t)(\theta_h\uu)^\epsilon (\theta_h\rho)^\epsilon j^\epsilon dxdt\\
		&=F_{21}+\int_{0}^{T}\int_\Omega \phi(t)(\theta_h\uu)^\epsilon (\theta_h\rho)^\epsilon j^\epsilon dxdt.
	\end{align*}
	
	We employ the same argument as before, we have 
	\begin{align*}
		|F_{21}|\leq
		C\|\phi\|_{C^0}\|\uu\|_{L^q(0,T;L^q(\Omega))}\|\rho\|_{B_r^{\tau,\infty}([0,T]\times\Omega)}\|j\|_{L^{p_2}(\Omega)}\epsilon^\tau.
	\end{align*}
	provided 
	\begin{align*}
		\frac{1}{p_2}+\frac{1}{q}+\frac{1}{r}\le1\quad&\Longleftrightarrow\quad\Big(1-\frac{1}{p}\Big)(\frac{rq}{r+q}-n)> 2.
	\end{align*}
	We obtain 
	\begin{align}\label{new j}
		\lim_{\epsilon,h\to0}\int_{0}^{T}\int_\Omega
		\phi(t)(\theta_h\uu)^\epsilon (\theta_h\rho j)^\epsilon 
		dxdt=\int_{0}^{T}\int_\Omega
		\phi(t)\, \uu (\rho j)
		dxdt.
	\end{align}
	

	Thanks to \eqref{A}-\eqref{new j}, by letting the parameters $\epsilon$ and $h$ tend to zero,
	we can pass to the limit in \eqref{CNSV proof Main}. Consequently, we obtain the following integral identity:
	\begin{eqnarray}\label{8.16}
		\int_{0}^{T}\int_\Omega \phi(t)\left(\frac{1}{2}\rho|\uu|^2+\frac{\rho^\gamma}{\gamma-1}\right)_t dxdt 
		+ \int_{0}^{T}\int_\Omega \phi(t)\left(\mu|\nabla_x\uu|^2+\lambda |\mathrm{div}\uu|^2\right)dxdt \notag\\
		+ \int_{0}^{T}\int_{\Omega}\phi(t)\uu \cdot \left(\int_{\RR}\rho(\uu-\vv)f\,d\vv\right) dxdt = 0,
	\end{eqnarray}
	which holds for any test function $\phi\in\mathcal{D}(0,+\infty)$.

The Vlasov equation $\eqref{CNSVsystem}_3$ has the same structure of $\eqref{NSVsystem}_3$, thus we take the same test function 
\begin{align*}
    \Phi_{h,R}(x,\vv,t)=\phi(t)\theta_h(x)\Big(1+\frac{|\vv|^2}{2}\Big)\chi_R(\vv).
\end{align*}

We plug $\Phi_{h,R}$ into the weak formulation of the Vlasov equation:
\begin{align*}
\int_0^T\int_{\Omega}\int_{\RR}
    f\big(\partial_t\Phi_{h,R} + \vv\cdot\nabla_x\Phi_{h,R}
    + \rho( \uu - \vv )\cdot\nabla_\vv \Phi_{h,R} \big)\,d\vv\,dxdt = 0 .
\end{align*}

Therefore
\begin{align*}
&\int_0^T\phi'(t)
    \left(\int_{\Omega}\int_{\RR} 
        f\theta_h(x)\Big(1+\frac{|\vv|^2}{2}\Big)\chi_R(\vv)\,d\vv\,dx\right) dt \\
&\qquad+
\int_{0}^{T}\phi(t)\left(\int_\Omega\int_{\RR}\vv f \nabla_x\theta_h(x)\Big(1+\frac{|\vv|^2}{2}\Big)\chi_R(\vv)d\vv dx\right)dt\\
&\qquad+
\int_0^T\phi(t)
    \left(\int_{\Omega}\int_{\RR}
        \rho(\uu-\vv)\cdot
        \Big(1+\frac{|\vv|^2}{2}\Big)\nabla_\vv\chi_R(\vv)\, \theta_h(x)f
        \,d\vv\,dx\right) dt \\
&\qquad+
\int_0^T\phi(t)
    \left(\int_{\Omega}\int_{\RR}
        \rho(\uu-\vv)\cdot\vv\, \theta_h(x)f \chi_R(\vv)\,d\vv\,dx\right) dt
:=\sum_{i=1}^{4} G_i=0 .
\end{align*}

It is easy to check that 
\begin{align}\label{CNSV G_1-G_2}
	\lim_{R\to\infty}\lim_{h\to0}G_1
	&=\int_0^T\phi'(t)
    \left(\int_{\Omega}\int_{\RR} 
        f\Big(1+\frac{|\vv|^2}{2}\Big)\,d\vv\,dx\right) dt,\\
	\lim_{R\to\infty}\lim_{h\to0}G_2&=0.
\end{align}

The treatment of $G_2$ here is completely analogous to that in Section \ref{se-pro-th1} and likewise relies on the trace condition $\alpha
p>1$
assumed in Theorem \ref{CNSV thm}:
by Lemmas \ref{le-MR}--\ref{le-layer}, the truncated macroscopic momentum flux admits a well-defined trace on $\partial\Omega$,
which is tangential to the boundary by the specular reflection condition \eqref{eq-bdy-comp},
so that the boundary integral vanishes in the limit.

For term $G_3$, we split it as follows
\begin{align*}
	G_3
	&=\int_0^T\phi(t)
    \left(\int_{\Omega}\int_{\RR}
        \rho(\uu-\vv)\cdot
        \Big(1+\frac{|\vv|^2}{2}\Big)\nabla_\vv\chi_R(\vv)\, \theta_h(x)f
        \,d\vv\,dx\right) dt\\
	&=\int_{0}^{T}\int_\Omega\int_{\RR} \phi(t)\theta_h f \rho\uu\left(1+\frac{|\vv|^2}{2}\right)\nabla_\vv\chi_R(\vv) dxd\vv dt\\
	&\quad-\int_{0}^{T}\int_\Omega\int_{\RR} \phi(t)\theta_h f \rho\vv\left(1+\frac{|\vv|^2}{2}\right)\nabla_\vv\chi_R(\vv) dxd\vv dt\\
	&:= G_{31}+G_{32}.
\end{align*}
We employ a similar argument to that for $\eqref{G_3}$ to estimate both terms. One obtains
\begin{align}\label{CNSV G_3}
	\lim_{R\to\infty}\lim_{h\to0}G_{31}=0,\quad \lim_{R\to\infty}\lim_{h\to0} G_{32}=0,
\end{align}
if 
\begin{align*}
	\frac{1}{p_2}+\frac{1}{r}\leq1\quad&\Longleftrightarrow\quad\Big(1-\frac{1}{p}\Big)(r-n)> 2.
\end{align*}
Then for term $G_4$, we also easily obtain 
\begin{align}\label{CNSV G_4}
	\lim_{R\to\infty}\lim_{h\to0}G_4=\int_0^T\phi(t)
    \left(\int_{\Omega}\int_{\RR}
        \rho(\uu-\vv)\cdot\vv\, f \,d\vv\,dx\right) dt.
\end{align}

Combining \eqref{CNSV G_1-G_2}-\eqref{CNSV G_4}, we conclude that
	\begin{align}\label{new 4.13}
		\int_0^T \phi'(t) F(f(t))\,dt
		&+
		\int_0^T \phi(t)
		\int_{\Omega} \uu\cdot
		\Big(\int_{\RR} \rho(\uu-\vv) f\,d\vv\Big)\,dxdt
		\notag\\&\qquad=
		\int_0^T \phi(t)
		\int_{\Omega}\int_{\RR}
		\rho|\uu-\vv|^2 f \,d\vv\,dxdt .
	\end{align}
	where 
	\begin{align*}
		F(f(t))=\int_{\Omega}\int_{\RR}
		f(t,x,\vv)\Big(1+\frac{|\vv|^2}{2}\Big)\,d\vv\,dx.
	\end{align*}

	
	Combining \eqref{8.16} with \eqref{new 4.13}, and integrating by parts in time 
	in the term involving $\phi_t$, we arrive at the local energy balance for the 
	coupled system:
	\begin{align}\label{8.18}
		-\int_{0}^{T}\int_{\Omega} \phi_t(t)
		\left(
		\tfrac{1}{2}\rho|\uu|^{2}
		+\tfrac{\rho^\gamma}{\gamma-1}
		+\int_{\RR} f\Big(1+\tfrac{|\vv|^{2}}{2}\Big)d\vv
		\right) dx\,dt \notag \\
		\quad
		+\int_{0}^{T}\int_{\Omega} \phi(t)
		\left(
		\mu|\nabla\uu|^{2}
		+\lambda|\mathrm{div}\,\uu|^{2}
		+\int_{\RR}\rho f|\uu-\vv|^{2} d\vv
		\right) dx\,dt
		= 0 .
	\end{align}
	
	We note that the hypothesis $(1-\frac{1}{p})\big(\frac{rq}{r+q}-n\big)>2$ of Theorem \ref{CNSV thm} absorbs all three H\"older conditions
	used in the estimates of $F_{11}$, $F_{12}$ and $F_{21}$ above, since $\frac{rq}{r+q}<\frac{2rq}{2r+q}<\frac{2rq}{r+q}$;
	together with $\beta>\frac{1}{r}+\frac{2}{q}$, which provides the compactness required in the drag term,
	this ensures that all commutator errors vanish in the limit.

	\textbf{Step 6.}
	The final step is to extend our result \eqref{8.18} for the test function $\phi(t)\in\mathcal{D}(-1,\infty)$.
	To this end, we need the continuity of the total energy on $[0,T]$.
	
	Let us define the total energy
	\begin{align*}
		E(t)=\int_\Omega
		\Big(\tfrac{1}{2}\rho|\uu|^2+\tfrac{\rho^\gamma}{\gamma-1}
		+\!\int_{\RR}f(1+\tfrac{|\vv|^2}{2})\,d\vv\Big)dx.
	\end{align*}
	Then \eqref{8.18} can be written as
	\begin{align*}
		-\int_0^T \phi_t(t)E(t)\,dt
		+\int_0^T \phi(t)\,D(t)\,dt=0,
	\end{align*}
	where 
	\begin{align*}
		D(t)=\int_\Omega
		\Big(\mu|\nabla\uu|^2+\lambda|\mathrm{div}\uu|^2
		+\!\int_{\RR}\rho f|\uu-\vv|^2\,d\vv\Big)dx.
	\end{align*}
	
	Since $\rho$ and $\uu$ are only known a priori to be weakly continuous in time,
	the continuity of $t\mapsto E(t)$ does not follow directly; we establish it in three substeps.

	\textbf{Step 6a (continuity of the fluid energy).}
	Since $\beta>\frac1r+\frac2q>\frac1q$,
	the Morrey-type embedding in the time variable applied to $\uu\in B_q^{\beta,\infty}([0,T]\times\Omega)$ gives $\uu\in
	C([0,T];L^q(\Omega))$, and hence $\uu\in C([0,T];L^2(\Omega))$ because $q>2$.
	Since $\rho\leq\overline{\rho}$ and $\uu\in L^\infty(0,T;L^2(\Omega))\cap L^2(0,T;H^1(\Omega))$ (the former inclusion follows from
	$\rho\geq\underline{\rho}>0$ and the energy bound $\sqrt{\rho}\,\uu\in L^\infty(0,T;L^2(\Omega))$),
	we have $\rho\uu\in L^2(0,T;L^2(\Omega))$,
	and the continuity equation yields $\partial_t\rho=-\mathrm{div}(\rho\uu)\in L^2(0,T;H^{-1}(\Omega))$;
	together with $\rho\in L^\infty(0,T;L^\gamma(\Omega))$, this implies $\rho\in C_\mathrm{weak}([0,T];L^\gamma(\Omega))$.
	Moreover, the DiPerna--Lions renormalization theory \cite{diperna1989ons},
	applicable since $\uu\in L^1(0,T;W^{1,1}(\Omega))$ and $\rho\in L^\infty((0,T)\times\Omega)$,
	gives $\rho\in C([0,T];L^s(\Omega))$ for every $1\le s<\infty$;
	in particular, the map $t\mapsto\int_\Omega\frac{\rho^\gamma}{\gamma-1}\,dx$ is continuous on $[0,T]$.
	For the kinetic part of the fluid energy,
	the momentum equation yields $\partial_t(\rho\uu)\in L^a(0,T;H^{-m}(\Omega))$ for some $a>1$ and $m\ge1$,
	which combined with $\rho\uu\in L^\infty(0,T;L^2(\Omega))$ gives $\rho\uu\in C_\mathrm{weak}([0,T];L^2(\Omega))$. Writing
	\[
	\int_\Omega \rho|\uu|^2(t)\,dx-\int_\Omega \rho|\uu|^2(t_0)\,dx
	=\int_\Omega \rho\uu(t)\cdot(\uu(t)-\uu(t_0))\,dx
	+\int_\Omega \big(\rho\uu(t)-\rho\uu(t_0)\big)\cdot\uu(t_0)\,dx,
	\]
	the first term on the right-hand side tends to $0$ as $t\to t_0$,
	since $\uu\in C([0,T];L^2(\Omega))$ and $\rho\uu$ is bounded in $L^\infty(0,T;L^2(\Omega))$,
	while the second term tends to $0$ by the weak continuity of $\rho\uu$;
	hence the map $t\mapsto\int_\Omega\frac12\rho|\uu|^2\,dx$ is continuous on $[0,T]$.

	\textbf{Step 6b (continuity of the particle energy).}
	Since $\alpha p>1$,
	the Morrey-type embedding in the time variable applied to $f\in B_p^{\alpha,\infty}([0,T]\times\Omega\times\RR)$ gives $f\in
	C([0,T];L^p(\Omega\times\RR))$.
	Hence, for every $t_0\in[0,T]$ and every $R>0$, the truncated particle energy
	$\int_\Omega\int_{\{|\vv|\le R\}}f(t)\big(1+\frac{|\vv|^2}{2}\big)\,d\vv dx$ is continuous at $t_0$;
	together with the uniform bound
	$f\big(1+\frac{|\vv|^2}{2}\big)\in L^\infty(0,T;L^1(\Omega\times\RR))$ from Definition \ref{def weak},
	this yields the lower semicontinuity
	\[
	\int_\Omega\int_{\RR}f(t_0)\Big(1+\frac{|\vv|^2}{2}\Big)d\vv dx
	\le\liminf_{t\to t_0}\int_\Omega\int_{\RR}f(t)\Big(1+\frac{|\vv|^2}{2}\Big)d\vv dx,
	\qquad t_0\in[0,T].
	\]
	On the other hand, the strong energy inequality in Definition \ref{def weak} gives
	$\limsup_{t\to t_0^+}E(t)\le E(t_0)$ for almost every $t_0\in[0,T)$;
	since the fluid part of $E$ is continuous by Step 6a, the same bound holds for the particle part:
	\[
	\limsup_{t\to t_0^+}\int_\Omega\int_{\RR}f(t)\Big(1+\frac{|\vv|^2}{2}\Big)d\vv dx
	\le\int_\Omega\int_{\RR}f(t_0)\Big(1+\frac{|\vv|^2}{2}\Big)d\vv dx,
	\qquad\text{a.e. } t_0\in[0,T).
	\]
	The two inequalities show that the particle energy is right-continuous at almost every $t_0\in[0,T)$;
	consequently, after a possible redefinition on a time null set,
	$t\mapsto\int_\Omega\int_{\RR}f\big(1+\frac{|\vv|^2}{2}\big)\,d\vv dx$ is continuous on $[0,T]$.

	\textbf{Step 6c (energy equality up to the endpoints).}
	Steps 6a and 6b show that, after a possible redefinition on a time null set,
	$E$ coincides with a continuous function on $[0,T]$, while $D\in L^1(0,T)$ by the energy bounds.
	Approximating the indicator function $\mathbf{1}_{[0,t]}$ by test functions $\phi\in\mathcal{D}(-1,\infty)$ in \eqref{8.18} and passing to
	the limit therefore gives, for this continuous representative,
	\[
	E(t)+\int_0^t D(s)\,ds=E(0)
	\qquad\text{for every } t\in[0,T],
	\]
	and subtracting this identity at $t_1$ from the one at $t_2$ yields the global energy equality
	\[
	E(t_2)+\int_{t_1}^{t_2}\!\!D(t)\,dt = E(t_1),
	\qquad 0\le t_1<t_2\le T,
	\]
	that is,
	\[
	E(t_2)+\int_{t_1}^{t_2}\!\!\int_\Omega
	\Big(\mu|\nabla\uu|^2+\lambda|\mathrm{div}\uu|^2
	+\!\int_{\RR}\rho f|\uu-\vv|^2\,d\vv\Big)dxdt
	=E(t_1).
	\]
	This yields the desired global energy equality on the whole time interval $[0,T]$.

\end{proof}

\section*{Acknowledgments}
The author is deeply grateful to Wentao Cao (Academy for Multidisciplinary Studies, Capital Normal University)
for suggesting this problem and for his constant guidance, careful reading of the manuscript, and many helpful
discussions during the preparation of this work.

\medskip
\noindent\textbf{Data availability.} Data sharing is not applicable to this article,
as no datasets were generated or analysed during the current study.

\end{document}